\documentclass[reqno,11pt]{amsart}

\usepackage{amsmath}
\usepackage{amsthm}
\usepackage{amssymb}
\usepackage{amscd}
\usepackage{xypic}
\usepackage{verbatim}

\usepackage[plainpages=false,colorlinks,hyperindex,pdfpagemode=None,bookmarksopen,linkcolor=red,citecolor=blue,urlcolor=blue]{hyperref}
\usepackage{pdflscape}
\usepackage{stmaryrd}
 
\usepackage{multirow}

\swapnumbers
\theoremstyle{plain}
\newtheorem{theorem}[subsubsection]{Theorem}

\newtheorem*{theorem*}{Theorem}
\newtheorem{proposition}[subsubsection]{Proposition}
\newtheorem*{proposition*}{Proposition}
\newtheorem{lemma}[subsubsection]{Lemma}
\newtheorem*{lemma*}{Lemma}

\newtheorem{corollary}[subsubsection]{Corollary}
\newtheorem*{corollary*}{Corollary}

\theoremstyle{definition}
\newtheorem*{definition}{Definition}

\theoremstyle{remark}
\newtheorem*{remark}{Remark}

\renewcommand{\comment}[1] {  }

\DeclareFontFamily{OT1}{rsfs}{}
\DeclareFontShape{OT1}{rsfs}{n}{it}{<-> rsfs10}{}
\DeclareMathAlphabet{\mathscr}{OT1}{rsfs}{n}{it}

\newcommand{\Ad}{\mathrm{Ad}}
\newcommand{\Res}{\mathrm{Res}}

\newcommand{\Z}{\mathbb{Z}}

\newcommand{\adele}{\mathbb{A}_k}

\newcommand{\CC}{\mathbb{C}}
\newcommand{\FF}{\mathbb{F}}

\newcommand{\PP}{\mathbb{P}}

\newcommand{\RR}{\mathbb{R}}

\newcommand{\QQ}{\mathbb{Q}}

\newcommand{\ab}{{\operatorname{ab}}}

\newcommand{\aff}{{\operatorname{aff}}}

\newcommand{\Hom}{\operatorname{Hom}}

\newcommand{\Aut}{{\operatorname{Aut}}}

\newcommand{\varchi}{\mathcal{X}}
\newcommand{\Gm}{\mathbb{G}_m}
\newcommand{\Ga}{\mathbb{G}_a}

\newcommand{\GL}{\operatorname{GL}}

\newcommand{\PGL}{\operatorname{PGL}}

\newcommand{\SL}{\operatorname{SL}}

\newcommand{\SO}{{\operatorname{SO}}}

\newcommand{\tr}{\operatorname{tr}}

\newcommand{\spec}{\operatorname{spec}}
\newcommand{\Vol}{\operatorname{Vol}}
\newcommand{\diag}{{\operatorname{diag}}}

\newcommand{\disc}{{\operatorname{disc}}}

\newcommand{\reg}{{\operatorname{reg}}}

\newcommand{\Id}{\operatorname{Id}}

\newcommand{\val}{{\operatorname{val}}}

\newcommand{\aut}{{\operatorname{aut}}}
\newcommand{\Eis}{{\operatorname{Eis}}}
\newcommand{\Ram}{{\operatorname{Ram}}}
\newcommand{\cusp}{{\operatorname{cusp}}}
\newcommand{\temp}{{\operatorname{temp}}}
\newcommand{\RTF}{{\operatorname{RTF}}}
\newcommand{\KTF}{{\operatorname{KTF}}}
\newcommand{\rest}{{\operatorname{rest}}}
\newcommand{\Kl}{{\mathcal{K}}}
\newcommand{\pinfty}{{\infty +}}
\newcommand{\AvgVol}{{\operatorname{AvgVol}}}

\begin{document}
\numberwithin{equation}{section}
\setcounter{tocdepth}{1}
\title[Beyond Endoscopy for the relative trace formula II]{Beyond Endoscopy for the relative trace formula II: global theory}
\author{Yiannis Sakellaridis}
\email{sakellar@rutgers.edu}

\address{Department of Mathematics and Computer Science, Rutgers University at Newark, 101 Warren Street, Smith Hall 216, Newark, NJ 07102, USA.}


\begin{abstract}
For the group $G=\PGL_2$ we perform a comparison between two relative trace formulas: on one hand, the relative trace formula of Jacquet for the quotient $T\backslash G/T$, where $T$ is a non-trivial torus, and on the other the Kuznetsov trace formula (involving Whittaker periods), applied to \emph{non-standard} test functions. This gives a new proof of the celebrated result of Waldspurger on toric periods, and suggests a new way of comparing trace formulas, with some analogies to Langlands' ``Beyond Endoscopy'' program. \end{abstract}

\maketitle

\tableofcontents

\section{Introduction.} 

\subsection{The result of Waldspurger} The celebrated result of Waldspurger \cite{Waldspurger-torus}, relating periods of cusp forms on $\GL_2$ over a nonsplit torus (against a character of the torus, but here we will restrict ourselves to the trivial character) with the central special value of the corresponding quadratic base change $L$-function, was reproven by Jacquet \cite{JW1} using the relative trace formula. Both of the proofs, however, rely on coincidences that are particular to this case, and do not generalize to most higher-rank cases of the Gross--Prasad conjectures \cite{II} and their generalizations \cite{GGP, SV}. In the case of Waldspurger, the coincidence is the appearance of this period in the setting of the theta correspondence; in the case of Jacquet, it is the appearance of the same $L$-value in periods over a split torus, previously studied by Hecke. None of these coincidences exist in arbitrary higher rank.

The purpose of the present article is to provide yet another proof of the result of Waldspurger, developing a new method, also based on the relative trace formula, which might admit generalization. Although one would need many more examples of this method in order to talk of any serious evidence of possible generalization, some of its features are pleasantly aligned with the formulation of the general conjectures of \cite{SV}. The present paper is based on the local results proven in \cite{SaBE1}, and as was the case there, I feel free to use the work of Jacquet (and hence the aforementioned coincidences) in order to shorten proofs of a few local statements that could also be proven ``by hand''. However, the global argument is completely independent.

The method relies on a ``non-standard'' comparison of relative trace formulas equipped with ``non-standard'' test functions and, hence, has similarities to the ``Beyond Endoscopy'' project of Langlands \cite{Langlands-BE}. Among others, our methods give an independent, trace formula-theoretic proof of the meromorphic continuation of quadratic base change $L$-functions. It should be kept in mind, though, that our comparisons, spectrally, correspond to some version of \emph{relative functoriality \cite{SV}} for the \emph{identity} map of dual groups, as opposed to the much more ambitious goal of Langlands involving arbitrary maps of $L$-groups.

\subsection{The relative trace formula and its conjectural spectrum} The relative trace formula (RTF) of Jacquet should be seen as a potential generalization of the invariant trace formula of Arthur and Selberg, as well as of the twisted trace formula. In the most general setting,\footnote{In fact, this is not quite the most general setting, since we can also introduce line bundles defined by complex adele class characters of unipotent groups -- as in the Kuznetsov trace formula used here.} one starts with two (homogeneous, quasi-affine) spherical varieties $X_1, X_2$ for a given (connected) reductive group $G$ over a global field $k$, and constructs a distribution on the adelic points of $X_1\times X_2$, invariant under the diagonal action of $G(\adele)$, which is naively (i.e., ignoring 
analytic 
issues) defined as
$$\Phi_1 \otimes \Phi_2 \mapsto \int_{[G]} \Sigma \Phi_1(g) \cdot \Sigma\Phi_2(g) dg,$$
where $\Sigma$ denotes the morphism
$$\mathcal S(X_i (\adele))\ni \Phi\mapsto \sum_{\gamma \in X_i(k)}\Phi(\gamma g) \in C^\infty([G]),$$ $[G]$ denotes the automorphic quotient $G(k)\backslash G(\adele)$, and $\mathcal S$ denotes the space of Schwartz functions (we will call them ``standard test functions'').

The presentation here is, actually, oversimplifying: instead of considering a $G^\diag(\adele)$-invariant distribution on the adelic points of $X_1\times X_2$, one should talk about distributions on the adelic points of the quotient stack $(X_1\times X_2)/G^\diag$. This is necessary, even in the simple cases that we are considering here, in order to include ``pure inner forms'' of the spaces under consideration into the picture and get a complete comparison between relative trace formulas. The appropriate notions for harmonic analysis on stacks were developed in \cite{SaStacks}; however in this paper we will only use the notion of stacks symbolically, and explicitly define the spaces of test functions that we need, without making use of that theory.

In the special case $X_1=X_2=$a group $H$, under the action of $G=H\times H$ by left and right multiplication, the relative trace formula specializes to the Arthur-Selberg trace formula, while if we twist the action of $\{1\}\times H\subset G$ on the second copy by an automorphism of $H$ we get the twisted trace formula. Notice the stack-theoretic isomorphism of $(H\times H)/G$ (diagonal action of $G$) with the quotient of $H$ by itself via conjugation.

Let us concentrate on the case $X_1=X_2=:X$. The relative trace formula admits a geometric and a spectral expansion. The conjectures proposed in \cite{SV} imply that the \emph{most tempered} part of the spectral expansion is supported on the set of automorphic representations with Arthur parameter (assuming the existence of the hypothetical Langlands group $\mathcal L_k$) of the form
$$ \mathcal L_k \times \SL_2 \xrightarrow{\varphi\times \Id} {^LG_X} \times\SL_2\to {^LG},$$
where $^LG_X$ is the ``$L$-group of $X$''. (This was defined in \cite{SV}, based on the work of Gaitsgory and Nadler \cite{GN}, only when the group is split; the general case of $L$-groups for spherical varieties has not been developed yet, although there are many examples where the answer is clear.) The map ${^LG_X} \times\SL_2\to {^LG}$ is a canonical one (up to conjugacy by the canonical Cartan subgroup of ${^LG}$). In particular, when $X, Y$ are spherical varieties for groups $G, H$, respectively, and $r:{^LG_X}\to {^LH_Y}$ is an $L$-homomorphism between their dual groups, by an extension of Langlands' ``Beyond Endoscopy'' philosophy this should induce comparisons between their (stable) relative trace formulas. The problem is highly non-trivial already when the dual groups are isomorphic, which is the case at hand here.

The spectral side involves periods of automorphic forms over stabilizers of points on $X$, and the values of periods are expected in a wide variety of cases to be related to special values of $L$-functions. More precisely, according to the conjectures of \cite[Chapter 17]{SV}, under some assumptions on $X$ the contribution of such a parameter $\varphi$ to the spectral expansion will involve a certain quotient $L_X(\varphi)$ of special values of automorphic $L$-functions associated to  the $L$-group ${^LG_X}$, whose Euler factor at an unramified place $v$ is related to the Plancherel measure for $L^2(X_v)$ (normalized in some canonical way that will not be explained here). In the group case one has $L_X=1$, which is why one does not see $L$-functions on the spectral side. Comparisons between relative trace formulas give rise to relations between periods and the associated ``relative characters'', in exactly the same way that character relations arise in the endoscopic comparison of trace formulas.

\subsection{The limitations of standard comparisons} The geometric side of the relative trace formula is, roughly (i.e., ignoring ``pure inner forms'') a sum of orbital integrals over $G(k)$-orbits on $(X_1\times X_2)(k)$; these orbits correspond, at least generically and in the stable case, to $k$-points on the ``base'' $\mathcal B = (X_1\times X_2)\sslash G:= \spec k[X_1\times X_2]^G$. The latter is very often an affine space, so summation over $G(k)$-orbits becomes some kind of ``Poisson sum''.

Experience suggests that when the dual groups and the related $L$-values for two relative trace formulas match,\footnote{We are really referring to stable or ``quasi-stable'' trace formulas here, e.g., in the case of the Arthur-Selberg trace formula the individual summands of the invariant trace formula which are matched with stable trace formulas of endoscopic groups; these summands can be considered as ``quasi-stable'' trace formulas with their own $L$-group, namely the corresponding endoscopic $L$-group.} then usually there is a natural matching of the geometric sides: a map between stable, closed rational orbits, and an identification of $G(\adele)$-coinvariants of the spaces of standard test functions on $(X_1\times X_2)(\adele)$ with the corresponding space of the second RTF which preserves (up to scalar ``transfer factors'') the corresponding orbital integrals. ``Standard test functions'', by definition, are generated by functions of the form $\prod_v \Phi_v$, where $\Phi_v$ is the characteristic 
function of $X_i(\mathfrak o_v)$ for almost every place $v$ (where $\mathfrak o_v$ denotes the integers of $k_v$), and an arbitrary Schwartz function at remaining places.

While there are many known cases of such comparisons, it is clear from the multitude of different $L$-functions attached to spherical varieties (s.\ the table at the end of \cite{SaSph}) that one cannot hope that every RTF will have one or more ``matching'' ones, in the above sense. For example, no relative trace formula comparison has been proposed for attacking the Gross--Prasad conjecture for the space $X=\SO_n^\diag\backslash(\SO_n\times \SO_{n+1})$. On the other hand, it is clear that the relative trace formula provides the correct setting for understanding the -- still far from being understood --  fine points behind the general period conjecture of \cite{SV}, such as how many elements inside of an $L$-packet contribute to the spectral expansion. 
For example, in the case of Gross--Prasad or Whittaker periods there is, locally, only one distinguished representation inside of each (Vogan) $L$-packet, which corresponds to the fact that the corresponding relative trace formulas are stable (i.e., there is, at least generically, no 
distinction between orbits and stable orbits). The appearance of pure inner forms in the conjectures can also be understood in terms of the relative trace formula, and more precisely in terms of the quotient stack $(X\times X)/G$, cf.\ \cite[\S 16.5]{SV}, \cite{SaStacks}.

This poses the dilemma: How can the relative trace formula be on one hand fundamental for the correct statement of the conjectures but on the other hand insufficient for their proof? My hope is that the answer will be given through a ``non-standard'' comparison of relative trace formulas, in the setting where dual groups match but the associated $L$-values are unequal. In order to set up an equality between two such trace formulas one has to replace the characteristic function of $X_i(\mathfrak o_v)$ by a non-standard unramified function which (as in Langlands' ``Beyond Endoscopy'' program) will force a correction by suitable $L$-values on the spectral side. Again as in the ``Beyond Endoscopy'' program, one should not expect an orbit-by-orbit comparison of the geometric sides in this case. 
Instead, the transfer factors will be integral operators between the two spaces of orbits; let us call them ``transfer operators''. The biggest conceptual difficulty here, as I see it, is to show that these operators are automorphic, i.e., that they preserve ``Poisson sums''.

This is not the first time that such a ``non-standard'' comparison has been performed. Rudnick's thesis \cite{Rudnick}, predating the ``Beyond Endoscopy'' program by more than a decade, can be seen as such a comparison between the Petersson--Kuznetsov formula and the Selberg trace formula for $\GL_2$ (restricted to holomorphic cusp forms). The success and similarity of these different cases provides grounds for optimism.

\subsection{A non-standard comparison}\label{ssnonstandardintro} In this article I employ the ideas above in the simplest possible case, namely the comparison of the Kuznetsov trace formula (i.e., the one associated to the Whittaker period) with the relative trace formula of Jacquet \cite{JW1} for the case $X=T\backslash G$. In both cases, $G$ is the group $\PGL_2$, and $T$ denotes a non-trivial torus in it, split or nonsplit. In the split case the corresponding relative trace formula has never appeared in print to the best of my knowledge, and is quite interesting analytically.

The local comparisons of trace formulas needed here were performed in the article \cite{SaBE1}, including a local matching theorem \cite[Theorem 5.1]{SaBE1} and a fundamental lemma \cite[Theorem 5.4]{SaBE1}. Here I show that they induce a global equality of relative trace formulas, and perform the spectral analysis. The end result is a new proof (see Theorem \ref{thmWaldspurger}) of the celebrated result of Waldspurger on toric periods with trivial character on the torus:

\begin{theorem}[\cite{Waldspurger-torus}] \label{Waldspurgertheorem}
Let $\pi \hookrightarrow L^2([G'])$ be a cuspidal automorphic representation of $G'$, an inner form of $G$, and write it as a restricted tensor product $\pi = \otimes^\prime_v \pi_v$ of unitary representations of $G'(k_v)$. Let $T$ be a nonsplit torus in $G'$, splitting over a quadratic extension with associated idele class character $\eta$. We endow the groups with the Tamagawa measures, and we endow $\pi$ with the norm induced from $L^2([G'])$. Then, for $\phi = \otimes_v \phi_v \in \pi$,
\begin{equation}
 \left|\int_{[T]} \phi(t) dt \right|^2 = \frac{1}{2}\prod_v' \int_{T(k_v)} \left<\pi_v(t)\varphi_v, \varphi_v\right> dt,
\end{equation}
where the Euler product, outside of a finite set of places, should be understood as a partial $L$-value; the corresponding Euler factors are
\begin{equation}\label{unr} \frac{\zeta_v(2)L(\pi_v,\frac{1}{2}) L(\pi_v\otimes \eta_v,\frac{1}{2})}{L(\eta_v,1)^2 L(\pi_v,\Ad,1)}.
\end{equation}
\end{theorem}

The determination, in terms of epsilon factors, of whether $\pi_v$ is $G_v$-distinguished, which is a result of Tunnell and Saito, can also be obtained from this relative trace formula-based approach, see Theorem \ref{thmepsilon}, although we need to rely on the proof of Jacquet, combined with Jacquet--Langlands \cite{JL}, in order to verify that the root numbers are the correct ones. The proof includes an independent proof of the meromorphic continuation of the base change $L$-function $L(\pi, \frac{1}{2}+s) L(\pi\otimes \eta,\frac{1}{2}+s)$, see Corollary \ref{analyticcont-cuspidal}, in the spirit of ``Beyond Endoscopy''.

The case of a split torus amounts to a classical result of Hecke, and will also be included in the results of this paper; in that case, the ``Tamagawa measure'' on $[T]$ (induced, formally, from a $k$-rational volume form) has to be multiplied by a factor of $\zeta_v(1)$ at almost every place, and, correspondingly, for those places the factor $L(\eta_v,1)^2 = \zeta_v(1)^2$ will disappear from the denominator of \eqref{unr}, ensuring that the result is formally the same. Notice that the factor $\frac{1}{2}$ appears because we normalize the inner product on $\pi$ according to the inner product on $L^2([\PGL_2])$; if, instead, we were using the norm of the space $L^2(\GL_2(k)\backslash \GL_2(\adele)^1)$, where $\GL_2(\adele)^1$ denotes the elements whose determinant has adelic absolute value $1$, then this factor would not appear.

The input for proving this formula is the analogous formula for Whittaker periods, cf.\ \cite[Theorem 4.1]{LM-Whittaker} or \cite[Theorem 18.3.1]{SV}, where the local $L$-factors $L(\pi_v,\frac{1}{2})L(\pi_v\otimes\eta_v,\frac{1}{2})$ from the numerator of \eqref{unr} are missing. Whittaker periods appear on the spectral side of the Kuznetsov trace formula, and the missing $L$-factors will be ``added'' to the Kuznetsov formula. Classically, this corresponds to a series of Kuznetsov formulas according to the Dirichlet series of this product of $L$-functions, as, for example, in the thesis of Rudnick \cite{Rudnick}. Adelically, it corresponds to using a space of ``non-standard'' Whittaker functions, strictly larger than the usual Schwartz space, that was explained in \cite[\S 4.5]{SaBE1}. At almost every place, this space contains a distinguished unramified ``basic vector'', corresponding to generating function of the unramified Euler factor of the above $L$-function.

To go from the formula for Whittaker periods to the formula for torus periods, we need to compare two ``trace formulas'': the relative trace formula for $T\backslash G/T$, and the Kuznetsov trace formula for $G$, corresponding to the two-sided quotient of $G$ by a unipotent subgroup $N$ equipped with a non-trivial character $\psi$. Symbolically, we write $\mathcal S(\mathcal Z(\adele))$ for the space of orbital integrals of test functions for $T\backslash G/T$  (thinking of $T\backslash G/T$ as an algebraic stack $\mathcal Z$), and similarly $\mathcal S(\mathcal W(\adele))$ for the space of orbital integrals of (our non-standard) test functions for the Kuznetsov formula (with the symbol $\mathcal W$ coming from ``Whittaker''). The invariant-theoretic quotients $T\backslash G\sslash T$ and $N\backslash G\sslash N$ are both affine lines, to be denoted by $\mathcal B$ (for ``base''). Choosing appropriate coordinates, and a trivialization of the line bundle corresponding to the character $\psi$ in the case of the Kuznetsov formula, we can think of the spaces $\mathcal S(\mathcal Z(\adele))$ and $\mathcal S(\mathcal W(\adele))$ of orbital integrals as functions on the adelic points of a Zariski dense open subset of ``regular'' points on the affine line. (This subset is different for $\mathcal Z$ and $\mathcal W$.)

The non-standard matching theorem and a fundamental lemma between those spaces of test functions were proved in \cite{SaBE1}, using a certain ``transfer operator'', that is, an explicit linear isomorphism
\begin{equation}\label{G-intro}|\bullet|\mathcal G: \mathcal S(\mathcal Z(\adele)) \xrightarrow\sim  \mathcal S(\mathcal W(\adele))
\end{equation}
obtained from local isomorphisms under which (at non-Archimedean places) the elements corresponding to orbital integrals of the ``basic vectors'' correspond, as well as the orbital integrals of their convolutions by the same element of the unramified Hecke algebra of $\PGL_2$.

In this article, I compare the global ``trace formulas''. I will consider the relative trace formula for the quotient $\mathcal Z=T\backslash G/T$ as a functional on the space $ \mathcal S(\mathcal Z(\adele))$, denoted by $\RTF$, and similarly the Kuznetsov formula as a functional on $\mathcal S(\mathcal W(\adele))$, denoted by $\KTF$. I take the geometric sides of the trace formulas as the definitions of these functionals; these are, essentially, the Poisson sums of the above functions of orbital integrals over the ``base'' $\mathcal B$:
\begin{equation}\label{naivesum}\RTF(f)\mbox{ or }\KTF(f):= \sum_{\xi\in \mathcal B(k) = k} f(\xi).
\end{equation}
The above expression has to be taken with a grain of salt, as the evaluations of $f$ at singular points of the base have to be interpreted appropriately. The precise definitions are given in \eqref{RTFdef}, \eqref{KTFdef}.

As mentioned previously, now that we have a ``transfer operator'' $\mathcal G$ instead of scalar transfer factors, the comparison of trace formulas cannot simply be obtained by the triviality of some scalars when evaluated on rational elements. Rather, to show that the transfer operator $\mathcal G$ preserves the sums \eqref{naivesum} amounts to a Poisson summation formula. What makes the argument work in this case is that the transfer operator is explicitly described as a consecutive application of Fourier transforms and birational maps on the base, thus ``in principle'' allowing for an application of the (usual) Poisson summation formula for Fourier transform. This is \emph{the first ``miracle''} on which the global method is based, and it eventually leads to a proof of equality between the two trace formulas. The word ``eventually'' conceals a lot of analytic deviousness: since elements of $\mathcal S(\mathcal Z(\adele)), \mathcal S(\mathcal W(\adele))$ are not defined as functions but in a meager set of the adeles (namely, the adelic points of a Zariski open subset), the Poisson summation formula for Fourier transform cannot be applied directly. I give an overview of the technicalities involved in the next subsection.

Relative trace formulas are averages over automorphic representations, and it is quite straightforward in the case of standard comparisons to isolate those (if one has the fundamental lemma for Hecke algebras) in order to obtain a representation-by-representation (or at least packet-by-packet) comparison. However, isolating the representations is much deeper in the case of our non-standard comparison, as the Kuznetsov formula for non-standard test functions is not given by a convergent sum, but by analytic continuation. This is to be expected, of course, since the $L$-value that we inserted is not given by a convergent Euler product. Separating representations amounts to showing that this analytically continued formula is still a measure on the space of Satake parameters, which could be established by appealing to analytic estimates for the $L$-function $L(\pi,\frac{1}{2}+ s)L(\pi\otimes\eta,\frac{1}{2}+ s)$. However, such an approach would beat the purpose, in view of possible applications of this method in higher rank. Indeed, one would like to extract properties of the pertinent $L$-functions from the relative trace formula, and not vice versa. Thus, I follow a different approach that is made possible by appealing to \emph{a second ``miracle''}; the existence of another explicit transform (also satisfying some form of Poisson summation) between certain spaces of orbital integrals; this transform is essentially a reflection of the functional equation of this $L$-function at the level of orbital integrals.

While this article introduces several analytic methods which might work quite generally, the most important question that needs to be resolved, in my opinion, is of algebraic nature: What makes the aforementioned two miracles possible, that is: \emph{why are the comparison of trace formulas and the functional equation represented by Fourier transforms and birational maps at the level of orbital integrals?} Can these miracles be generalized to higher rank? Clearly, we will not know the answer to these questions before more examples of non-standard RTF comparisons are examined.

\subsection{Poisson summation formula} Making the principle ``Fourier transforms and birational maps (i.e., the constituents of the transfer operator $\mathcal G$) preserve Poisson summation'' into a proof, namely, proving the formula
\begin{equation}
 \RTF(f) = \KTF(|\bullet|\mathcal G f),\,\,\mbox{ for all }f\in \mathcal S(\mathcal Z(\adele))
\end{equation}
involves a good deal of adelic analysis, which occupies the first part of this article. The problem is that the elements of the spaces $\mathcal S(\mathcal Z(\adele))$ and $\mathcal S(\mathcal W(\adele))$ of orbital integrals  are only defined on a Zariski dense open subset $\mathcal B^\reg$ of the base $\mathcal B$ (in our case, the affine line). (The Zariski open subset $\mathcal B^\reg$ is different for each of the two RTFs that we are considering.) Locally, the subset $\mathcal B^\reg(k_v)$ of $\mathcal B(k_v)$ is dense; globally, its adelic points, however, are of measure zero in the adelic points of $\mathcal B$. 

A first approach to the Poisson summation formula would be to replace the local factors outside of a finite set of places $S$ by standard Schwartz functions on the affine space $\mathcal B(k_v)$; this will make them honest functions on a subset of $\mathcal B(\adele)$ of full measure, and one could hope to take a limit with $S$.

However, this still does not work, because it leads to logarithmically divergent terms. The solution lies in ``deforming'' the spaces $\mathcal S(\mathcal Z(\adele))$, $\mathcal S(\mathcal W(\adele))$ of orbital integrals. In our case, the key fact is that the germs of orbital integrals (locally) at the ``irregular'' points of $\mathcal B(k_v)$ have singularities which behave like generalized eigenfunctions for the multiplicative group of a local parameter. For instance, they may be of the form
$$C_1(\xi)\log|\xi|_v+C_2(\xi),$$ where $C_1, C_2$ are smooth functions, and $\xi$ is a coordinate for $\mathcal B(k_v)$; this is a generalized eigenfunction for the multiplicative group $k_v^\times$, of degree two. We may continuously deform its eigencharacter so that our function obtains the form $D_1(t)+D_2(t)|\xi|_v^t$, where $t\in \CC$ is a parameter; this eliminates some logarithmic divergence and opens the way for an application of the Poisson summation formula (when $\Re(t)\gg 0$). (The precise forms of germs and their deformations are described in \S \ref{ssaxiomslocal}.)

To demonstrate this argument, and establish several useful facts that we need, we first prove a Poisson summation formula for a ``baby case'' in sections \ref{sec:Poisson-baby} and \ref{sec:direct-baby}, where the relative trace formulas get replaced by the quotient of a quadratic extension by the group of its elements of norm one. This is a good ``infinitesimal'' model for our theory, and proving some theorems in this setting makes it easier to understand the argument, and saves us from a lot of heavy notation. However, for the comparison of relative trace formulas that we are interested in, things are more complicated because the $L$-values inserted in the Kuznetsov trace formula are not represented by convergent Euler products. Analytically, this is reflected by the fact that the orbital integrals $\mathcal S(\mathcal W(\adele))$ of our non-standard test functions for the Kuznetsov formula are not of sufficiently fast decay at infinity, thus the sum \eqref{naivesum} does not converge. Therefore, in section \ref{sec:Poisson-RTF} we vary the $L$-values that we insert in the Kuznetsov formula with a second parameter $s$, constructing a family of spaces of non-standard orbital integrals $\mathcal S(\mathcal W^s(\adele))$, which for $s\in \mathbb C$ is tailored to produce the $L$-value $L(\pi_v,\frac{1}{2}+ s)L(\pi_v\otimes\eta_v,\frac{1}{2}+ s)$ (at least at unramified places). This is to be compared with a similar deformed space $\mathcal S(\mathcal Z^s(\adele))$ for the torus RTF, which is not interpreted as a space of orbital integrals, but just as a space of functions on $\mathcal B^\reg(\adele)$, specializing to our space of orbital integrals when $s=0$. For $\Re(s)\gg 0$, the technique described above proves a Poisson summation formula:
\begin{equation}
 \label{R=K}
 \RTF(f) = \KTF(|\bullet|^{s+1} \mathcal Gf),
\end{equation}
s.\ Theorem \ref{PSF}.

On the side of $\mathcal S(\mathcal Z^s(\adele))$, now, our test functions are of rapid decay, and the Poisson sum \eqref{naivesum} converges for arbitrary $s$. This proves the analytic continuation of the Kuznetsov formula for non-standard test functions given by an arbitrary value of $s$. The next goal is to isolate the contributions of individual $L$-packets to \eqref{R=K}. The notion of (global) $L$-packets is used in analogy to the local $L$-packets of Vogan \cite{Vogan}, and is due to the fact that the ``stack'' $\mathcal Z = T\backslash G/T$ includes contributions from inner forms of $G$. Conveniently, global packets in our case are determined by strong multiplicity one: automorphic representations for $\PGL_2$ and its inner forms belong to the same packet if and only if they are locally equivalent almost everywhere.

\subsection{Spectral decomposition}

The fact that the Kuznetsov trace formula for $s=0$ is only described as the analytic continuation of some expression which converges for large values of $\Re(s)$ makes the spectral analysis much more complicated than in usual trace formula comparisons.  However, it is important to stress that going beyond the domain of convergence does not rely on hard analytic number theory: for the sum over automorphic representations that constitutes the Kuznetsov formula, it is a direct outcome of the above argument, where the Kuznetsov sum is equated to the analogous Poisson sum for the deformation $\mathcal S(\mathcal Z^s(\adele))$ of the space of orbital integrals of the torus trace formula. 

Nonetheless, this just proves the analytic continuation of a weighted average of $L$-functions over all automorphic representations, and is not enough to separate the equality of trace formulas representation-by-representation. The problem here is that while in usual RTFs the spectral expansion (or the interesting part of it) is absolutely convergent and hence, under the action of the Hecke algebra on test functions, a measure on the set of Satake parameters, here it is not a priori so. To exhibit it as a measure on the set of Satake parameters (s.\ the theorems stated in \S \ref{ssspectral-results}), and to obtain period relations for each individual packet out of the equality \eqref{R=K} between the two trace formulas, we use what was before called the ``second miracle'', which is a reflection of the functional equation of $L(\pi,\frac{1}{2}+ s)L(\pi\otimes\eta,\frac{1}{2}+s)$ at the level of orbital integrals. This is an explicit linear isomorphism
\begin{equation}\label{T-intro} \mathcal T: \mathcal S(\mathcal W^{-s}(\adele)) \xrightarrow\sim \mathcal S(\mathcal W^s(\adele))
\end{equation}
which again satisfies the ``fundamental lemma for all elements of the unramified Hecke algebra'' and preserves Poisson sums, s.\ Theorem \ref{2ndmiracle}. An application of the Phragm\'en-Lindel\"of principle now allows us to bound the Kuznetsov formula as a functional on the Hecke algebra inside of the critical strip for the $L$-function, and to isolate terms with different Satake parameters in the equality \eqref{R=K}. This spectral analysis is performed in sections \ref{sec:spectral-main} and \ref{sec:spectral-proofs}.

\subsection{The result of Waldspurger} 

Finally, in section \ref{sec:Waldspurger} I use this packet-by-packet comparison that is obtained from the previous section to deduce the result of Hecke and Waldspurger, Theorem \ref{Waldspurgertheorem}, on toric periods with trivial character.

The starting point is the packet-by-packet identity obtained from \eqref{R=K}, which has the form
$$ \mathcal J_\varphi (f)= \mathcal I_\varphi (|\bullet|\mathcal G f).$$
Here $\varphi$ denotes a collection of Hecke eigenvalues outside of a finite set of places, corresponding to a generic automorphic representation. The distributions $\mathcal J_\varphi$ and $\mathcal I_\varphi$ are ``relative characters'' (also called spherical characters or Bessel distributions) obtained from the period functionals
$$ F \mapsto \int_{[T]} F(t) dt$$
and
$$ F\mapsto \int_{[N]} F(n) \psi^{-1}(n) dn$$
on automorphic forms. The operator $|\bullet|\mathcal G$ is the transfer operator, as before.

On the Kuznetsov side, one has a well-known Euler factorization into local functionals, that I alluded to in \S \ref{ssnonstandardintro}, which has the form
$$ \mathcal I_\varphi = \frac{1}{2}\prod^\prime_v I_{\varphi_v},$$
where the Euler product is not literally convergent (which is why it is denoted by $\prod^\prime$), but can be interpreted using partial $L$-functions.

This gives an Euler factorization of the relative character $\mathcal J_\varphi$ for the torus period, but it is quite indirect: the local factors are described as pull-backs, via the transfer operators, of the local functionals $I_{\varphi_v}$ on the Kuznetsov side.
To obtain Waldspurger's result, one needs to describe them intrinsically as $T_v$-biinvariant distributions on $G_v$ (or an inner form). This is a usual problem with the relative trace formula: what is the transfer of relative characters? Obtaining the answer is often quite involved (see, for example, \cite{Wei1}).

It is very encouraging that this method seems tailored to give the correct local factors without any complicated arguments. Namely, both $I_{\varphi_v}$ and the desired factors $J_{\varphi_v}$ can be characterized in terms of the Plancherel formula for the pertinent homogeneous spaces, and their matching via the transfer operators is an immediate consequence of the fact that transfer operators preserve $L^2$-inner products! Thus, the present method is directly fitted to the framework of the general period conjecture \cite[Conjecture 17.4.1]{SV}.

\subsection{Relation to other methods}

The two ``miracles'' that make the method of this paper possible, both of local nature, are reflections at the level of orbital integrals of methods that have been used before to prove the same final result. I explained this briefly in \cite[section 5]{SaBE1}, and explain it again in the proof of the ``second miracle'' in \S \ref{ssfe}; this fact is used to avoid local calculations -- calculations, to be sure, that can in a straightforward albeit tedious way be performed directly. In a nutshell, orbital integrals for the relative trace formula for $T\backslash G/T$ 
are equal to orbital integrals for the relative trace formula $A\backslash G/(A,\eta)$ by the work of Jacquet \cite{JW1}, where $A$ is a split torus and $\eta$ is the quadratic character associated to the splitting field of $T$, and the ``unfolding'' method of Hecke provides a passage from $C_c^\infty(A\backslash G)$ (or $C_c^\infty(A\backslash G, \eta)$) to Whittaker functions (non-compactly supported). This passage descends, roughly, to our ``transfer operator'' $|\bullet|\mathcal G$ at the level of orbital integrals. 

Because of this fact, I do not know if this method can be generalized to higher rank -- where the methods of Jacquet and Hecke certainly do not generalize. I certainly hope so, and if the adequacy of the relative trace formula for expressing the fine details of distinction, or the direct relevance of this method to the general Plancherel-theoretic setting are any indication, one has reasons to be optimistic. After the first version of this paper was written and submitted, I noticed that this method can be extended to prove the full result of Waldspurger, with a character on the torus, while Jacquet's comparison of \cite{JW1} cannot. (Jacquet eventually used an entirely different comparison in \cite{JW2} to address the general case.) In any case, one needs to examine many more examples of non-standard comparisons, which is something I plan to do in the near future.

In any case, it is important to remark that we make absolutely no use of the methods of Jacquet and Hecke in global arguments. The equality between the two relative trace formulas is obtained by completely independent means, namely the Poisson summation formula that we described above. Given that the local calculations can also be done ``by hand'', our method is completely self-contained. We also don't make use of any hard facts about $L$-functions, except for the meromorphic continuation and polynomial growth in bounded vertical strips of partial abelian $L$-functions.   

\subsection{Relation to ``Beyond Endoscopy''}

The ``Beyond Endoscopy'' pro\-ject of R.\ Langlands \cite{Langlands-BE} is a very ambitious project aiming at proving functoriality to its full extent. The vision, very simplistically, is to compare, for any embedding ${^LG}_1\to {^LG}_2$ of $L$-groups, the stable trace formula for $G_1$ with that part of the stable trace formula of $G_2$ which corresponds, spectrally, to the expected lift of representations. To isolate the latter, one uses non-standard test functions\footnote{In classical language, these non-standard test functions correspond to
``series of trace formulas''; for example, some version of the Kuznetsov trace formula can isolate the $n$-th Fourier coeffient of automorphic forms on $\GL_2$, and one takes a weighted sum over $n$ corresponding to the Dirichlet series expressing the desired $L$-function in terms of Fourier coefficients.} in the trace formula of $G_2$ to introduce suitable $L$-functions on the spectral side; and one hopes to be able to calculate residues that will ``capture'' functorial lifts.

I repeat that the above is a very simplistic presentation of the proposed project. However, even the smallest steps give rise to tremendous difficulties. In particular, a lot of effort has been focused on obtaining analytic continuation of the expressions obtained when introducing $L$-functions. Some of the papers doing this for various $L$-functions include \cite{V,Herman, White}; in particular, \cite{White} treats the same $L$-functions that we treat in this paper, showing analytic continuation in a strip beyond the domain of convergence of the Euler product. 
In applications to functoriality, one usually faces the problem of isolating individual representations in a trace formula comparison, which has not been successfully resolved: a common recourse is estimates for $L$-functions obtained by other methods \cite{V}, something which clearly should be avoided for completion of the project. An exception is \cite{Herman-FE}, where the analytic continuation of the standard $L$-function for $\GL_2$ is obtained by an analog of the functional equation at the level of orbital integrals; this is similar to the method we use here.

Comparing the present paper to the above methods, several similarities and differences should be observed:
\begin{itemize}
 \item I also use non-standard test functions to introduce $L$-functions into the trace formulas. 
 \item However, I obtain the \emph{full} analytic continuation of these $L$-functions without any hard analysis; rather, the method is more conceptual, and relies on being able to compare the Kuznetsov formula with non-standard test functions depending on a parameter $s$ (and convergent for $\Re(s)\gg 0$) with a \emph{deformation} of another relative trace formula with \emph{standard} test functions, which therefore has meromorphic continuation for \emph{every $s$}.
 \item Here, as in Beyond Endoscopy, the comparisons are not via scalar transfer factors, but by a series of Fourier transforms and other operations, i.e., by \emph{Poisson summation}. Poisson summation has been used in \cite{FLN, Altug1} to isolate the contribution of the trivial representation, and in various other references to obtain estimates for the ``series of trace formulas'' that allow continuation beyond the domain of convergence. However, to the best of my knowledge, a full comparison of trace formulas using Poisson summation formulas has not appeared before, with the notable exception of Rudnick's thesis \cite{Rudnick}.
 \item Despite these similarities in method, it should be emphasized that the type of ``relative functoriality'' that I prove here is closer in spirit to endoscopy rather than ``Beyond Endoscopy''. Indeed, if one admits the point of view that I alluded to above, namely: each invariant trace formula is a sum of its ``quasi-stable'' parts, and each has its own $L$-group, then endoscopy is a matter of comparing quasi-stable trace formulas with the \emph{same} $L$-group. Here, too, I am comparing the relative trace formula for $T\backslash G/T$ with the Kuznetsov trace formula for $G(=\PGL_2)$; both are stable, and their $L$-group is $\SL_2$. In contrast, the goal of Langlands is to ``extract'' from a given stable trace formula the contribution of a smaller $L$-group. While this is far from my scope, the non-standard comparisons that I introduce here, and hope to study in the future, may give some indication of more conceptual ways to proceed with the desired comparisons of ``Beyond Endoscopy''.
\end{itemize}

\subsection{Acknowledgements} I would like to thank Akshay Venkatesh for pointing my attention to his thesis \cite{V} as a possible source of ideas for attacking the period conjectures of \cite{SV}. I would also like to thank Joseph Bernstein, who taught me the correct way to think about several aspects of the relative trace formula. Finally, I am very grateful to the anonymous referee for a careful reading and numerous small corrections and suggestions. This work was supported by NSF grants DMS-1101471 and DMS-1502270. 

\subsection{Notation}\label{ssnotation} Some of the notation is local, redefined in every section; for example, $X$ and $Y$ are reserved for varieties which change throughout the text. Here we give a summary of the symbols that are used globally; more notation, used in the second part of the paper, is introduced and summarized in \S \ref{notation-spectral}.

\begin{itemize}
 \item $k$ is a global field with ring of adeles $\adele$, $E$ a quadratic etale extension, hence either a quadratic field extension of $k$ or the ring $k\oplus k$. The ring of integers of $k$ at a non-Archimedean place $v$ will be denoted by $\mathfrak o_v$, its residue field degree by $q_v$ and a uniformizing element by $\varpi_v$. We will denote $E_v = E\otimes_k k_v$, and at non-Archimedean places $\mathfrak o_{E_v}$ will denote the ring of integers of $E_v$. The quadratic idele class character associated to $E$ is denoted by $\eta$. If it is clear from the context that $\xi$ is an element of $k_v^\times$, We feel free to write $\eta(\xi)$ for the evaluation of $\eta$ via the embedding $k_v^\times\hookrightarrow\adele^\times$; but sometimes, for emphasis, we write $\eta_v$ instead. The same comment holds for absolute values, as well as zeta- and $L$-functions: we write $L(\eta_v,s)$ or $L_v(\eta_v,s)$ when we want to emphasize that we are referring to the local factors, etc. The usual, unnormalized, absolute values which satisfy the product formula are being used on the completions $k_v$; thus, for non-Archimedean places the absolute value is $q_v$ raised to the opposite of the valuation, while for complex places the absolute value is the square of the usual one. For a variety $X$ over $k$, we denote $X_v:= X(k_v)$.
 \item We fix throughout a complex character $\psi$ of $\adele/k$ and a factorization $\psi= \prod_v \psi_v$, such that outside of a finite set of places the conductor of $\psi_v$ is the ring of integers of $k_v$.
 \item We fix the standard Tamagawa measure $dx$ on $\adele$, together with the factorization: $dx=\prod_v dx_v$ into self-dual measures with respect to the characters $\psi_v$. For non-Archimedean places unramified over $\QQ_p$ or $\FF_p((t))$, when the conductor of $\psi_v$ is $\mathfrak o_v$ this measure is such that $dx_v(\mathfrak o_v)=1$.
 \item Fourier transform on $k_v$ is defined as:
 \begin{equation}\label{Fourier-intro} \mathcal F(\Phi) (x)= \hat \Phi (x):= \int_{k_v} \Phi(y) \psi_v^{-1}(xy) dy,
 \end{equation}
 where $dy$ and $\psi_v$ are the aforementioned measure and character.
 \item The ``transfer operator''\footnote{This is the transfer operator for the ``baby case''. For the comparison between the two relative trace formulas, the transfer operator is $|\bullet| \mathcal G$, as in \eqref{G-intro}.} $\mathcal G$ is defined between certain spaces of densely defined functions on $k_v$, considered as tempered distributions (sometimes by analytic continuation). It is given by: 
$$ \mathcal G = \mathcal F \circ \iota\circ \mathcal F,$$
where $\mathcal F$ is the Fourier transform \eqref{Fourier-intro}, and $\iota$ is the transformation:
$$ \iota f(x) = \frac{\eta_v(x)}{|x|_v} f\left(\frac{1}{x}\right).$$ 
 \item The action of the multiplicative group $k_v^\times$ on functions on $k_v$ is normalized in \eqref{unitaryaction} in order to be unitary on $L^2$; this normalization makes Fourier transform on $k_v$ anti-equivariant with respect to the action of $k_v^\times$.
 \item A finite set $S$ of places of $k$ will always, implicitly, include all Archime\-dean places, all places which are ramified over $\QQ_p$ or $\FF_p((t))$, and all finite places where $\psi_v$ does not have the ring of integers as its conductor, together with any other places specified in the text. Statements about ``almost all'' places will, implicitly, exclude such a set $S$. The $S$-integers of $k$ will be denoted by $k_S$.
 \item For a finite set $S$ of places, expressions of the form $\zeta^S(s)$, $L^S(\eta,s)$, $L^S(\pi,s)$ will denote the partial $L$-functions indicated (i.e., the $L$-functions with the Euler factors at places of $S$ omitted). We will use $(\zeta^S(s))^*$, $(L^S(\eta,s))^*$, etc.\ to denote the leading term in the Laurent expansion of this partial $L$-function at $s$.
 \item For $p$-adic groups, the usual notion of ``smooth'' vectors and representations typically gives rise to LF vector spaces, i.e., strict inductive limits of Fr\'echet spaces. To achieve uniformity with the Archimedean case, I described in \cite[Appendix A]{SaBE1} a notion of ``almost smooth'' vectors which gives rise to Fr\'echet space representations. For uniformity of presentation, we work with such Fr\'echet spaces both in the Archimedean and non-Archimedean cases, calling these vectors (by abuse of language) ``smooth''; however, the reader can ignore this and focus on smooth vectors in the traditional sense, replacing the Fr\'echet spaces that we consider with the corresponding LF spaces of their smooth vectors. 
 \item $\mathcal S$ generally denotes spaces of test functions or their orbital integrals. Unless otherwise stated, for a smooth variety $Y$ we denote by $\mathcal S(Y_v)$ the space of Schwartz functions on $Y_v=Y(k_v)$, namely the space of rapidly decaying, smooth functions on $Y_v$. Here, again, at non-Archimedean places one can consider the usual LF space of locally constant, compactly supported functions, or the Fr\'echet space of ``almost smooth'', rapidly decaying functions, defined in \cite[Appendix A]{SaBE1}. The reader can choose to consider either of the two, but for uniformity of language we will be talking about Fr\'echet spaces. 

The spaces denoted by $\mathcal S$ are always sections of ``Schwartz cosheaves'' in the language of \cite[Appendix B]{SaBE1}. Thus, for example, $\mathcal S(k_v^\times)$ denotes smooth functions that vanish faster than positive and negative powers of $x$ both close to $0$ and close to $\infty$.
 \item The notion of ``stalk'' for Schwartz cosheaves over a closed (semialgebraic) subset $Z\subset Y_v$ is again that of \cite[Appendix B]{SaBE1}: by definition, the stalk is the quotient of sections over $Y_v$ by sections over $Y_v\smallsetminus Z$. For example, two Schwartz functions have the same germ over $Z$ if they differ by a Schwartz function on the complement of $Z$. 
\item For certain families of Fr\'echet spaces parametrized by a complex parameter we introduce in Appendix \ref{sec:families} notions of ``polynomial families of seminorms'' and of ``sections of polynomial growth/rapid decay''. These notions always refer to their behavior as the parameter varies in \emph{bounded vertical strips}.
 \item The action of an element $h$ in the Hecke algebra of smooth, compactly supported measures on a $p$-adic or real group $G$ on a vector $v$ in a smooth representation is denoted by $h\star v$; we denote by $h^\vee$ the linear dual of $h$, i.e., $h^\vee(g) = h(g^{-1})$. (No topology on the full Hecke algebra will be used in this paper, and we will only need locally constant, compactly supported measures at non-Archimedean places, but again one can instead consider the Fr\'echet Hecke algebra of ``almost smooth'' Schwartz measures, both at Archimedean and non-Archimedean places.)
 \item $G$ denotes the group $\PGL_2$ over $k$. We let $[G]=G(k)\backslash G(\adele)$ (and similarly for other groups), and $[G]_\emptyset =  [G]_\emptyset = A(k)N(\adele)\backslash G(\adele)$, where $B=AN$ is a Borel subgroup of $G$ with a Levi decomposition, with $N$ the unipotent radical of $B$.
\item The (smooth) principal series representation of $G_v$ unitarily induced from a character $\chi_v$ of a Borel subgroup $B(k_v)$ is denoted by $I(\chi_v)$; that is, $I(\chi_v)$ is the space of smooth function on $G_v$ satisfying: $f(bg) = \chi_v\delta^\frac{1}{2}(b) f(g)$ for every $b\in B_v$ and $G\in G_v$, where $\delta$ denotes the modular character of the Borel (i.e., the quotient of right by left Haar measure).
\item $T$ denotes a $k$-torus in $G$, associated to the quadratic extension $E$ that we fixed before; that is, $T$ is split if $E=k\oplus k$, and it splits over $E$, if $E$ is a field. It can be identified with the quotient $\Res_{E/k}\Gm/\Gm$, or with the group of norm one elements of $\Res_{E/k}\Gm$. For an isomorphism class $\beta$ of $T$-torsors (over $k$ or a completion $k_v$), we let $T^\beta$, $G^\beta$ denote the isomorphism classes of the groups $\Aut_T(R^\beta)$ and $\Aut_G(G\times^T R^\beta)$, where $R^\beta$ is a representative of $\beta$; here and throughout, $\times^T$ denotes the quotient of the product of the two varieties by the (free) diagonal action of $T$.
 Of course, $T^\beta$ is isomorphic to $T$, while $G^\beta$ is an inner form of $G$, and such a realization (fixing $R^\beta$) gives rise to an embedding: $T^\beta\hookrightarrow G^\beta$.
\item $\mathcal Z$ is used as a symbol for the quotient stack $T\backslash G/T$. The stack-theoretic point of view is not necessary for reading this paper; what matters is the space $\mathcal S(\mathcal Z_v)$ of orbital integrals for the $T\backslash G/T$-relative trace formula at a place $v$ of $k$. This space was introduced in \cite{SaBE1}, s.\ references in \S \ref{globalSchwartz}. I note that it encodes orbital integrals not only for a single pair $(G,T)$, but of a class of such pairs parametrized by the first Galois cohomology group of $T$ (i.e., by the set of isomorphism classes of $T$-torsors).
\item Similarly, $\mathcal W$ is a symbol for the stack $N\backslash G/N$, but equipped with a line bundle that is determined by the non-degenerate character $\psi$ of $N(\adele)\simeq \adele$, where $N$ is identified with $\Ga$ over $k$. Again, the stack-theoretic point of view is not necessary, and instead what matters is a space $\mathcal S(\mathcal W_v)$ of orbital integrals (s.\ again \ref{globalSchwartz}). These are orbital integrals of a space of \emph{non-standard Whittaker functions} tailored to produce the $L$-value $L(\pi_v,\frac{1}{2})L(\pi_v\otimes\eta_v,\frac{1}{2})$. This space will also be denoted by $\mathcal S(\mathcal W_v^0)$, where $\mathcal S(\mathcal W_v^s)$ is, more generally, a space of orbital integrals of Whittaker functions tailored to produce the $L$-value $L(\pi_v,\frac{1}{2}+ s)L(\pi_v\otimes\eta_v,\frac{1}{2}+ s)$.
\item Similarly, $\mathcal X$ is a symbol for the stack $\Res_{E/k}\Ga/T$, where $T, E$ are as above, with $T$ now identified with the group of elements of $E^\times$ of norm one. This is the ``baby case'' -- an infinitesimal version of $T\backslash G/T$ -- discussed in sections \ref{sec:Poisson-baby} and \ref{sec:direct-baby}, and again it is not the stack-theoretic point of view that matters, but the associated space $\mathcal S(\mathcal X_v)$ of orbital integrals.
\item $\mathcal B$ denotes one-dimensional affine space, the ``base'' of our quotient stacks $\mathcal X, \mathcal Z,\mathcal W$. In each of these cases, $\mathcal B$ is identified with the associated invariant theoretic quotient, except that in the Kuznetsov case ($\mathcal W$) we also invert the variable. (So, the invariant-theoretic quotient $\spec k[N\backslash G]^N$ is isomorphic to the affine line, but identified with $\Gm \cup \{\infty\}$, with $\Gm\subset \mathcal B$, while in the other cases $\mathcal B \simeq \spec k[T\backslash G]^T \simeq \spec k[\Res_{E/k} \Ga]^T$ through specific isomorphisms that we fix.)

In each of these cases, there is an open dense subset $\mathcal B^\reg$ of $\mathcal B$ (different in each case), which is identified with the ``regular'' set of the corresponding quotient. We will be using this notation when it is clear which quotient space we are referring to, and the notation $\mathcal B^\reg_{\mathcal Z},\mathcal B^\reg_{\mathcal W}$, etc.\ when we want to indicate the quotient space.

The Schwartz spaces of orbital integrals $\mathcal S(\mathcal X_v), \mathcal S(\mathcal Z_v), \mathcal S(\mathcal W_v)$ are identified with functions on the pertinent open subset $\mathcal B^\reg_v = \mathcal B^\reg(k_v)$ of regular points on the base.

\end{itemize}

\part{Poisson summation}

\section{Generalities and the baby case.} \label{sec:Poisson-baby}

\subsection{Global Schwartz spaces} \label{globalSchwartz}

If $v$ is any place of $k$, I introduced in \cite{SaBE1} certain local ``Schwartz spaces'' of measures and functions on a dense open subset $\mathcal B^\reg$ (depending on the case considered) of $\mathcal B(k_v) = k_v$, denoted (here with the appropriate subscript $v$) by
$$ \mathcal M(\mathcal X_v), \,\, \mathcal M(\mathcal Z_v), \,\, \mathcal M(\mathcal W^s_v),$$
for the measures, resp.
$$ \mathcal S(\mathcal X_v), \,\, \mathcal S(\mathcal Z_v), \,\, \mathcal S(\mathcal W^s_v)$$
for the functions. The reader may restrict their attention to the quotient $\mathcal X$ and then read the treatment of the ``baby case'' in this and the next section, before returning here for the general definitions.

All of these spaces are obtained as push-forwards of measures, resp.\ regular orbital integrals, for certain quotients: the first is for the quotient (symbolically written $\mathcal X$) of a two-dimensional quadratic space by $T = \SO_2$ (or, equivalently, a two-dimensional etale algebra over our base field, divided by the kernel of the norm map) \cite[\S 2.4, 2.5, 2.9]{SaBE1}, the second (denoted $\mathcal Z$) for the quotient $(T\backslash G \times T\backslash G)/G = T\backslash G/T$, where $G = \PGL_2$ \cite[\S 3.5]{SaBE1}, and the third for the quotient associated to the Kuznetsov trace formula for $\PGL_2$, but equipped with non-standard test functions depending on a parameter $s$, as explained in \cite[\S 4.3, 4.5, 4.6, 6.1]{SaBE1}. When $s=0$, we will also be using the notation $\mathcal S(\mathcal W_v)$, without the $s$-exponent. All of these spaces are Fr\'echet and, more precisely, sections of ``Schwartz cosheaves'' over $\PP^1(k_v)$, in the language of \cite[Appendix B]{SaBE1}.

Let us for now use the symbol $\mathcal Y$ to stand for either of $\mathcal X, \mathcal Z$ or $\mathcal W^s$. The spaces of measures are more canonical (since they are obtained as push-forwards of measures ``upstairs'', although there are some choices involved in the case of $\mathcal W^s$ in order to trivialize the bundles associated to a character of a unipotent subgroup), but in any case we have defined, by choosing suitable measures on $\mathcal B(k_v)$, linear isomorphisms
\begin{equation}\label{MtoS}
 \mathcal M(\mathcal Y_v) \xrightarrow{\sim} \mathcal S(\mathcal Y_v),
\end{equation}
These isomorphisms are not important for the present article, since we will not be working with measures.

Now assume that the above quotients are defined over a global field $k$. That is, there is a quadratic etale algebra $E/k$ and, letting $T = $ the kernel of the norm map from $E$ to $k$, we have $\mathcal X = \Res_{E/k}\Ga/T$, and $\mathcal Z = T\backslash \PGL_2/T$, where $T$ has been embedded in $\PGL_2$. For the Kuznetsov quotient $\mathcal W^s$, we need to identify a unipotent subgroup $N$ of $\PGL_2$ with $\Ga$ (over $k$), and then use the additive adele class character $\psi$ and its factorization that we fixed in the notation section \ref{ssnotation}.

We will now define global Schwartz spaces as restricted tensor products of these local Schwartz spaces. For this, we will need to choose a ``basic vector'' at almost every place. It would seem more natural at first, for $\mathcal Y = \mathcal X $ or $\mathcal Z$, to choose the characteristic measure of the integral points of the space ``upstairs'', and take its image in $\mathcal S(\mathcal Y_v)$ to be the basic vector. For example, for $\mathcal Y = \mathcal X = \Res_{E/k} \Ga$, this vector at a non-Archimedean place $v$ would be the image in $\mathcal M(\mathcal X_v)$ of the probability Haar measure on the ring of integers $\mathfrak o_{E_v}$. Let us temporarily call this the ``characteristic measure of $\mathcal Y(\mathfrak o_v)$''. (The case $\mathcal Y = \mathcal W^s$ should be treated separately.)

One can see that the restricted tensor product $\otimes^\prime_v\mathcal M(\mathcal Y_v)$ with respect to this ``basic vector'' is a vector space of well-defined measures on $\mathcal B(\adele)$. However, these push-forwards are not absolutely continuous with respect to Haar measure on $\mathcal B(\adele) = \adele$, and hence do not naturally give rise to functions on that space. Therefore, this basic vector is not well-suited for evaluation at rational points. 

Instead, we will be working with the spaces $\mathcal S(\mathcal Y_v)$ of Schwartz functions and define, in all three cases, a ``global Schwartz space of functions''
\begin{equation}
 \mathcal S(\mathcal Y(\adele)) := \widehat\otimes^\prime_v \mathcal S(\mathcal Y_v),
\end{equation}
defined with respect to some ``basic vectors/functions'' $f_{\mathcal X_v}^0$, $f_{\mathcal Z_v}^0$, $f_{\mathcal W_v^s}^0$.

For $\mathcal Y = \mathcal X$ or $\mathcal Z$, the basic vector is \emph{not} the function corresponding under (\ref{MtoS}) to the characteristic measure of $\mathcal Y(\mathfrak o_v)$; instead it will be taken \emph{with respect to the multiple thereof},  \emph{characterized by the the property}:
\begin{equation}\label{basicvectornormalization}
 f_{\mathcal Y_v}^0|_{\mathcal B^\reg(\mathfrak o_v)} \equiv 1.
\end{equation}
This way, elements of the restricted tensor product can be considered as functions on $\mathcal B^\reg(\adele)$. Explicitly, the basic vector for $\mathcal Y_v = \mathcal X_v$ and $\mathcal Z_v$ is obtained from the $T(k_v)$-orbital integrals of the characteristic function of $\mathfrak o_{E_v}$, resp.\ $T\backslash \PGL_2(\mathfrak o_v)$, with the measure on $T(\mathfrak o_v)$ normalized to be equal to $1$. In the case of $\mathcal Z_v$, if we think of it as the quotient of $(T\backslash G)^2$ by the diagonal action of $G$, then the basic vector $f_{\mathcal Z_v}^0$ is obtained from the $G(k_v)$-orbital integrals of the characteristic function of $(T\backslash G)^2(\mathfrak o_v)$ with the measure on $G(\mathfrak o_v)$ equal to $1$.

I remind that coordinates have been chosen in \cite{SaBE1} so that the regular sets are $\mathcal B_{\mathcal X}^\reg = \mathcal B\smallsetminus\{0\}$ and $\mathcal B_{\mathcal Z}^\reg= \mathcal B\smallsetminus\{0,-1\}$. Notice that $\mathcal B^\reg(\adele)$ is a set of additive measure zero in $\mathcal B(\adele)$, and that elements of these restricted tensor products \emph{do not} make sense as functions on $\mathcal B(\adele)$. 

For $\mathcal S(\mathcal W^0_v)$ the basic vector $f_{\mathcal W_v^s}^0$ will be a \emph{multiple} of what was called ``basic vector'' and denoted by $f_s^0$ in \cite[\S 6.2, (6.3)]{SaBE1}. I postpone the precise definition of this scalar (which does not depend on $s$), and point the reader to \S \ref{spacesRTF}. The value of $f_{\mathcal W_v^s}^0$ on  $\mathcal B_{\mathcal W}^\reg(\mathfrak o_v) = \mathfrak o_v^\times$ is that of the function $f_{4,v}^{0,0}$ in table \eqref{table}. The Euler product of these regular values does not make sense unless $\Re(s)\gg 0$, which is enough for us because it is only for such values of $s$ that the Kuznetsov trace formula will be expressible as a convergent sum. For $s=0$, the basic vector $f_{\mathcal W_v^s}^0$ is the image of the basic vector $f_{\mathcal Z_v}^0$ under the transfer operator 
$$|\bullet|\mathcal G: \mathcal S(\mathcal Z_v) \xrightarrow\sim \mathcal S(\mathcal W^s_v)$$
that we will recall later.

Notice that for a finite number of factors, there is a canonical notion of completed tensor product, since the spaces are nuclear. The infinite, restricted tensor product over all places is therefore an LF topological vector space.

\subsection{Relative trace formulas and their comparison}

The relative trace formula for each of the cases $\mathcal Y = \mathcal X, \mathcal Z, \mathcal W^s$ will be a functional
$$ \mathcal S(\mathcal Y(\adele)) \to \CC$$
defined, roughly, as ``summation over $k$-points of the base $\mathcal B$''. We are not ready to give a self-contained definition, but we will define it with references to the following sections. 

In the baby case $\mathcal Y=\mathcal X$, where $\mathcal B_{\mathcal X}^\reg= \mathcal B\smallsetminus\{0\}$, the functional is defined as
\begin{equation} f\mapsto \sum_{\xi\ne 0} f(\xi) + \tilde O_0(f),\end{equation}
where $\tilde O_0$ denotes an ``irregular orbital integral'' to be defined in \S \ref{ssPoissonbaby}. 

In the torus case, $\mathcal Y=\mathcal Z$, where $\mathcal B^\reg_{\mathcal Z}= \mathcal B\smallsetminus\{0,-1\}$, the functional will be denoted by $\RTF$ and is defined as
\begin{equation}\label{RTFdef}
 \RTF(f) = \sum_{\xi\ne 0, -1} f(\xi) + \tilde O_0(f) + \tilde O_{-1}(f).
\end{equation}
Because the local behavior of orbital integrals around $\xi = 0, -1$ is the same as in the baby case around zero (for $\xi = -1$, just change the variable $\xi$ to $\xi+1$), the definition of $\tilde O_0$ and $\tilde O_{-1}$ is completely analogous to that of $\tilde O_0$ in the baby case and will not be given separately. 

Finally, in the Kuznetsov case $\mathcal Y=\mathcal W^s$ the functional will be denoted by $\KTF$. The definition here cannot be given by an explicit sum except when $\Re(s)\gg 0$. For other values of $s$, it will be obtained by analytic continuation in \S \ref{sscontinuation}. For $\Re(s)\gg 0$ we define:
\begin{equation}\label{KTFdef}
 \KTF(f) = \sum_{\xi \in k^\times} f(\xi) + \tilde O_0(f) + \tilde O_\infty(f).
\end{equation}
The irregular orbital integral at $0$ is defined in the same way as for the previous cases, and the one at $\infty$ is defined in \S \ref{ssirregular}.

The goal of the first part of the paper, except for proving the analytic continuation of $\KTF$ to $s=0$, is to show that our ``transfer operators'' preserve these functionals. More precisely, in the baby case we have the transfer operator (which will be recalled in \S \ref{ssPoissonbaby}):
$$ \mathcal G : \mathcal S(\mathcal X(\adele)) \xrightarrow{\sim} \mathcal S(\mathcal X(\adele)).$$

Between the Schwartz spaces of $\mathcal Z$ and $\mathcal W$, we have according to \cite{SaBE1} the ``transfer operator'':
$$ |\bullet|\cdot \mathcal G: \mathcal S(\mathcal Z(\adele))\xrightarrow\sim \mathcal S(\mathcal W(\adele)).$$

In both cases, our main goal is to show that \emph{these transfer operators preserve the corresponding relative trace formulas}; this amounts to a ``Poisson summation formula'' for the transfer operators.

\subsection{Preliminaries on tori}\label{sstoriprelim}

We recall a few facts about Galois cohomology: By Kottwitz's 
interpretation of Tate-Nakayama duality \cite{Kottwitz}, for every torus $T$ over $k$ and any place $v$ of $k$ we have isomorphisms:
\begin{equation}
 H^1(k_v,T) = \pi_0\left(\check T^{\Gamma_v}\right)^*,
\end{equation}
where $\check T$ denotes the connected dual torus of $T$, $\Gamma_v$ is the decomposition group at $v$ of the absolute Galois group $\Gamma$ of $k$, $\pi_0$ denotes the connected components and star denotes the Pontryagin dual. Moreover, the restriction maps
$$\pi_0\left(\check T^{\Gamma_v}\right)^* \to \left(\check T^{\Gamma}\right)^*$$
give rise to an exact sequence
\begin{equation}\label{Hasseprinciple}
 1\to \ker^1(k,T)\to  H^1(k,T) \to \prod_v' H^1(k_v,T) \to \left(\check T^{\Gamma}\right)^* ,
\end{equation}
where $\prod^\prime$ indicates that almost all factors should be equal to $1$ and $\ker^1(k,T)$ is defined by this sequence. (In our case, both split and nonsplit, $\ker^1(k,T)=1$ and the last map is surjective, but with a view towards possible generalizations we won't use that.)

Now we introduce a notion of ``average volume'' for $T(k_v)$ or $[T]$. Most of it was introduced in \cite[\S 2.5]{SaBE1}  already, but it was not called so.

If $F$ is a local field and $T$ is a torus over $F$, we have a canonical ``absolute value'' map
$$ T(F) \xrightarrow{\operatorname{abs}} V_T:=\Hom(\varchi^*(T)_F,\RR_+^\times)$$
given by the pairing: $(t,\chi)\mapsto |\chi(t)|$ ($t\in T(F),\chi\in \varchi^*(T)_F$), where $\varchi^*(T)_F$ denotes the group of $F$-rational characters of $T$. This character group being free of some rank $r$, if we choose generators we have an identification of $V_T$ with $(\RR^\times_+)^r$, so we can endow it with the Haar measure corresponding to the standard Haar measure $d^\times x=\frac{dx}{x}$ on $\RR^\times_+$; this measure on $V_T$ does not depend on the choice of generators.

Now, given a Haar measure $dt$ on $T$ we define:
$$\AvgVol(T(F)) = \lim_{N\to\infty} \frac{dt(\operatorname{abs}^{-1}\left([\frac{1}{N},N]^r\right))}{d^\times x\left([\frac{1}{N},N]^r\right)},$$
where we have used the identification of $V_T$ with $(\RR^\times_+)^r$. (This is the same number as what was denoted by $\Vol(T(F)_0)$ in \cite[\S 2.5]{SaBE1}.)

Globally, now, if $T$ is defined over a global field $k$, we similarly have a canonical map
$$ T(\adele) \to [T] \xrightarrow{\operatorname{abs}} V_T:=\Hom(\varchi^*(T)_k,\RR_+^\times),$$
where as usual: $[T]=T(\adele)/T(k)$. Thus, given a Haar measure $dt$ on $[T]$, by choosing a similar isomorphism of the right hand side with $(\RR^\times_+)^r$, we define:
$$\AvgVol([T]) = \lim_{N\to\infty} \frac{dt(\operatorname{abs}^{-1}\left([\frac{1}{N},N]^r\right))}{d^\times x\left([\frac{1}{N},N]^r\right)}.$$

Of course, the ``average volume'' is just the volume if $T$ is anisotropic (both locally and globally).

\subsection{Global irregular distributions, global transfer operators, and statement of Poisson summation in the baby case}\label{ssPoissonbaby}

Let us now focus our attention to the ``baby case'' of \cite[section 2]{SaBE1}, namely $E/k$ a quadratic etale extension (with associated idele class character $\eta$) and $\mathcal X:= X/ T$, where $X=\Res_{E/k} \Ga $ and $T = \ker N_k^E$ (kernel of the norm map from $E^\times$ to $k^\times$). In this case, the Schwartz space $\mathcal S(\mathcal X_v)$ constists of sections of the Schwartz cosheaf over $\mathcal B(k_v)=k_v$ consisting of functions on $\mathcal B^\reg(k_v) = k_v^\times$ with the following description:
\begin{itemize}
 \item over $k_v^\times$, the cosheaf coincides with the usual cosheaf of Schwartz functions;
 \item in a neighborhood of zero, the functions are of the form
$$ f(\xi) = \begin{cases} -C_1(\xi) \log|\xi|_v + C_2(\xi) ,&\mbox{ if }\eta_v=1\\
C_1(\xi)+C_2(\xi)\eta_v(\xi),&\mbox{ if } \eta_v \ne 1 \end{cases}$$
where the $C_i$'s are smooth functions.
\end{itemize}

The map
 $$\mathcal S(k_v) \oplus \mathcal S(k_v) \to \mathcal S(\mathcal X_v)$$
given by
 $$(C_1,C_2)\mapsto \mbox{ the function $f$ as above}$$
identifies $\mathcal S(\mathcal X_v)$ with a quotient of the Fr\'echet space $\mathcal S(k_v) \oplus \mathcal S(k_v)$. The induced quotient topology coincides with the topology obtained by thinking of $\mathcal S(\mathcal X_v)$ as a coinvariant space.

Recall that the ``transfer operator'' $\mathcal G$ is an automorphism of the local Schwartz space $\mathcal S(\mathcal X_v)$; it is defined as the composition
\begin{equation}\label{G} \mathcal G = \mathcal F \circ \iota \circ \mathcal F,\end{equation}
where $\mathcal F$ is usual Fourier transform with respect to a fixed character as described in the notation section of the introduction and $\iota$ is the operator
$$ \iota(f)(\bullet) = \frac{\eta_v(\bullet)}{|\bullet|}f\left(\frac{1}{\bullet}\right).$$
 Moreover, by \cite{SaBE1} Propositions 2.10 and  2.16, it relates the orbital integrals of a function ``upstairs'' with those of its Fourier transform. (Here ``upstairs'' does not refer only to $E_v$ but also to a second copy of it in the nonsplit case -- we will recall how to lift elements of $\mathcal S(\mathcal X_v)$ later in \S \ref{indirect-nonsplit}.) In particular, since the characteristic function of $\mathfrak o_{E_v}$ is stable under Fourier transform (at almost every non-Archimedean place $v$), the transform $\mathcal G$ preserves the ``basic vectors'', and hence gives rise to an automorphism of the global Schwartz space:
\begin{equation}
 \mathcal G: \mathcal S(\mathcal X(\adele)) \xrightarrow{\sim} \mathcal S(\mathcal X(\adele)).
\end{equation}

In \cite{SaBE1} Propositions 2.5 and 2.14 I defined local ``irregular distributions'' $\tilde O_{0_v}, \tilde O_{u_v}$ (when $E_v$ is split) and $\tilde O_{0_v,+}$, $\tilde O_{0_v,-}$ (when $E_v$ is inert) on $\mathcal S(\mathcal X_v)$; we have added here an index $v$ to clarify that we are talking about $0\in \mathcal B_v$. We recall the definitions: 

In the split case, 
\begin{equation}\label{Oasympsplit}f(\xi)= - C_1(\xi)\cdot \log |\xi|_v + C_2(\xi),
\end{equation}
we set $\tilde O_{0_v}(f) = C_1(0)$ and $\tilde O_{u_v}(f) = C_2(0)$.

In the nonsplit case, 
\begin{equation}\label{Oasympnonsplit} f(\xi) =C_1(\xi)+C_2(\xi)\eta_v(\xi),
\end{equation}
we set $\tilde O_{0_v,+}(f) = C_1(0)$ and $\tilde O_{0_v,-}(f)=C_2(0)$.

All the distributions with index $0_v$ can be unified by the notation $\tilde O_{0_v,\kappa}$, where $\kappa\in H^1(k_v,T)^*$ (and $~^*$ denotes the Pontryagin dual); in the split case this cohomology group is trivial, and in the nonsplit case we let the index ``$-$'' correspond to the nontrivial element of $H^1(k_v,T)^*$. Moreover, for every $v$ we have a natural map
$$\check T^{\Gamma}\to \pi_0\left(\check T^{\Gamma_v}\right) =  H^1(k_v,T)^*, $$
so for every $\kappa \in \check T^{\Gamma}$ and every $v$ we define $\tilde O_{0_v,\kappa}$ via the image of this map.

In the \emph{nonsplit case} we define globally, for $\kappa \in \check T^{\Gamma}$:
\begin{equation}
 \tilde O_{0,\kappa} = \prod_v' \tilde O_{0_v,\kappa}.
\end{equation}
The Euler product on the right hand side is not absolutely convergent (which is why we denote it by $\prod^\prime$), and outside of a finite number of places $S$ it should be interpreted as a special value of a Dirichlet $L$-function. More precisely, recall that the basic vector is that multiple of the orbital integrals of $1_{\mathfrak o_{E_v}}$ under the norm map $E_v\to k_v=\mathcal B(k_v)$ which satisfies the normalization condition \eqref{basicvectornormalization}. It is easy to compute it: 

\begin{lemma}
 The basic vector $f_{\mathcal X_v}^0\in \mathcal S(\mathcal X_v)$ (defined for almost every finite place $v$, with residual degree $q_v$) is supported on the set of $\xi\in \mathcal B^\reg_{\mathcal X}(k_v) = k_v^\times$ with $|\xi|_v\le 1$, and its value on those points is equal to
$$ \begin{cases}
              1-\log_{q_v}|\xi|_v & \mbox{ if $E_v/k_v$ is split,}\\
              \frac{1+\eta_v(\xi)}{2} & \mbox{ if $E_v/k_v$ is non-split.}
             \end{cases}$$
Therefore, in the split case we have:
$$ \tilde O_{0_v}(f_{\mathcal X_v}^0) = (\log q_v)^{-1}, \,\, \tilde O_{u_v} (f_{\mathcal X_v}^0) = 1,$$
while in the nonsplit case: 
$$ \tilde O_{0_v,+}(f_{\mathcal X_v}^0) = \frac{1}{2} = \tilde O_{0_v,-}(f_{\mathcal X_v}^0).$$
\end{lemma}

Notice that in each case the values of $\tilde O_{0_v,\kappa}(f_{\mathcal X_v}^0)$ are the leading coefficients of the local $L$-function $L_v^S(\eta_v,s)$ at $s=0$, and that the corresponding global $L$-function is, in the nonsplit case, regular at $s=0$. We now interpret the partial Euler product, in the nonsplit case:
$$\prod_{v\notin S}' \tilde O_{0_v,\kappa}(f_{\mathcal X_v}^0)$$ \emph{as the leading term of the Taylor series of the partial $L$-function}:
\begin{equation}
 L^S(\eta,t)
\end{equation}
at $t = 0$. It is clear that the full ``Euler product'' $\prod_v' \tilde O_{0_v,\kappa}$, defined as the product of the above with the factors corresponding to the remaining places, is independent of the set $S$ chosen to define it.

Finally, we are ready to define the contribution of $0$ to the global Poisson sum in the nonsplit case. The notation $\tilde O_0$ that follows should not be confused with the local notation of \cite{SaBE1}, which we designate here by the additional index $v$:

\begin{definition}
 In the nonsplit case we define the distribution $\tilde O_0$ on the global Schwartz space $\mathcal S(\mathcal X(\adele))$ as
\begin{equation}\label{defzerononsplit} \tilde O_0 := \sum_{\kappa\in \check T^{\Gamma}} \tilde O_{0,\kappa}.
\end{equation}
\end{definition}
The sum on the right consists, of course, of only two terms.

For the split case, we notice that at almost every place we have $\tilde O_{0_v}(f_{\mathcal X_v}^0)=(\log q_v)^{-1} = $ the leading term of the local Dedekind zeta function $\zeta_v(t)$ at $t=0$, while $\tilde O_{u_v}(f_{\mathcal X_v}^0) = 1 = $ twice the constant coefficient in the Laurent expansion of the local zeta function at $s=0$. We will define a global irregular distribution $f\mapsto \tilde O_0(f)$ which, formally, is the limit with $t\to 0$ of a sum of Euler products with opposite residues, the first of which is outside of a finite number of places equal to the partial Dedekind zeta function $\zeta^S(t)$ and the second of which is equal to $\zeta^S(-t)$.

More precisely, take $S$ a sufficiently large set of places, and let $a_{-1}^S, a_0^S$ be the leading, resp.\ the second, coefficient in the Laurent expansion of the partial zeta function $\zeta^S(t)$ around $t=0$; notice that the order of zero of this partial zeta function is $|S|-1$. Then we have the following:

\begin{definition}
In the split case we define the distribution $\tilde O_0$ as
\begin{equation}\label{defzerosplit} \tilde O_0(\prod_v f_v) = 2a_0^S \prod_{v\in S} \tilde O_{0_v}(f_v) + a_{-1}^S \prod_{v\in S} \tilde O_{0_v} (f_v) \cdot \left(\sum_{v\in S} \frac{\tilde O_{u_v}(f_v)}{\tilde O_{0_v} (f_v)}\right).
\end{equation}
\end{definition}

We are now ready to state the Poisson summation formula in the baby case:

\begin{theorem}\label{Poissonbaby}
 For $f\in \mathcal S(\mathcal X(\adele))$, in either the split or the nonsplit case, we have
\begin{equation}\label{eqPoissonbaby}
 \sum_{\xi\ne 0} f(\xi) + \tilde O_0(f) 
 = \sum_{\xi\ne 0} \mathcal Gf(\xi) + \tilde O_0(\mathcal G f).
\end{equation}
\end{theorem}

\begin{remark}
Although the Poisson summation formula for $X=\Res_{E/k}\Ga$ can be stated, equivalently, at the finite level (that is, for functions on $\prod_{v\in S} E_v$, for some finite set $S$ of places), the Poisson summation formula of Theorem \ref{Poissonbaby} for $\mathcal X = X/T$ is genuinely adelic, in the following sense: The sums over $\xi \in \mathcal B^\reg(k)$ can be restricted to $\xi\in \mathcal B^\reg(k)\cap \mathcal B(k_S)$, for a sufficiently large set of places $S$, since the support of the basic function is $\mathcal B(\mathfrak o_v)$; however, for a point $\xi$ which is not in $\mathcal B^\reg(k_S)$ the functions appearing in the sum depend on the coordinates of $\xi$ outside of $S$, since the basic function is not identically equal to $1$ on $\mathcal B(\mathfrak o_v)$. If $f$ is compactly supported then we can enlarge $S$ so that all $\xi \in \mathcal B^\reg(k)$ in its support belong to $\mathcal B^\reg(k_S)$; however, this is not possible simultaneously for both sides of the 
formula. This comment is valid for all Poisson 
summation formulas that we will encounter in this paper; in fact, for one of them (the Kuznetsov trace formula with non-standard test functions) we will not even be able to restrict to $\mathcal B^\reg(k)\cap \mathcal B(k_S)$.
\end{remark}

\subsection{Indirect proof of Poisson summation in the nonsplit case.}\label{indirect-nonsplit}

We will first deduce the Poisson summation of Theorem \ref{Poissonbaby} from the Poisson summation formula for Fourier transform on the adeles of $E$. In the next section, we will prove it directly using the explicit expression \eqref{G} for $\mathcal G$, building the tools that we will need for the comparison of RTFs in later sections.

 First of all, recall from \cite[\S 2.10]{SaBE1} that, locally, elements of $\mathcal S(\mathcal X_v)$ are obtained as orbital integrals of elements in $\oplus_{\alpha} \mathcal S(X^\alpha)$; here $X=\Res_{E/k}\Ga$; $\alpha$ ranges over isomorphism classes of $T$-torsors over $k_v$ (parametrized by $H^1(k_v,T)$), and for such a torsor $R^{\alpha}$ we set $X^{\alpha}=X\times^T R^{\alpha}$. Since $\Aut_T(R^{\alpha})\simeq T$, the space $X^{\alpha}$ carries a $T$-action; it is non-canonically isomorphic to $X$.
Given an element $f\in \mathcal S(\mathcal X_v)$ we will call ``lift'' of $f$ an element  $\Phi\in \bigoplus_\alpha \mathcal S(X_v^\alpha)$, together with a Haar measure $dt_v$ on $T_v$ such that the orbital integrals of $\Phi$ give $f$, see \cite[2.10]{SaBE1}. 

The sums over $\xi$ in \eqref{eqPoissonbaby} are, in fact, over $\xi\in k_S\smallsetminus\{0\}$ for some finite number of places $S$, since $f$ will be equal to the basic vector outside of $S$, which is supported only on the image of $X(\mathfrak o_v)$. Thus, the sums will be only over those $\xi$ whose preimage under $X^\reg\to \mathcal B^\reg$ is a $T$-torsor (over $k$) that is trivial outside of $S$. Since there is only a finite number of such torsors, let us fix, for now, such an isomorphism class of $T$-torsors over $k$, denoted by $\beta$,\footnote{We will consistently be using $\beta$ for torsors defined over $k$ and $\beta_v$ for their localizations, while the symbols $\alpha$, $\alpha_v$ will be reserved for torsors defined locally, not necessarily as restrictions of global torsors.} and let us assume that
$$ f= \prod_{v\notin S} f_{\mathcal X_v}^0 \cdot \prod_{v\in S} f_v$$
(since by continuity it is enough to prove the Poisson summation for a dense subspace), where $f_v$, for all $v$, $f_v$ is in the image of $\mathcal S(X^\beta_v)$. Fix lifts $(\Phi_v, dt_v)$ so that $\prod_v dt_v$ is a factorization of the Tamagawa measure on $T(\adele)$, and $\Phi_v$ is the characteristic function of $\mathfrak o_{E_v}$ for $v\notin S$. (Hence, $dt_v(T(o_v))=1$ for $v\notin S$.)

In \cite[(2.25)]{SaBE1}, I defined an extension of Fourier transform to the space $X^\beta$; I claim that $\Phi:=\prod_{v\in S} \Phi_v$ satisfies the usual Poisson summation formula with respect to this Fourier transform, that is:
\begin{equation}\label{grr}
 \sum_{\xi\in X^\beta(k_S)} \Phi(a\xi) = \frac{1}{|a|} \sum_{\xi\in X^\beta(k_S)} \hat \Phi\left(\frac{\xi}{a}\right).
\end{equation}

Indeed, we can fix, over $k_S$, an isomorphism $\iota: X\to X^\beta$ mapping $1$ to some element $e\in X^\beta(k_S)$, and let $\Phi^1_v= \iota^* \Phi_v$, a function on $E_v$. Recall the formula \cite[(2.27)]{SaBE1}:
$$\iota^* \widehat{\Phi_v}(y)  = |a|_v \widehat{\Phi^1_v} (ay),$$
where $a = N_k^E(e) \in k_S$. Taking into account that $\prod_{v\in S} |a|_v =1$, and that $\iota$ induces a bijection between $X(k_S)$ and $X^\beta(k_S)$, the claim follows from Poisson summation on $\mathbb A_E$.

Now denote by $0_\beta\in X^\beta (k)$ the ``zero'' point in $X^\beta$. For an arbitrary $f = \otimes_v f_v$ (not necessarily supported on the $X^\beta$ corresponding to a single torsor), with a lift $\Phi \in \otimes^\prime_v \oplus_{\alpha\in H^1(k_v,T)} \mathcal S(X^\alpha_v)$ (with respect to the Tamagawa measure on $T(\adele)$), integrating \eqref{grr} over $[T]$ we get:
\begin{equation}\label{comp1}
\sum_{\xi\ne 0} f(\xi) +\Vol([T]) \sum_{\beta\in H^1(k, T)} \Phi(0_\beta) = \sum_{\xi\ne 0} \mathcal Gf(\xi) + \Vol([T]) \sum_{\beta\in H^1(k, T)} \hat \Phi(0_\beta).
\end{equation}

By the short exact sequence (\ref{Hasseprinciple}) and abelian Fourier analysis we have
$$\Vol([T])\sum_{\beta\in H^1(k,T)} \Phi(0_\beta)= $$ 
$$ =\frac{\left|\ker^1(k,T)\right|}{\left|\check T^\Gamma\right|} \Vol([T])\sum_{\kappa\in \check T^\Gamma} \prod_v \sum_{\alpha_v\in H^1(k_v,T)} \left<\alpha_v,\kappa\right> \Phi_v(0_{\alpha_v}).$$

In the notation that we introduced previously, we have, by definition,
$$\tilde O_{0_v,\kappa}(f_v) =  \frac{\AvgVol(T_v)}{|H^1(k_v,T)|} \sum_{\alpha_v\in H^1(k_v,T)} \left<\alpha_v,\kappa\right> \Phi_v(0_{\alpha_v}),$$
and therefore there remains to show:
\begin{equation}\label{torusvolume}
\frac{\left|\ker^1(k,T)\right|}{\left|\check T^\Gamma\right|} \Vol([T]) \prod_v' \frac{|H^1(k_v,T)|}{\AvgVol(T_v)} = 1, 
\end{equation}
where the Euler product should be interpreted in an analogous way as in the definition of the global distribution $\tilde O_0$, and is therefore denoted by $\prod^\prime$. More precisely, if we fix a factorization of the Tamagawa measure on $T(\adele)$ as $\prod_v dt_v$ with $dt_v(T(\mathfrak o_v))=c_v$ for $v\notin S$ so that the Euler product $\prod_{v\notin S} c_v$ is convergent (for example, $c_v=1$ as before), then for $v\notin S$ we have
$$ \frac{|H^1(k_v,T)|}{\AvgVol(T_v)} =\frac{1}{c_v\cdot  (L_v(\eta_v,0))^*},$$
where, I remind, the exponent $^*$ denotes the leading term in the Laurent expansion. Thus, to interpret the formal Euler product $\prod^\prime$, we set
$$ \prod^\prime_v \frac{|H^1(k_v,T)|}{\AvgVol(T_v)} = \frac{1}{\left(L^S(\eta,0)\right)^*\prod_{v\notin S} c_v} \cdot \prod_{v\in S} \frac{|H^1(k_v,T)|}{\AvgVol(T_v)}.$$

We could deduce \eqref{torusvolume} directly from known formulas about Tamagawa measures of tori, i.e., essentially the Dirichlet class number formula, but since the formula does not really depend on choices of measures and explicit calculations, let us instead sketch its reduction to the ``basics'' (more precisely, to Tate's thesis) using only the minimum of necessary arguments:

For \emph{any} tori $T_i$ over $k$ with a short exact sequence
$$1\to T_1\to T_2\to T_3\to 1,$$
and any compatible choice of Haar measures on their adelic points (in the obvious sense, i.e., integration over $T_2$ should be equal to integration over $T_1$ followed by integration over $T_3$), it is known that the quantities
$$\mu(T_i):= \frac{\left|\ker^1(k,T_i)\right|}{\left|\check T_i^\Gamma\right|} \AvgVol([T_i])  $$
satisfy
$$ \mu(T_2)=\mu(T_1)\mu(T_3);$$
see \cite[Theorem 4.4.1]{Ono-Tamagawa}.

Similarly, for a similar sequence over a completion $k_v$ we have a long exact sequence
$$ 1\to T_1(k_v) \to T_2(k_v) \to T_3(k_v) \to H^1(k_v,T_1) \xrightarrow{\iota} H^1(k_v,T_2) \to \dots,$$
from which it follows that
\begin{equation}\label{green} \frac{\AvgVol(T_1(k_v))}{|\ker(\iota)|} = \frac{\AvgVol(T_2(k_v))}{\AvgVol(T_3(k_v))}.
\end{equation}
Notice that $|\ker(\iota)| = (k_v^\times: N_{k_v}^{E_v} E_v^\times)$.

Applying these considerations to our case:
\begin{equation}\label{torisequence} 1\to T \to \Res_{E/k} \Gm \to \Gm\to 1,
\end{equation}
the statement (\ref{torusvolume}) reduces to the statement
\begin{equation}
 \frac{\AvgVol([\Res_{E/k}\Gm])}{\AvgVol([\Gm])} = \frac{(\zeta^S_E(0))^*}{(\zeta^S_k(0))^*} \prod_{v\in S} \frac{\AvgVol(E_v^\times)}{\AvgVol(k_v^\times)}.
\end{equation}
The statement is now a corollary of Tate's thesis, more precisely of the fact that the residue at $s=0$ of a zeta integral
$$ \int_{\adele^\times} \Phi(x) |x|^s d^\times x$$
is equal to $\Phi(0) \AvgVol([\Gm])$ (and similarly for Tate integrals for $E$). 

\qed

\begin{remark}
 Tate's thesis shows that
 \begin{equation} \AvgVol{[\Gm]} = \mbox{``Res''} \prod_v \AvgVol(k_v^\times),
 \end{equation}
 where the expression on the right should not be taken literally but interpreted, again, as an expression of the form: $(\zeta_k^S(0))^*$ times a convergent product. Indeed, the residue of the zeta integral of a Schwartz function $\Phi$ at $s=0$ is on one hand equal to $\AvgVol{[\Gm]} \Phi(0)$, and on the other hand the ``Euler product'' of the right hand side times $\Phi(0)$.
 
 Further observing that, for a self-dual measure $dx_v$ with respect to an additive character $\psi_v^{-1}$ used to define Fourier transform, using the corresponding multiplicative measure 
\begin{equation}\label{multmeasure}
 d^\times x_v = |x_v|^{-1}dx_v
\end{equation}
 we have 
 $$\Phi_v(0) \AvgVol(k_v^\times) =  \Res_{s=0} \int_{k_v^\times} \Phi_v(x) |x|^s d^\times x = $$
 $$ = \Res_{s=0} \gamma_v(1, 1-s,\psi_v^{-1}) \int_{k_v^\times} \hat\Phi_v(x) |x|^{1-s} d^\times x = \Phi(0) \Res_{s=0} \gamma_v(1,1-s,\psi_v^{-1}),$$
where $\gamma_v (1,s,\psi_v^{-1})$ denotes the gamma factor for the trivial multiplicative character $\chi=1$, we get that for such a measure and character we have
 \begin{equation} \label{green2} \AvgVol(k_v^\times) = \Res_{s=0} \gamma_v(1, 1-s,\psi_v^{-1}) = \Res_{s=0} \gamma_v(1, s,\psi_v)^{-1}.
 \end{equation}

Globally, the triviality of gamma factors shows that $\AvgVol{[\Gm]}=1$ under the usual normalization for Tamagawa measure -- i.e.,, multiplication by convergence factors corresponding to $\zeta_{\QQ_p}(1)$ at the finite places over $\QQ_p$. This, of course, is just a reformulation of the usual derivation of the class number formula from Tate's thesis.

 Combining \eqref{green2} (applied to both $k_v^\times$ and $E_v^\times$) with \eqref{green}, we get
 \begin{equation}\label{AvgVol}
  \gamma_v^*(\eta_v, 0,\psi_v)\frac{\AvgVol(T(k_v))}{(k_v^\times: N_{k_v}^{E_v} E_v^\times)} =1.
 \end{equation}
Here, again, $^*$ denotes the dominant term in the Laurent expansion when $0$ is replaced by $s$. This formula is identical to \eqref{green2} when $E_v/k_v$ is split, and when $E_v$ is a quadratic field, the factor $\gamma_v^*(\eta_v, 0,\psi_v) = \gamma_v(\eta_v, 0,\psi_v)$ arises as the quotient 
$$ \frac{\Res_{s=0} \gamma_{k_v}(1, s,\psi_v)^{-1}}{\Res_{s=0} \gamma_{E_v}(1, s,\psi_v \circ \tr)^{-1}}$$
(where we now indicate with an index the algebra to which each gamma factor is attached). Here the measures on $k_v$ and $E_v$ have been taken to be self-dual with respect to the characters $\psi_v$ and $\psi_v\circ \tr$, respectively, the measures on $k_v^\times$ and $E_v^\times$ are given by \eqref{multmeasure}, and the measure on $T(k_v)$ is determined by compatibility with respect to the sequence \eqref{torisequence}. 

Global triviality of gamma factors, applied to \eqref{AvgVol}, shows that for Tamagawa measures
\begin{equation}\prod^\prime_v \frac{\AvgVol(T_v)}{|H^1(k_v,T)|} = 1.
\end{equation}

\end{remark}

\subsection{Indirect proof of Poisson summation in the split case}

Now we consider the case $E = k\oplus k$. Again we let $f=\prod_v f_v$ and fix preimages $\Phi_v\in \mathcal S(X_v)$ of $f_v$ (with compatible measures on the torus), almost always equal to the characteristic function of $\mathfrak o_v^2$. The left and right sides of \eqref{eqPoissonbaby} are continuous functionals on the LF-space $\mathcal S(\mathcal X(\adele))$, and the transfer operator $\mathcal G$ is continuous on it; therefore, we might prove equality just for a dense subspace. We will assume, therefore, that for all $v$ we have: $\Phi_v (x,y)= \Phi_{1,v}(x)\Phi_{2,v}(y)$ for some Schwartz functions $\Phi_{1,v}, \Phi_{2,v}$ in one variable (equal to the characteristic function of $\mathfrak o_v$ almost everywhere).

The function $\Phi:=\prod_v \Phi_v$ satisfies the usual Poisson summation formula with respect to Fourier transform, that is:
\begin{equation}
 \sum_{\xi\in X(k)} \Phi(a\xi) = \frac{1}{|a|}\sum_{\xi\in X(k)} \hat \Phi\left(\frac{\xi}{a}\right),
\end{equation}
where $|a| = |a_1\cdot a_2|$ for $a= (a_1,a_2)$.

Integrating over $[T]=\adele^\times/k^\times$ we get
$$\sum_{\xi\in k^\times} f(\xi) - \sum_{\xi\in k^\times} \mathcal Gf(\xi) = \int_{\adele^\times/k^\times} \left( \sum_{\gamma\in (k^\times)^2} \Phi(a\gamma) - \sum_{\gamma\in (k^\times)^2} \hat\Phi\left(\frac{\gamma}{a}\right) \right) = $$
\begin{eqnarray}=\int_{\adele^\times/k^\times} \left( - \sum_{\gamma_1\in k^\times} \Phi_1(a\gamma_1) \Phi_2(0) - \sum_{\gamma_2\in k} \Phi_1(0) \Phi_2(a^{-1}\gamma_2) +\right. \nonumber \\
\left. +\sum_{\gamma_1\in k} \hat\Phi_1(a^{-1}\gamma_1)\hat\Phi_2(0) + \sum_{\gamma_2\in k^\times} \hat\Phi_1(0) \hat\Phi_2(a\gamma_2)\right) da.
\end{eqnarray}
 Notice that the integration over $[T]$ giving rise to the left hand side is absolutely convergent, and hence so is the right.

By one-dimensional Poisson summation, the right hand side can be written
$$\int_{\adele^\times/k^\times} \left( - \sum_{\gamma_1\in k^\times} \Phi_1(a\gamma_1) \Phi_2(0) - \sum_{\gamma_2\in k} |a|\Phi_1(0) \hat\Phi_2(a\gamma_2) +\right.$$
$$\left. +\sum_{\gamma_1\in k} |a|\Phi_1(a\gamma_1)\hat\Phi_2(0) + \sum_{\gamma_2\in k^\times} \hat\Phi_1(0) \hat\Phi_2(a\gamma_2)\right) da.
$$

Notice that the term corresponding to $\gamma_2=0$ in the second sum cancels the term with $\gamma_1=0$ in the third. By interpreting the remaining integrals as Tate integrals, we get
$$ \lim_{s\to 0} \left( - \zeta(\Phi_1,s)\Phi_2(0) - \Phi_1(0) \zeta(\hat\Phi_2,s+1) + \hat\Phi_2(0) \zeta(\Phi_1,s+1) + \hat\Phi_1(0) \zeta(\hat\Phi_2,s)\right) = $$
\begin{equation} \label{difference}
 = \lim_{s\to 0} \left( - \zeta(\Phi|_x,s) - \zeta(\Phi|_y,-s) + \zeta(\hat\Phi|_x,-s) +  \zeta(\hat\Phi|_y,s)\right), 
\end{equation}
where $\Phi|_x$ and $\Phi|_y$ denote, respectively, the restrictions to $y=0$ and $x=0$, considered as functions of the variable $x$, resp.\ $y$. The last step is by the functional equation of Tate integrals.

 Then we claim:

\begin{lemma}\label{irregularsplit}
Let $L$ denote the functional on $\mathcal S(X(\adele))$:
$$L(\Phi)=  \lim_{t\to 0} \left( \zeta(\Phi|_x,t) + \zeta(\Phi|_y,-t)\right),$$
then
\begin{eqnarray}
L(\Phi) =  \tilde O_0(f),
\end{eqnarray}
the ``irregular'' distribution of \eqref{defzerosplit}.
\end{lemma}

\begin{proof} 
One just needs to check the definitions. Fixing a sufficiently large finite set of places $S$, we have

$$ \zeta(\Phi|_x,t) = \frac{a_{-1}}{t} + a_0 + \mbox{higher order terms},$$
with
$$a_{-1} = a_{-1}^S \cdot \prod_{v\in S} \tilde O_{0_v}(f_v)$$ and $$a_0 = \prod_{v\in S} \tilde O_{0_v} (f_v) \left( a_0^S+  a_{-1}^S \cdot \sum_{v\in S} \frac{a_{0,v}}{\tilde O_{0_v} (f_v)}\right),$$ where $a_{0,v} = $ the constant term in the Laurent expansion of the local zeta integral $\zeta(\Phi|_x,t)$ at $t=0$. Similarly for $\zeta(\Phi|_y,t)$, with $a_{0,v}$ replaced by the constant term $b_{0,v}$ of the Laurent expansion of the local zeta integral $\zeta(\Phi|_y,t)$ at $t=0$. From this it follows that
$$ \lim_{t\to 0} \left( \zeta(\Phi|_x,t) + \zeta(\Phi|_y,-t)\right)= \prod_{v\in S} \tilde O_{0_v} (f_v) \left( 2a_0^S+  a_{-1}^S \cdot \sum_{v\in S} \frac{a_{0,v}+b_{0,v}}{\tilde O_{0_v} (f_v)}\right),$$
and keeping in mind that $a_{0,v}+b_{0,v} = \tilde O_{u_v}(f_v)$ by \cite[2.10]{SaBE1}, this is equal to $\tilde O_0(f)$.
\end{proof}
 
Hence, (\ref{difference}) can be written as
$$ -\tilde O_0(f) + \tilde O_0(\mathcal Gf),$$
and this implies the Poisson summation formula (\ref{eqPoissonbaby}).

\qed

\section{Direct proof in the baby case} \label{sec:direct-baby}

\subsection{Motivation} The proof of the Poisson summation formula for $\mathcal X = \Res_{E/k} \Ga/T$ which was presented in the previous section is unsatisfactory for two reasons: first, it is not a direct proof on the Schwartz space $\mathcal S(\mathcal X(\adele))$, but it uses properties of the space ``upstairs'' $\mathcal S(\mathbb A_E)$. For the more complicated Poisson summation formulas that one will encounter, here and elsewhere, using properties of the space upstairs is precisely what one would like to avoid  -- in fact, we would like a direct proof at the level of the base $\mathcal B$, in order to deduce properties of the space upstairs. The fact that our transfer operator $\mathcal G$ is given in terms of Fourier transforms and birational maps suggests that such a direct proof should be possible, using at some point the classical Poisson summation formula for Fourier transform.

Hence, in the present section we will discuss a direct proof of the Poisson summation formula that was proven in the previous section, using only the given spaces of orbital integrals and not the fact that they arise as coinvariants of the usual Schwartz space on a vector space.

The Poisson summation formula will be proven by a method of analytic continuation as the given ``Schwartz spaces'' vary according to a complex parameter $t$. Instead of directly focusing on the specific case of interest, we present an axiomatic approach to these Schwartz spaces in order to single out the properties that we are using in the proof. With some modifications, this approach will allow us to prove the Poisson summation formula for the comparison of relative trace formulas in the next section.

\subsection{Local Schwartz spaces varying with a parameter} \label{ssbabySchwartz}

Throughout our discussion there will be a complex parameter $t$. We will say ``for large $t$'' (denoted: $\Re(t)\gg 0$) for statements that hold on a half-plane of the form $\Re(t)\ge \sigma$. For the discussion of non-Archimedean places with residual degree $q_v$, the parameter $t$ is considered to be varying in $\CC/\frac{2\pi i}{\log q_v} \Z$.

We will give axioms for four different ``Schwartz spaces'' $\mathcal S^t_1,\mathcal S^t_2,\mathcal S^t_3, \mathcal S^t_4$, depending on the parameter $t$, and will define certain integral transforms between them, again depending on $t$. In fact, for the baby case the first and fourth spaces, as well as the second and third spaces, will be identical, but in order to get used to the scheme, let us keep them in mind as different spaces. These spaces will be restricted tensor products, over all places, of local Fr\'echet spaces varying ``analytically'' as described in Appendix \ref{sec:families}. We will start by describing axioms for their local factors, and later (in \S \ref{basicbaby}) we will add some axioms on their basic vectors.

The (local) Schwartz spaces $\mathcal S^t_{1,v}$ and $\mathcal S^t_{4,v}$ are sections of a Schwartz cosheaf over $\mathcal B(k_v)=k_v$, which away from $0\in \mathcal B$ coincides with the cosheaf of usual Schwartz functions. (In particular, we have rapid decay at $\infty$.) We now describe the behavior close to $\xi=0\in \mathcal B(k_v)$.

As with the ``model $\mathcal A^t$'' of Appendix \ref{sec:families}, we will define the stalks of $\mathcal S^t_{1,v}$ and $\mathcal S^t_{4,v}$ over $0$ in such a way that the \emph{fibers} are annihilated by the operator
\begin{equation}
 (\Id - \eta_v(a) |a|_v^{-t-\frac{1}{2}} a\cdot  )(\Id - |a|_v^{-\frac{1}{2}} a\cdot ),
\end{equation} for every $a\in k_v^\times$, where $a\cdot$ we denote the normalized action of $k_v^\times$ on functions on $\mathcal B$
\begin{equation}\label{unitaryaction}
 (a\cdot f)(x) := |a|_v^\frac{1}{2} f(ax).
\end{equation}

The annihilator of the fibers does not, of course, provide a complete description in the Archimedean case. The precise definition, for generic $t$, is that the elements of $\mathcal S_{1,v}^t$ and $\mathcal S_{4,v}^t$ are, in a neighborhood of zero,  smooth functions on $k_v^\times$ which have the form
$$C_1(\xi) +C_2(\xi) \eta_v(\xi) |\xi|_v^t,$$
where $C_1$ and $C_2$ extend to smooth functions in a neighborhood of zero, except:
\begin{itemize}
 \item when $t = 0$ and $\eta_v=1$, in which case the functions have the form $$C_1(\xi)+C_2(\xi)\log|\xi|_v,$$ i.e., they specialize to elements of $\mathcal S(\mathcal X_v)$;
 \item when $t\in 2\Z$, $\eta_v=1$ and $k_v=\RR$, or $t\in (2\Z+1)$, $\eta_v\ne 1$ and $k_v  = \RR$, or $t\in \Z$ and $k_v  = \CC$ in which case the functions have the form
$$\begin{cases}
   C_1(\xi)+C_2(\xi)\eta_v(\xi)|\xi|_v^t\log|\xi|_v, & \mbox{ when } t\ge 0;\\
   C_1(\xi)\eta_v(\xi)|\xi|_v^t+C_2(\xi)\log|\xi|_v, & \mbox{ when } t<0.
\end{cases}$$
\end{itemize}

The (local) Schwartz spaces $\mathcal S^t_{2,v}$ and $\mathcal S^t_{3,v}$ consist of sections of the cosheaf over $\PP^1(k_v)$ of functions $\mathcal B = k_v$ which
away from $\infty$ coincides with the cosheaf of Schwartz functions, and 
in a neighborhood of infinity the functions have the form
$$C(\frac{1}{\xi})\cdot \eta_v(\xi) |\xi|_v^{-t-1},$$ for some smooth function $C$.

Clearly, these spaces are stable under the involution
\begin{equation}\label{defiotat}
\iota_t: f\mapsto \frac{\eta_v(\bullet)}{|\bullet|_v^{t+1}}f\left(\frac{1}{\bullet}\right). 
\end{equation}

Appendix \ref{sec:families} includes a long discussion of these spaces, including their topology, their analytic structure as $t$ varies, and a notion of ``polynomial families of seminorms'' (as $t$ varies in bounded vertical strips). We summarize the results that we need, noting first that there is some exceptional behavior at certain values of $t$ \eqref{exceptions}, which will not be included in the result that follows. These exceptions have to do with poles of local gamma factors, where Fourier transforms of characters on $k_v^\times$ are not characters on $k_v^\times$.

\begin{proposition}\label{variousproperties}
 For $t$ different than the values of \eqref{exceptions}, Fourier transform $\mathcal F$ induces an isomorphism between the Fr\'echet spaces $\mathcal S^t_{1,v}$ and $\mathcal S^t_{2,v}$ (or $\mathcal S^t_{3,v}$ and $\mathcal S^t_{4,v}$). 

The composition
\begin{equation}\label{defGt}\mathcal G_t: \mathcal F\circ \iota_t\circ\mathcal F: \mathcal S^t_1\to \mathcal S^t_4\end{equation}
makes sense, by analytic continuation, for \emph{every} $t$. 

Both $\mathcal F$ and $\mathcal G_t$ are bounded by polynomial families of seminorms on the corresponding spaces, as $t$ varies, and preserve analytic sections.

For $t$ outside of the values of \eqref{exceptions2}, \eqref{exceptions3}, analytic sections of $\mathcal S_{i,t}^t$ (where $i=1$ or $4$) are of the form:
$ \xi \mapsto C_1^t(\xi)+C_2^t(\xi) \eta_v(\xi) |\xi|_v^t$
where $t\mapsto C_1^t$, $t\mapsto C_2^t$ are strongly meromorphic sections into the Fr\'echet space $\mathcal S(k_v)$ of Schwartz functions on $k_v$. Such a section extends to $t=0$ iff $C_1^t$ and $C_2^t$ have simple poles with opposite residues at $t=0$, with the residue an element of $\mathcal S(k_v^\times)$ when $\eta_v\ne 1$, and in that case $C_1^t, C_2^t$ can be chosen to be holomorphic at $t=0$. \end{proposition}

Recall that a strongly meromorphic section into a Fr\'echet space is one which, in a neighborhood of any point $t_0$, becomes (weakly=strongly) holomorphic after multiplication by a power of $(t-t_0)$.

For the last statement of the proposition, notice that given an element of $\mathcal S_{i,t}^t$, $i=1,4$, the pair $(C_1^t, C_2^t)$ is only defined up to an element of $\mathcal S(k_v^\times)$, embedded as $f\mapsto (f, f \eta_v |\bullet|_v^{-t})$.  One can formulate similar statements about when a meromorphic section extends to a holomorphic one at the other exceptional values \eqref{exceptions2}, \eqref{exceptions3} which involve logarithms. I leave this description to the reader, as it will not be needed.

We can also relate the asymptotic constants of elements of these Schwartz spaces and their Fourier transforms:

\begin{lemma}\label{lemmad} 
Let $t\notin \Z$. If $f_1\in \mathcal S^t_{1,v}$ is equal to
$$ C_1(\xi)+C_2(\xi)\eta_v(\xi)|\xi|_v^t$$ in a neighborhood of $\xi=0$, with $C_1$ and $C_2$ smooth functions, and $\mathcal F f_1\in \mathcal S^t_{2,v}$ is of the form
$$D(\frac{1}{\xi})\cdot \eta_v(\xi) |\xi|_v^{-t-1}$$
in a neighborhood of infinity, then
\begin{equation}\label{dfromc}
 D(0)= \gamma(\eta_v, -t, \psi_v) \cdot C_2(0)
\end{equation}
(local abelian gamma-factors).

Similarly, if $f_2 \in \mathcal S^t_{2,v}$ and $\mathcal F\circ \iota_t (f_2) \in \mathcal S_{4,v}^t$ is of the form $E_1(\xi)+E_2(\xi)\eta_v(\xi)|\xi|_v^t$ in a neighborhood of $0$ then
\begin{equation}\label{dfrome}
 f_2(0) = \gamma(\eta_v, -t, \psi_v^{-1}) \cdot E_2(0).
\end{equation}
\end{lemma}

\begin{proof}
Recall that for (almost) every character $\chi$ of $k_v^\times$, considered as a tempered distribution on $k_v$ by meromorphic continuation according to Tate's thesis, we have a relation
\begin{equation}
 \widehat{\chi(\bullet)}= \gamma(\chi^{-1},0,\psi) \cdot |\bullet|\cdot \chi^{-1}(\bullet).
\end{equation}
(We omit the index $v$ for this proof.)
Indeed, this is just a reformulation of the functional equation for zeta integrals; in what follows, we denote the obvious \emph{bilinear} (not hermitian) pairing by angular brackets, and use the exponent $\psi$ when Fourier transform is taken with respect to the character $\psi$, instead of $\psi^{-1}$ which is our standard convention, 
$$\left<\phi,\widehat{\chi}\right> = \left<{\widehat{\hat\phi}}^\psi,\widehat{\chi}\right> = \left<\hat\phi,\chi\right>= Z(\hat \phi,\chi, 1) = $$
$$ =\gamma(\chi^{-1},0,\psi) Z(\phi,\chi^{-1},0) = \gamma(\chi^{-1},0,\psi) \left<\phi,\chi^{-1}(\bullet)\cdot |\bullet|^{-1}\right>.$$

To prove the desired relations between the asymptotic coefficients, it suffices to relate them for one element in the Schwartz space for which they are non-zero. We obtain such an element by multiplying $\chi(\xi)=\eta(\xi)|\xi|^t$ by the characteristic function of a neighborhood of the identity, thus smoothening its Fourier transform -- but leaving it invariant in a neighborhood of infinity. The claim now follows.
\end{proof}

\subsection{Basic vectors and global Schwartz spaces}\label{basicbaby}

We now assume that our local spaces $\mathcal S_{i,v}^t$ are endowed, for almost every (non-Archimedean) $v$, with analytic sections of ``basic vectors'' $f_{i,v}^{t,0}$, $t\in \CC/ \frac{2\pi i}{\log q_v} \Z$, which satisfy the following axioms:
\begin{enumerate}
 \item The value of $f_{i,v}^{t,0}$ on $\mathcal B^\reg(\mathfrak o_v) = \mathfrak o_v^\times$ is a constant $c_{i,v}^t$ such that the partial Euler product: $\prod_{v\notin S} c_{i,v}^t$ converges for $t$ large, locally uniformly in $\Re(t)$, and for $i=1,4$ admits analytic continuation to all values of $t$;
 \item $\mathcal F \left( f_{1,v}^{t,0} \right) = f_{2,v}^{t,0}$,\,\, $\iota_t \left( f_{2,v}^{t,0}\right) = f_{3,v}^{t,0}$ and $\mathcal F \left( f_{3,v}^{t,0} \right) = f_{4,v}^{t,0}$. 
 \item For $i=1,4$ and \emph{every} $t$, the basic functions $f_{i,v}^{t,0}$ are supported on $\mathcal B(\mathfrak o_v)\cap \mathcal B^\reg(k_v)$. In a neighborhood of $\xi=0$, for $t\ne 0$, they are equal to
$$ f_{i,v}^{t,0}(\xi) = L_v(\eta_v,t) + L_v(\eta_v,-t) \eta(\xi)|\xi|^t, $$
where $L_v(\eta_v,t)$ is the local Dirichlet $L$-function; notice that this extends analytically to $t=0$ (cf.\ Proposition \ref{variousproperties}). Finally, there is a constant $r_i^t\ge 0$, independent of $v$ and uniformly bounded in bounded vertical strips, such that the function
$$|\xi|^{r_i^t} \frac{|f_{i,v}^{t,0}(\xi)|}{|f_{i,v}^{t,0}(\mathfrak o_v^\times)|}$$
is $\le 1$. 
 \item For $i=2,3$ and for $\Re(t)\gg 0$ there is a constant $r_i\ge 0$, independent of $v$ or $t$, such that the function
$$|\xi|^{r_i} \frac{|f_{i,v}^{t,0}(\xi)|}{|f_{i,v}^{t,0}(\mathfrak o_v^\times)|}$$
is:
\begin{itemize}
 \item $\le 1$ for $|\xi|\le 1$;
 \item $\le |\xi|_v^{-M}$ for $|\xi|_v>1$, where $M$ is a prescribed large integer (depending on the global field $k$, s.\ the proof of Proposition \ref{babyconvergence}).
\end{itemize}
\end{enumerate}

In our application, we will actually have $f_{1,v}^{t,0}=f_{4,v}^{t,0}$ and $f_{2,v}^{t,0}=f_{3,v}^{t,0}$, but this will not be the case for the relative trace formula and therefore we consider them as different vectors, to fix ideas.

The relations between the basic vectors and the asymptotic behavior of $f_{1,v}^{t,0}$ and $f_{4,v}^{t,0}$ around $\xi=0$ also determine the asymptotic behavior of $f_{2,v}^{t,0}$ and $f_{3,v}^{t,0}$ around $\xi=0$ and $\xi=\infty$ by Lemma \ref{lemmad}. The following table summarizes the regular values and asymptotic behavior of basic vectors:

\begin{equation}\label{tablebaby}
 \begin{array}{|c|c|c|c|c|}
 \hline  
 i & f_{i,v}^{t,0} \mbox{ around } \xi=0 & f_{i,v}^{t,0} \mbox{ around } \xi=\infty & f_{i,v}^{t,0} (\mathfrak o_v^\times) \\
\hline
 1 \mbox{ or }4 & L_v(\eta_v,t) + L_v(\eta_v,-t) \eta_v(\xi)|\xi|_v^t & 0 & c^t_{i,v} \\
\hline
 2 \mbox{ or }3 & L_v(\eta_v,t+1) & L_v(\eta_v,t+1)\cdot \eta_v(\xi) |\xi|_v^{-t-1} & c^t_{i,v} \\
\hline
\end{array}
\end{equation}

The axioms allow us to make sense of the restricted, completed tensor products of local Schwartz spaces with respect to the basic vectors as functions on:
$$\mathcal B^\reg(\adele) = \adele^\times,$$
for every $t$ when $i=1$ or $4$, and for large $t$ when $i=2$ or $3$.
We will denote these global Schwartz spaces by $\mathcal S_i^t$, i.e., dropping the index $v$ from the local notation. The parameter $t$ now varies in $\CC$, in the number field case, and in $\CC/\frac{2\pi i}{\log q}\Z$, in the function field case (with base field $\FF_q$). Moreover, the axioms allow us to interpret Fourier transforms and the operators $\iota_t$ as isomorphisms between the global spaces:
\begin{eqnarray*} \mathcal F: \mathcal S_1^t &\xrightarrow{\sim}& \mathcal S_2^t \\
 \iota_t: \mathcal S_2^t&\xrightarrow{\sim}& \mathcal S_3^t \\
 \mathcal F: \mathcal S_3^t&\xrightarrow{\sim}& \mathcal S_4^t \\
 \mathcal G_t: \mathcal S_1^t&\xrightarrow{\sim}& \mathcal S_4^t.
\end{eqnarray*}

Recall that Fourier transform makes sense when $t$ (resp.\ $-t$) does not belong to the values \eqref{exceptions}, while $\mathcal G_t$ makes sense for \emph{every} $t$.

\subsection{Irregular distributions} \label{ssbabyirregular}

We define the functional $\tilde O_0$ on the global Schwartz spaces $\mathcal S^t_1$ and $\mathcal S_4^t$ which, formally, for $t\notin \Z$ assigns to an element $f=\otimes_v f_v$ with asymptotics:
$f_v (\xi) = C_{1,v}^t(\xi)+C_{2,v}^t(\xi)\eta_v(\xi)|\xi|^t$
the value
$$ C^t_1 +C^t_2 = \prod_{v} C_{1,v}^t(0) + \prod_{v} C_{2,v}^t(0).$$

The rigorous definition is as follows: 
\begin{equation}\label{zerocontrib}
 \tilde O_0(f):= L^S(\eta,t) \prod_{v\in S} C_{1,v}^t(0) + L^S(\eta,-t) \prod_{v\in S} C_{2,v}^t(0),
\end{equation}
where $S$ is large enough so that outside of $S$ we have $f_v = f_{i,v}^0$. It extends continuously to all elements of $\mathcal S_i^t$.

We similarly define functionals $\tilde O_0$ and $\tilde O_\infty$ on $\mathcal S^t_2$ and $\mathcal S^t_3$ when $t$ is large. Since there is no term of the form $\eta_v(\xi)|\xi|_v^t$ in a neighborhood of zero here, for an element $f= \otimes_v f_v \in \mathcal S^t_2$ or $\mathcal S^t_3$, we have
\begin{equation} \tilde O_0(f):= L^S(\eta,t+1) \prod_{v\in S} f_v(0),\end{equation}
while if $f_v(\xi) = D_v^t(\frac{1}{\xi}) \eta_v(\xi) |\xi|_v^{-t-1}$ for $\xi$ in a neighborhood of $\infty$, we have
\begin{equation}
 \tilde O_\infty(f) = L^S(\eta,t+1)\prod_{v\in S} D_v^t(0).
\end{equation}
The factor $L^S(\eta,t+1)$ has to do with the asympotic behavior of the basic function, cf.\ Table \eqref{tablebaby}.

We will now verify that for an analytic section of $\mathcal S_1^t$ or $\mathcal S_4^t$ the functional $\tilde O_0$ extends at $t=0$ to the ``irregular orbital integral'' $\tilde O_0$ defined in \S \ref{ssPoissonbaby}. 

\begin{proposition}\label{irregOI} 
For $i=1$ or $4$, and an analytic section $t\mapsto f^t \in \mathcal S_i^t$, the function $t\mapsto \tilde O_0(f^t)$ extends holomorphically to all $t$, and its value is bounded by polynomial seminorms on any bounded vertical strip; in particular, its value at some $t_0$ depends only on $f^{t_0}$ and not on the section. At $t=0$ it coincides with the functional denoted by $\tilde O_0(f^0)$ in \S \ref{ssPoissonbaby}.
\end{proposition}

\begin{proof}
 We start by proving the assertions for $t=0$. The issue is, of course, that as $t\to 0$ some of the local factors may blow up, according to Proposition \ref{variousproperties}.

 In the non-split case, the product $L^S(\eta,t) \prod_{v\in S} C_{1,v}^t(0)$ is holomorphic at $t=0$; indeed, this is the case for the full Dirichlet $L$-function $L(\eta,t)$, and the Euler factors of our product have at most the order of pole of the local $L$-factors, as follows from Proposition \ref{variousproperties}. Hence, the value of this expression at $t=0$ is equal to the product of the leading coefficients of its factors, which is precisely equal to
$$L^S(\eta,0)^* \cdot \prod_{v\in S} \tilde O_{0_v,+}(f^0_v),$$ where $L^S(\eta,0)^*$ denotes the leading term of $L^S(\eta,t)$ at $t=0$. Similarly for the other term of \eqref{zerocontrib}; the sum of the two terms coincides with the definition of $\tilde O_0$ in (\ref{defzerononsplit}). 

Regarding the bound by polynomial seminorms: If the order of zero of $L^S(\eta,t)$ at $t=0$ is $r$, then both the functions $L^S(\eta,-t)t^{-r}$ and $t^r\prod_{v\in S} C_{2,v}^t(0)$ are holomorphic in a vertical strip around zero; the first is of polynomial growth by standard properties of abelian $L$-functions, and the second is bounded by polynomial seminorms on $\otimes_{v\in S} \mathcal S_{i,v}^t$ by the definition of those in Appendix \ref{sec:families}: namely, if we take Fourier transforms of the local factors  $f^t_v$, then by Lemma \ref{lemmad} those will be of the form $D_v^t(\frac{1}{\xi}) \eta_v(\xi) |\xi|_v^{-t-1}$ in a neighborhood of infinity, with $D_v^t(0) = \gamma(\eta, -t, \psi) \cdot C_{2,v}^t(0)$. Recall that the factor $D_v^t(0)$ is \emph{by definition} bounded by polynomial seminorms, hence so is the product $t^r\prod_{v\in S} C_{2,v}^t(0)$. To prove that the product $t^r \prod_{v\in S}C_{1,v}^t(f^t_v)$ is bounded by polynomial seminorms, we recall, again from the appendix, that 
multiplication by $\eta_v(\bullet)|\bullet|_v^{-t}$ defines an isomorphism between the spaces $\mathcal S_{i,v}^t$ and $\mathcal S_{i,v}^{-t}$ which preserves the structures of polynomial seminorms, so this reduces the problem to the previous case.

 In the split case, $\zeta^S(t)$ has a zero of order $|S|-1$ at $t=0$, while the factors $C_{1,v}^t(0), \prod_{v\in S} C_{2,v}^t(0)$ each have a simple pole (at most) with opposite residues. Thus, the residue of
$$c_1(t):=\zeta^S(t)\prod_{v\in S} C_{1,v}^t(0)$$ is opposite to the residue of
$$c_2(t):=\zeta^S(-t)\prod_{v\in S} C_{2,v}^t(0),$$ and the sum of the two terms is regular at $t=0$. The proof of boundedness by polynomial seminorms is similar to the nonsplit case and is left to the reader. We now verify that the extension of the functional to $t=0$ coincides with that of \S \ref{ssPoissonbaby}. 

 In what follows, we set $A_v(t)=t C_{1,v}^t(0)$ and $B_v(t) = tC_{2,v}^t$; then $A_v(0)=-B_v(0) = \tilde O_0(f^0_v)$, and $A_v'(0)+B_v'(0)=\tilde O_u(f^0_v)$. We denote by $\zeta^S(0)^*$ the leading term of $\zeta^S(t)$ at $t=0$, and we write $-0$ instead of $0$ to signify that we are replacing $t$ by $-t$. 
$$\lim_{t\to 0}(c_1(t)+c_2(t))= \lim_{t\to 0} \left[(tc_1(t))\cdot \frac{1}{t}\left(1+\frac{c_1(t)}{c_2(t)}\right)\right]=$$
$$ =\zeta^S(0)^*\prod_{v\in S}\tilde O_0(f^0_v)\cdot \frac{d}{dt}\left.\frac{c_2(t)}{c_1(t)}\right|_{t=0}=$$
$$ =\zeta^S(0)^*\prod_{v\in S}\tilde O_0(f^0_v)\cdot \frac{c_2(0)}{c_1(0)} \cdot \left(\partial\log\frac{\zeta^S(-0)}{\zeta^S(0)}+\sum_{v\in S}\partial\log \frac{c_{2,v}(0)}{c_{1,v}(0)}\right) = $$
$$ = \zeta^S(0)^*\prod_{v\in S}\tilde O_0(f^0_v)\cdot \left(\partial\log\frac{\zeta^S(0)}{\zeta^S(-0)}+\sum_{v\in S} \left(\frac{A_v'(0)}{A_v(0)}-\frac{B_v'(0)}{B_v(0)}\right)\right) = $$
$$= \zeta^S(0)^* \prod_{v\in S} \tilde O_0(f^0_v) \cdot \left( \partial\log \frac{\zeta^S(0)}{\zeta^S(-0)}+ \sum_{v\in S} \frac{\tilde O_u(f^0_v)}{\tilde O_0(f^0_v)}\right).$$

This is precisely the term $\tilde O_0(f^0)$ of \eqref{defzerosplit}, which completes the proofs for $t=0$.

I leave the proof for other integer values of $t$ to the reader. I remark that, for example, when $k_v=\RR$ and $\eta_v=1$ the limit of $C_{2,v}^t(0)$ as $t$ approaches a positive even integer may be infinite, but this coincides with a trivial zero of the partial $L$-function $L^S(\eta,-t)$ representing the formal product: $\prod_{v\notin S} C_{2,v}^t(0)$.

\end{proof}

\subsection{The Poisson sum}

We define the following functionals on the global Schwartz spaces $\mathcal S^t_i$, all denoted by $PS$ for ``Poisson sum''. When $i=1$ or $4$:
\begin{equation}
 PS_i: \mathcal S^t_i\ni f \mapsto \tilde O_0(f) + \sum_{\xi\in k^\times} f(\xi),
\end{equation}

When $i=2$ or $3$ and $t$ is large:
\begin{equation}
 PS_i: \mathcal S^t_i\ni f \mapsto \tilde O_0(f) + \tilde O_\infty(f) + \sum_{\xi\in k^\times} f(\xi).
\end{equation}

The following is immediate (assuming convergence, which will be proved right afterwards):
\begin{lemma}\label{iotainvariant}
 Consider the map $\iota_t: \mathcal S^t_2 \xrightarrow{\sim} \mathcal S^t_3$ defined in \eqref{defiotat}. It preserves Poisson sums, i.e., the pull-back of $PS_3$ via this map is the functional $PS_2$.
\end{lemma}

Now we discuss convergence:

\begin{proposition}\label{babyconvergence}
The functional $\mathcal S_i^t\ni f^t\mapsto \sum_{\xi\in k^\times} f^t(\xi)$ converges absolutely for every $t$ when $i=1,4$, and for $\Re(t)\gg 0$ when $i=2,3$. For such values of $t$, the sum is bounded in vertical strips by polynomial seminorms on the spaces $S^t_i$; in particular, for an analytic section $t\mapsto f^t\in \mathcal S_i^t$ the value of the functional is analytic in $t$.  

Moreover, if we replace the basic functions by $1_{\mathfrak o_v}$ (the characteristic functions of the integers) outside of a finite set $S$ of places, the assertion remains true for $\Re(t)\gg 0$, and on any vertical strip there is a  bound by polynomial seminorms which is uniform in $S$. 
\end{proposition}

\begin{proof}
Let $f = \prod_{v\notin T} f_v^{t,0} \cdot f_T \in \mathcal S_i^t$, where $T$ is a finite set of places and $f_T \in \widehat\otimes_{v\in T} \mathcal S_{i,v}^t$. 

In cases $i=1$ or $4$, by the axioms of \S \ref{basicbaby}, the basic function $f_{i,v}^{t,0}$ is supported on the integers of $k_v$ and on every vertical strip we have a bound
$$ \left|\prod_{v\notin T} f_{i,v}^{t,0} (\xi_T)\right| \le C(t)\cdot |\xi_T|^{-r_i},$$
where $r_i\ge 0$ is a constant and $C(t)$ is of polynomial growth in vertical strips. Notice that for $\Re(t)\gg 0$, where the Euler product of the regular values is convergent, such an estimate holds, uniformly in $S$, if we replace $f_{i,v}^{t,0}$ by $1_{\mathfrak o_v}$ for $v\notin S$. 

On the other hand, $f_T$ is of rapid decay, i.e., $f_T(\xi_T)$ vanishes faster than any power of $|\xi_T|$, and from this it easily follows that the sum over $k^\times$ is absolutely convergent. 

For $\mathcal S^t_2$ and $\mathcal S^t_3$ I refer the reader to the more general Proposition \ref{convergence}.
\end{proof}

Combining this with Proposition \ref{irregOI} we get:

\begin{corollary}\label{analytic}
 For an analytic section $t\mapsto f^t \in \mathcal S^t_i$, the number $PS_i(f^t)$ varies analytically  for $\Re(t)\gg 0$ when $i=2,3$, and for all $t$ when $i=1,4$.
\end{corollary}

\subsection{Poisson summation formula}

We are ready to prove the main result of this section:
\begin{proposition}\label{babyPSF}
 For $\Re(t)\gg 0$, $f_1$ an element of the global Schwartz space $S^t_1$ and $f_3$ an element of $S^t_3$, we have: 
\begin{equation}\label{eqbabyPSF}
 PS_1(f_1) = PS_2\left(\mathcal Ff_1\right),
\end{equation}
$$ PS_3(f_3) = PS_4\left(\mathcal Ff_3\right).$$
\end{proposition}
Both equations amount to the same, of course.

 This immediately implies:
\begin{corollary}\label{babyPSF-special}
 For every $t$ and $f\in \mathcal S_1^t$:
\begin{equation}
  PS_1(f) = PS_4(\mathcal G_tf).
\end{equation}
\end{corollary}

\begin{proof}
 Given $f \in \mathcal S_1^{t_0}$ (for a fixed $t_0$), it can be realized as  is the specialization of an analytic section $f^t \in \mathcal S_1^t$, as $t$ varies in the parameter space. Since $PS_i(f^t)$ is analytic in $t$ (Lemma \ref{analytic}), it suffices to prove it for large $t$; but then it follows from the above proposition and the fact that Poisson sums are preserved under $\iota_t$ (Lemma \ref{iotainvariant}).
\end{proof}

\begin{proof}[Proof of Proposition \ref{babyPSF}]
 We can approximate both sides of (\ref{eqbabyPSF}) by expressions which depend on a finite set of places $T$, in the limit as $T$ tends to include all places.

Indeed, we notice:
\begin{lemma}\label{babyapproximation}
 If $f=\otimes^\prime_v f_v\in \mathcal S^t_i$, $\Re(t)\gg 0$, then
\begin{equation}
 \lim_T \sum_{\xi\in k} \prod_{v\notin T}1_{\mathfrak o_v}(\xi)\prod_{v\in T} f_v(\xi) = \sum_{\xi\in k} f(\xi).
\end{equation}
\end{lemma}

Notice that the sums on both sides include $\xi = 0$; since we have only explained how to think of elements of $\mathcal S_i^t$ as functions on $\adele^\times$, this needs some explanation. Recall from Table \eqref{tablebaby} that the asymptotic behavior of $f_{i,v}^{t,0}$ around $\xi=0$ is of the form
$$f_{i,v}^{t,0}(\xi) = L_v(\eta_v,t) + L_v(\eta_v,-t) \eta_v(\xi)|\xi|_v^t$$
when $i=1$ or $4$, and of the form
$$f_{i,v}^{t,0}(\xi) = L_v(\eta_v,1+t) $$
when $i=2$ or $3$. Moreover, for $\Re(t)>0$ the elements of $\mathcal S_{i,v}^t$ extend continuously to $\xi=0$. Therefore it is natural, for large $t$, to extend the evaluation of elements of $\mathcal S_i^t$ to $\xi=0$ by taking the Euler product of their local extensions to $\xi=0$.

To prove the lemma, notice that for any given $\xi$, we clearly have
 $$ \lim_T  \prod_{v\notin T}1_{\mathfrak o_v}(\xi)\prod_{v\in T} f_v(\xi) = f(\xi).$$ By Proposition \ref{babyconvergence} we may interchange the sum over $\xi$ and the limit over $T$, and this proves the lemma.

Now for every given $T$, the function $\prod_{v\notin T}1_{\mathfrak o_v}\prod_{v\in T} f_v$ satisfies conditions for the usual Poisson summation formula: it is continuous, decays faster than $|\xi|^{-1-\delta}$ at infinity, and its Fourier transform $\prod_{v\notin T}1_{\mathfrak o_v}\prod_{v\in T} \mathcal F(f_v)$ also has the same properties (since $\mathcal F f_v$ belongs to $\mathcal S^t_{2,v}$ if $f_v$ belongs to $\mathcal S^t_{1,v}$ and vice versa). Hence we have (say, for $f\in \mathcal S_1^t$)
$$\sum_{\xi\in k} \prod_{v\notin T}1_{\mathfrak o_v}(\xi)\prod_{v\in T} f_v(\xi) = \sum_{\xi\in k}\prod_{v\notin T}1_{\mathfrak o_v}(\xi)\prod_{v\in T} \mathcal F f_v(\xi).$$

Taking the limit with $T$ we get
\begin{equation}
 \sum_{\xi\in k} f(\xi) =  \sum_{\xi\in k} \mathcal Ff(\xi).
\end{equation}

Finally, if we add to the above the relations \eqref{dfromc} and \eqref{dfrome} and take into account the fact that the global gamma factors $\gamma(\eta,-t,\psi)$ and $\gamma(\eta,-t,\psi)$ are equal to $1$, we get the desired result.
\end{proof}

\subsection{Application: the baby case}

Finally, we describe a deformation $\mathcal S_1^t=\mathcal S_4^t$ of the global Schwartz space $\mathcal S(\mathcal X(\adele))$ of the baby case, and verify that the spaces $\mathcal S_i^t$ ($i=1\dots 4$, where the spaces $\mathcal S_2^t=\mathcal S_3^t$ are obtained by Fourier transform from $\mathcal S_1^t$) satisfy the postulated axioms.

First of all, it is clear from the definitions that the local Schwartz spaces $\mathcal S_{1,v}^t$, $\mathcal S_{4,v}^t$ specialize to $\mathcal S(\mathcal X_v)$ when $t=0$. We now endow them with the following basic function, which, as we will see, on one hand coincides with the basic function of $\mathcal S(\mathcal X_v)$ when $t=0$, and on the other satisfies the axioms of \S \ref{basicbaby}: 
\begin{equation}\label{basicfnbaby}
 f_{1,v}^{t,0} (\xi):= \begin{cases}
              \frac{1-\eta(\varpi_v)q_v^{-t} \cdot \eta_v(\xi)|\xi|_v^t}{1-\eta_v(\varpi_v)q_v^{-t}}, & \text{ when } t\ne 0\text{ or }\eta_v\ne 1;\\
              1-\log_{q_v}|\xi|_v,& \text{ when }t=0\text{ and }\eta_v=1.
             \end{cases}
\end{equation}
This function can be obtained by suitably normalized orbital integrals corresponding to $(k_v^\times,\eta_v(\bullet)|\bullet|_v^t)$-coinvariants of the characteristic function of $1_{\mathfrak o^2}$; however, this is not important for us here. What is important is the following: 

\begin{lemma}
Let $X=\Res_{E/k}\Ga$ as in the previous section, and consider the action of $T=U(1)$ on $X$; thus $\mathcal X=X/T$. Then for $t=0$ and for a suitable Haar measure on $T(k_v)$, the above basic function is equal to image (i.e., the orbital integrals) in $\mathcal S(\mathcal X(k_v))$ of the characteristic function of $X(\mathfrak o_v)$.
\end{lemma}

The proof is an easy calculation and will be omitted. This lemma shows that for $t=0$ we get, indeed, the global Schwartz space $\mathcal S(\mathcal X(\adele))$ of the ``baby case''.

A calculation as in the proof of Lemma \ref{lemmad} shows that for large $t$ and $|\xi|\le 1$
\begin{equation}\label{babyfn2a}
 \mathcal F(f_{1,v}^{t,0})(\xi) = \frac{1}{1-\eta(\varpi_v)q_v^{-t}} \left( 1- \eta(\varpi_v)q_v^{-t} \frac{1-q_v^{-1}}{1-\eta(\varpi_v)q_v^{-t-1}}\right)= L(\eta_v,t+1),
\end{equation}
while for large $t$ and $|\xi|>1$
\begin{equation}\label{babyfn2b}
 \mathcal F(f_{1,v}^{t,0})(\xi) = \frac{-\eta(\varpi_v)q_v^{-t}}{1-\eta(\varpi_v)q_v^{-t}}\cdot \frac{1-\eta(\varpi_v)q_v^t}{1-\eta(\varpi_v)q_v^{-t-1}}\cdot \eta(\xi)|\xi|^{-t-1} = L(\eta_v,t+1) \eta(\xi) |\xi|^{-t-1}.
\end{equation}

In particular, the function $f_{2,v}^{t,0}=  \mathcal F(f_{1,v}^{t,0})$ is equal to $L_v(\eta_v,t+1)$ on $\mathfrak o_v^\times$, and is preserved by the transformation $\iota_t$. Hence $f_{3,v}^{t,0}  = f_{2,v}^{t,0}$ and $f_{4,v}^{t,0}=f_{1,v}^{t,0}$. It is immediate to verify that these functions satisfy the axioms of \S \ref{basicbaby}, hence Corollary \ref{babyPSF-special} holds for $\mathcal S_1^t=\mathcal S_4^t = \mathcal S(\mathcal X(\adele))$. This gives a direct proof of Theorem \ref{Poissonbaby}.

\section{Poisson summation for the relative trace formula} \label{sec:Poisson-RTF}

We now return to our main problem, namely proving a Poisson summation formula between the torus relative trace formula \eqref{RTFdef} (viewed as a functional on the global Schwartz space $\mathcal S(\mathcal Z(\adele))$, defined in \S \ref{globalSchwartz}) and the Kuznetsov formula with non-standard sections \eqref{KTFdef} (viewed as a functional on the global Schwartz space $\mathcal S(\mathcal W^s(\adele))$). As explained in the introduction, this will not be possible ``on the nose'', because the sum \eqref{KTFdef} does not converge at the desired point of evaluation $s=0$; therefore, we will also need to deform the space $\mathcal S(\mathcal Z(\adele))$ with a parameter $s$, and prove an identity for $\Re(s)\gg 0$.

Thus, here we will have two complex parameters, $s$ and $t$. The parameter $t$ will, as in the previous chapter, help us deform an exponent of the asymptotics of Schwartz functions by ``larger'' exponents so that our functions vanish faster and are suitable for Poisson summation. The parameter $s$ will be more than just a technical tool, and it parametrizes the space $\mathcal S(\mathcal W^s(\adele))$ of non-standard test functions for the Kuznetsov trace formula in such a way that (when $t=0$) this space corresponds to the $L$-function
$$ L(\pi,\frac{1}{2}+s) L(\pi\otimes \eta,\frac{1}{2}+s),$$
as we saw in \S \ref{globalSchwartz}.

The main result of this section is Theorem \ref{PSF}. It is a Poisson summation formula for large values of the parameter $s$. Its Corollary \ref{analyticcont} shows that the functional $\KTF$ of \eqref{KTFdef} can be analytically continued to arbitrary values of $s$; however, we will have no explicit expression for this functional at $s=0$, and we will continue to work in a domain of convergence of the Euler product for the above $L$-function ($\Re(s)\gg 0$) when we perform the spectral analysis in the next sections.

As in the baby case, we proceed axiomatically by defining various Schwartz cosheaves on the projective line, before applying the theory to the relative trace formulas. To avoid heavy notation, throughout our discussion in most of this section, \textbf{the parameter $s$ will be fixed and will usually not appear explicitly in the notation}.
Moreover, ``the parameters'' refers to the parameters $s, t, 2s\pm t$, and ``large values of the parameters'' means large values of their real parts. 

\subsection{Axioms for the local Schwartz spaces} \label{ssaxiomslocal}

As in the baby case, we introduce four local Schwartz spaces $\mathcal S^t_{i,v}$ ($i=1,\dots, 4$) (the parameter $s$ will be implicit). For this section we denote by $\mathcal B^\reg$ the complement of $\{-1,0,\infty\}$ in one-dimensional projective space. (The point $-1$ will not literally be irregular in all cases, but to condense and unify notation we consider it as such.) The cosheaves $\mathcal S^t_{i,v}$, restricted to $\mathcal B^\reg(k_v)$, all coincide with the cosheaf of Schwartz functions. The definition of these four spaces, for \emph{generic values} of the parameters, is completed by the following table, which describes the asymptotic behavior of their elements close to the ``singular points''. Here, the $C_i$'s and $D_i$'s are smooth functions. 

\begin{equation}\label{tablegerms}
 \begin{array}{|c|c|c|}
 \hline  
 i & \mbox{ around } \xi=0 &  \mbox{ around } \xi=-1 \\
\hline
 1 \mbox{ or }4 & C_1(\xi) + C_2(\xi) \eta(\xi)|\xi|^t & C_3(\xi) + C_4(\xi) \eta(\xi+1)|\xi+1|^{2s} \\
\hline
 2 \mbox{ or }3 & D_1(\xi) + D_2(\xi) |\xi|^{-t+2s} \psi\left(\frac{1}{\xi}\right) & D_3(\xi)  \\
\hline
 \mbox{(Continued) }i & \mbox{ around } \xi=\infty  &\\
\hline
 1 \mbox{ or }4 & C_5(\xi^{-1})|\xi|^{-s+\frac{t}{2}-1}\cdot \Kl(\xi) &\\
\hline
 2 \mbox{ or }3 & D_4(\xi^{-1})\eta(\xi)|\xi|^{-t-1} + D_5(\xi^{-1})\psi(\xi)\eta(\xi)|\xi|^{-2s-1} &\\
\hline
\end{array}
\end{equation}

Here and later, $\Kl$ (for ``Kloosterman''), at non-Archimedean places, denotes the function which is supported on $|\xi|>1$ and equal to
\begin{equation}\label{defKl}\int_{|x|^2=|\xi|} \psi(x-\xi x^{-1}) dx \end{equation} there. For the Archimedean case, the analogous ``Kloosterman germ'' at infinity is the germ of functions as in \cite[(4.27)]{SaBE1}. This definition is actually correct for the space $\mathcal S_{4,v}^t$; for the space $\mathcal S_{1,v}^t$ we need to replace $\xi$ by $-\xi$; however, since in our application to the relative trace formula these Kloosterman germs will not appear for $i=1$ (i.e., the germ of the function $C_5$ at zero will be zero), we do not introduce new notation for this minor modification.

In the limit when the exponents become integers (notably, when $s=t=0$ and $\eta_v=1$), we may also have logarithmic terms; the limiting behavior in those cases has been described in Appendix \ref{sec:families}, and is completely analogous to the baby case. The Fr\'echet structure on these spaces is also described in the appendix.

Again, the spaces $\mathcal S^t_{2,v}$ and $\mathcal S^t_{3,v}$ are mapped isomorphically onto each other by the operator $\iota_t$ of \eqref{defiotat}:
$$\iota_t(f) = \frac{\eta(\bullet)}{|\bullet|^{t+1}} f\left(\frac{1}{\bullet}\right).$$ 

In analogy to Proposition \ref{variousproperties} and Lemma \ref{lemmad}, we have:
\begin{proposition} \label{variousproperties-RTF}
 For $t$ and $2s$ outside the exceptional values \eqref{exceptions}, Fourier transform carries $\mathcal S^t_{1,v}$ isomorphically to $\mathcal S^t_{2,v}$, and $\mathcal S^t_{3,v}$ to $\mathcal S^t_{4,v}$. If $f\in \mathcal S^t_{1,v}$ has asymptotics denoted by $C_i$ as in table \eqref{tablegerms}, the corresponding asymptotic coefficients for its Fourier transform are: 
\begin{eqnarray} \label{ad2}
 D_4(0) = \gamma(\eta_v,-t,\psi_v) C_2(0),\nonumber\\
 D_2(0) = C_5(0), \\
 D_5(0) = \gamma(\eta_v,-2s, \psi_v) C_4(0) \nonumber 
\end{eqnarray}
 Similarly, if $f\in \mathcal S^t_{3,v}$ has asymptotics denoted as in table \eqref{tablegerms}, then its asymptotic coefficients $D_i$ are related to the asymptotic coefficients of its Fourier transform as follows:
 \begin{eqnarray} \label{ad3}
 D_4(0) = \gamma(\eta_v,-t,\psi_v^{-1}) C_2(0),\nonumber\\
 D_2(0) = C_5(0), \\
 D_5(0) = \gamma(\eta_v,-2s, \psi_v^{-1}) C_4(0) \nonumber 
\end{eqnarray}
 The transformation $\mathcal G_t$ of \eqref{defGt} carries $\mathcal S^t_{1,v}$ isomorphically to $\mathcal S^t_{4,v}$, even for values of $t$ and $2s$ as in \eqref{exceptions}.
 
 Both $\mathcal F$ and $\mathcal G_t$ are bounded by polynomial families of seminorms on the corresponding spaces, as $t$ varies, and preserve analytic sections.
\end{proposition}

In our application to the relative trace formula, unlike in the baby case, the spaces $\mathcal S_{1,v}^t$ and $\mathcal S_{4,v}^t$ will not have the same basic functions; actually, we will restrict to the subspace of $\mathcal S_{1,v}^t$ where the germ of $C_5$ is zero (i.e., elements of $\mathcal S_{1,v}^t$ will standard Schwartz functions away from $0, -1$), and accordingly for $\mathcal S_{4,v}^t$ we will have $C_4=0$ (sections of $\mathcal S_{4,v}^t$ will be smooth at $-1$). This is very important in order to be able to continue to $s=0$ (continuation will be a result of the Poisson sum for $\mathcal S_{1,v}^t$ being the sum of a rapidly decaying function), but it is not important in proving the Poisson summation formula for large values of the parameters, and therefore we use this more general approach which applies to $\mathcal S_1^t$ and $\mathcal S_4^t$ simultaneously.

Finally, we note that the Schwartz space $\mathcal S(\mathcal W^s(\adele))$ for the Kuznetsov trace formula with non-standard test functions depending on a parameter $s$ (defined in \cite{SaBE1}, section 6) will belong to the space
$$|\bullet|^{s+1}\cdot \mathcal S^0_4.$$

\subsection{Global Schwartz spaces and convergence of the Poisson sum} \label{basicfnassumptions}

As before, we define global Schwartz spaces $\mathcal S^t_i$ as restricted tensor products with respect to basic vectors $f^{t,0}_{i,v}$ (again, the parameter $s$ is implicit in the notation). These basic vectors satisfy the compatibility relations:
\begin{eqnarray}\label{compatibility}
 f^{t,0}_{2,v} &=& \mathcal F(f^{t,0}_{1,v}), \nonumber\\
 f^{t,0}_{3,v} &=& \iota_t f^{t,0}_{2,v}\\
 f^{t,0}_{4,v} &=& \mathcal F(f^{t,0}_{3,v}). \nonumber
\end{eqnarray}

Besides, they are required to satisfy, \emph{for large values of the parameters}, the following axioms:

\begin{enumerate}
 \item Basic functions are constant on $\mathcal B^\reg(\mathfrak o_v)$ and their values there (their ``regular values'') are such that the corresponding Euler products converge, locally uniformly in the real parts of the parameters.
 \item The constant terms of the asymptotics of basic functions (i.e., $C_1(0)$, $C_3(0)$, $D_1(0)$, in the notation of Table \ref{tablegerms}) give convergent Euler products, locally uniformly in the real parts of the parameters. For the rest of the asymptotic constants (the rest of the numbers $C_i(0), D_i(0)$) we only need to assume that one representative in each implicit chain of identities given by \eqref{ad2}, \eqref{ad3} and \eqref{compatibility} has a convergent Euler product, locally uniformly in the real parts of the parameters. For example, the constants $C_5(0)$ of $f_{1,v}^{t,0}$, $D_2(0)$ of $f_{2,v}^{t,0}$, $D_5(0)$ of $f_{3,v}^{t,0}$ and $C_4(0)$ of $f_{4,v}^{t,0}$ are all determined by each other via these formulas, and it is enough to have a convergent Euler product for one of them.
 
 \begin{remark}
  In the baby case, \S \ref{basicbaby}, we explicitly gave the values of those asymptotic constants in our axioms for the basic functions. The reader can consult table \ref{table} to see the values that will be used for the relative trace formula.
 \end{remark}

 \item There is a constant $r_i\ge 0$ (possibly different for each of the spaces), independent of the place, with the property that the function: 
$$|1+\xi|_v^{r_i} \cdot |\xi|^{r_i} \cdot \frac{|f_{i,v}^{t,0}(\xi)|}{|f_{i,v}^{t,0}(\mathcal B^\reg(\mathfrak o_v))|}\mbox{ is}$$
\begin{itemize}
 \item $\le 1$ for $|\xi|\le 1$;
 \item $\le |\xi|^{-M}$ for $|\xi|>1$, where $M$ is a prescribed large positive number (depending on the global field, s.\ the proof of Proposition \ref{convergence}).
\end{itemize}
\end{enumerate}

I repeat that these axioms hold \emph{for large values of the parameters}, and they will be enough to prove a Poisson summation formula for such values, Proposition \ref{Poissonlarge}. We will then specialize to the specific basic functions showing up in our relative trace formulas, which will have additional properties allowing analytic continuation to other values of the parameters (\S \ref{sscontinuation}). (In the baby case, \S \ref{basicbaby}, we added these additional properties to the axioms, but that would be too cumbersome to do in this more complicated case.)

Given the above axioms, we have the following generalization of Proposition \ref{babyconvergence}:
\begin{proposition}\label{convergence} 
For large values of the parameters and $f$ in one of the global Schwartz spaces $\mathcal S^t_i$, the functional
$$f\mapsto \sum_{\xi\in \mathcal B^\reg(k)} f(\xi)$$ converges absolutely, and can be bounded by polynomial seminorms on the space $\mathcal S^t_i$ in any bounded vertical strip. 
 Moreover, if we replace the basic functions by $1_{\mathfrak o_v}$ outside of a finite set $S$ of places, the assertion remains true, and the polynomial bound can be taken to be independent of the set $S$. 
\end{proposition}

\begin{proof}
I prove it for number fields, and leave the case of function fields to the reader. Up to a convergent Euler product, the sum is the same if we replace $f_v$ by
$$|1+\xi|_v^{r_i}\cdot |\xi|_v^{r_i} \frac{f_v(\xi)}{f_v(\mathcal B^\reg(\mathfrak o_v))}$$ (where in a finite set of places we interpret the denominator as $1$), therefore we may assume that the basic functions themselves are $\le 1$ for $|\xi|_v\le 1$, $\le |\xi|_v^{-M}$ for $|\xi|_v>1$, and their regular value is $1$.

We define a ``height function'' on every $k_v$ by: $r_v(\xi) = \max(1,|\xi|_v)$. Notice that for $\xi \in k^\times$ we have $r_v(\xi)=1$ for almost every place, therefore $r(\xi)=\prod_v r_v(\xi)$ makes sense. If $k=\QQ$ then, for $(m,n)=1$ we have
$$ r\left(\frac{m}{n}\right) = \max(1,\left|\frac{m}{n}\right|_\RR) \cdot |n|_\RR = \max(|m|_\RR,|n|_\RR).$$

In general, we claim:
\begin{lemma}\label{heightslemma}
 There is a positive number $N$ such that
\begin{equation}
 \#\{\xi\in k| r(\xi)< T\} < T^N\text{ for all }T\gg 0.
\end{equation}
\end{lemma}

Indeed, choose a basis $(v_i)_{i=1}^n$ of $k$ over $\QQ$ and define modified height functions:
$$r'_v(x=\sum a_i v_i) = \prod_{i=1}^n r_v(a_i),$$
where the height functions on the right hand side are those of $\QQ_v$. At almost every place $k_v/\QQ_v$ is unramified, and the elements $v_i$ are integral and a basis for the unramified residue field extension; for those places: $|x|_v=\max_i|a_i|_{k_v} = \max_i |a_i|^n_{\QQ_v}$ and therefore $r'_v(x)\le r_v(x)\le r'_v(x)^n$. A similar relation holds for the remaining finite set of places, but with certain positive constants: $m_{1,v} r'_v(x)\le r_v(x)\le m_{2,v} r'_v(x)^n$, and therefore also globally on $k$-points:
$$m_{1} r'(\xi)\le r(\xi)\le m_{2} r'(\xi)^n, \,\,\xi\in k.$$
This reduces the lemma to the case of $k=\QQ$, where it is obvious.

\vspace{4pt}

We continue with the proof of the proposition. By our axioms for the basic functions, we have the following estimate on $f$:
$$|f(\xi)|\le C\cdot r(\xi)^{-M}$$
for $\xi\in \mathcal B^\reg(k)$ and some constant $C$ bounded by polynomial seminorms. (We assume that the real values of the parameters are large enough, so that the bound by a multiple of $|\xi|_v^{-M}$ when $|\xi|_v>1$ holds at every place, not just for the basic functions.)

Hence
$$\sum_\xi |f(\xi)| \ll   \sum_{T\in \mathbb N} T^N\cdot T^{-M},$$  
which converges absolutely for large $M$, and the implicit constant is bounded by polynomial seminorms.

Notice that if we replace $f_v$ by ${f_v(\mathcal B^\reg(\mathfrak o_v))}\cdot 1_{\mathfrak o_v}$ outside of a finite number of places $S$, the same estimates are true, uniformly in $S$. Finally, by the (locally uniform in the real part of the parameters) absolute convergence of the partial Euler product $\prod_{v\notin S}{f_v(\mathcal B^\reg(\mathfrak o_v))}$, we deduce that the same is true if we replace $f_v$ by $1_{\mathfrak o_v}$. 
\end{proof}

\subsection{Irregular distributions.}\label{ssirregular} 

In complete analogy with the previous section, we define functionals $\tilde O_0$, $\tilde O_{-1}$ and $\tilde O_\infty$ on the global Schwartz spaces $\mathcal S_i^t$ which formally, in the notation of Table \ref{tablegerms} but adding an index $v$, are (for generic values of the parameters):

For $i=1$ or $4$: 
$$\tilde O_0(f) = \prod_v' C_{1,v}(0) + \prod_v' C_{2,v}(0),$$
$$\tilde O_{-1}(f) = \prod_v' C_{3,v}(0) + \prod_v' C_{4,v}(0),$$
$$\tilde O_\infty(f) = \prod_v' C_{5,v}(0).$$

For $i=2$ or $3$:
$$\tilde O_0(f) = \prod_v' D_{1,v}(0) + \prod_v' D_{2,v}(0),$$
$$\tilde O_{-1}(f) = \prod_v' D_{3,v}(0),$$
$$\tilde O_\infty(f) = \prod_v' D_{4,v}(0) + \prod_v' D_{5,v}(0).$$

The rigorous definition is by using partial Euler products as in \eqref{zerocontrib}, which is possible by our axioms for the basic functions. These functionals extend analytically to all values of the parameters, as in \S \ref{ssbabyirregular}.

\subsection{Poisson summation for large values of the parameters}

The functionals $PS_i$ will be defined on the global Schwartz spaces $\mathcal S^t_i$ in a completely analogous way to the baby case, namely: 

\begin{equation}\label{Poissonsum}
 PS_i(f) = \tilde O_0 (f) + \tilde O_{-1}(f) + \tilde O_\infty(f) + \sum_{\xi \in k\smallsetminus\{-1,0\}} f(\xi).
\end{equation}

They are well-defined when the parameters are large. 
Again, the following is clear:

\begin{lemma}\label{itpreserves}
 The transformation $\iota_t: \mathcal S_2^t\xrightarrow{\sim} \mathcal S_3^t$ preserves Poisson sums.
\end{lemma}

We are now ready to prove the Poisson summation formula for large parameters:
\begin{proposition} \label{Poissonlarge}
 For large values of the parameters, $f_1$ an element of the global Schwartz space $\mathcal S_1^t$ and $f_3$ an element of $\mathcal S_3^t$ we have:
\begin{equation}
 PS_1(f_1) = PS_2 (\mathcal F(f_1));
\end{equation}
\begin{equation}
 PS_3(f_3) = PS_4 (\mathcal F(f_3)).
\end{equation}
\end{proposition}

Of course, both formulas amount to the same.

\begin{proof}
We essentially go over the same steps as in the baby case, Proposition \ref{babyPSF}. First, the analog of Lemma \ref{babyapproximation} holds: for large values of the parameters and $f=\otimes_v f_v \in \mathcal S_i^t$ we have:
\begin{equation}\lim_{T} \sum_{\xi\in k} \prod_{v\notin T} 1_{\mathfrak o_v} \prod_{v\in T} f_v(\xi) = \sum_{\xi\in k} f(\xi),\end{equation}
where $T$ is a finite set of places and the limit is taken as $T$ includes every place. Again, we take summation over all points, not just regular ones, and here is where we use the fact that the asymptotic constant terms at irregular points give convergent Euler products. The argument is the same, using this time Proposition \ref{convergence}. 

The rest of the proof is also the same: we apply the usual Poisson summation formula to the sum: $\sum_{\xi\in k} \prod_{v\notin T} 1_{\mathfrak o_v} \prod_{v\in T} f_v(\xi)$, taking the limit over $T$, and finally add the contributions of the rest of the terms to the Poisson sum, which match because of Proposition \ref{variousproperties-RTF} and the global triviality of gamma factors.

\end{proof}

Combining Proposition \ref{Poissonlarge} and Lemma \ref{itpreserves} we get:

\begin{corollary}\label{GPoisson}
 For large values of the parameters and $f\in \mathcal S_1^t$ we have:
 $$ PS_1(f) = PS_4(\mathcal Gf).$$
\end{corollary}

\subsection{The spaces of the relative trace formula} \label{spacesRTF}

In \cite{SaBE1}, section 6, I defined a local ``Schwartz space'' of functions (here with an index $v$) $\mathcal S(\mathcal W_v^s)$,
corresponding to the orbital integrals for a space of non-standard test functions for the Kuznetsov trace formula, depending on a parameter $s$. We also endowed this space (at almost every place) with a basic vector (which we will modify here by a volume factor, see below).

The following completely characterizes the space $\mathcal S(\mathcal W_v^s)$:

\begin{lemma}\label{Wembedding}
 Multiplication by $|\bullet|^{-s-1}$ is an isomorphism:
\begin{equation}
 \mathcal S(\mathcal W_v^s)\xrightarrow{\sim} (\mathcal S^0_{4,v})',
\end{equation}
where $(\mathcal S^0_{4,v})'$ denotes the closed subspace of $\mathcal S^0_{4,v}$ consisting of elements for which, in the notation of Table \ref{tablegerms}, the stalk of $C_4$ at $-1$ is zero (i.e., those elements of $\mathcal S^0_{4,v}$ which are smooth at $\xi=-1$). 

Moreover, the transform $\mathcal G$ is an isomorphism:
\begin{equation}
(\mathcal S^0_{1,v})' \xrightarrow{\sim} (\mathcal S^0_{4,v})',
\end{equation}
where $(\mathcal S^0_{1,v})'$ denotes the closed subspace of $\mathcal S^0_{1,v}$ consisting of elements for which, in the notation of Table \ref{tablegerms}, the stalk of $C_5$ at $0$ is zero (i.e., those elements of $\mathcal S^0_{1,v}$ which are of rapid decay at infinity).
\end{lemma}

The first statement is a straightforward generalization of the description of $\mathcal S(\mathcal W_v^0)$ in \cite{SaBE1}, section 4, and the second statement is easy; we omit the proofs. For later use, we denote the space $(\mathcal S^0_{1,v})'$ by $\mathcal S(\mathcal Z_v^s)$; at $s=0$ it specializes to the space $\mathcal S(\mathcal Z_v)$ of orbital integrals for the torus RTF, as proven in the matching theorem \cite[Theorem 5.1]{SaBE1}.

To apply the Poisson summation formula that we proved to this space, we need to define the degeneration of its basic function with a parameter $t$ and verify that it, and the corresponding basic functions for the other spaces $\mathcal S_{i,v}^t$, satisfy the axioms of \S \ref{basicfnassumptions}. In this paper we will use the basic function for $\mathcal S(\mathcal W_v^s)$ in \cite{SaBE1}, Lemma 6.4, but \textbf{divided by a volume term} (for convenience, since such a volume term does not affect the Poisson summation formula); this modification, as we will see, will be compatible with our requirement for $f_{\mathcal Z_v}^0$ to have regular value 1. Here is that function, translated to the space $\mathcal S_{4,v}^0$ (i.e., multiplied by $|\bullet|_v^{-s-1}$):

\begin{equation}\label{basicvector}
 f_{4,v}^{0,0}(\xi) := L(\eta_v,2s+1)\left((I-q_v^{-2s-1} \varpi_v^2\cdot) f(\xi) + |\xi|_v^{-s-1} 1_{|\xi|_v=q_v^2}  + |\xi|_v^{-s-1} \Kl(\xi)\right),
\end{equation}
where
\begin{itemize}
 \item $\Kl(\xi)$ was defined in \eqref{defKl};
 \item $f$ is the basic function of the ``baby case'' Schwartz space (supported on $|\xi|_v\le 1$):
 $$ f(\xi) = \begin{cases}
              1-\log_{q_v}|\xi|_v, & |\xi|_v\le 1, \mbox{ in the split case,}\\
              \frac{1+\eta_v(\xi)}{2}, & |\xi|_v\le 1, \mbox{ in the non-split case;}
             \end{cases}$$
 \item the action of $\varpi_v^2$ is normalized as in \eqref{unitaryaction}.
\end{itemize}

Now we modify this to define the basic function $f_{4,v}^{t,0}$ as follows:
\begin{itemize}
 \item we replace the function $f$ by the basic function $f^t$ of the baby case corresponding to arbitrary $t$, which for $t\ne 0$ is equal to:
$$ (L(\eta_v,t) + L(\eta_v,-t) \eta_v(\xi)|\xi|_v^t) \cdot 1_{\mathfrak o_v};$$
 \item we replace the factor $|\xi|_v^{-s-1}$ by $|\xi|_v^{-s+\frac{t}{2}-1}$.
\end{itemize}

For the calculations that follow, we split $f_{4,v}^{t,0}$ into a sum $F_1 + F_2$ (we hide $t$ and $v$ from this notation, as we do throughout for the parameter $s$), where the summands are:
\begin{eqnarray}\label{basicfnvariation} F_1(\xi) &=&L(\eta_v,2s+1)(I-q_v^{-2s-1} \varpi_v^2\cdot) f_t(\xi) \nonumber \\
F_2(\xi) &=& L(\eta_v,2s+1) \left( q_v^{-2s+t-2} 1_{|\xi|_v=q_v^2}  + |\xi|_v^{-s+\frac{t}{2}-1} \Kl(\xi) \right).
\end{eqnarray}

\begin{remark}
 It can easily be computed that the function $F_1$ is equal, for small $|\xi|_v$, to
$$\left(\frac{L(\eta_v,2s+1)}{\zeta_v(2s+2)} L(\eta_v,t) + \frac{L(\eta_v,2s+1)}{\zeta_v(2s+2t+2)} L(\eta_v,-t) \eta_v(\xi)|\xi|_v^t\right).$$
\end{remark}

\subsection{Verification of the axioms}\label{transformcalculations}

Here we verify that the basic function $f_{4,v}^{t,0}$ defined in the previous subsection gives rise to functions $f_{i,v}^{t,0}$  ($i=1\dots4$) which satisfy the axioms of \S \ref{basicfnassumptions}. For this subsection we drop the index $v$ and set: $\epsilon = \eta_v(\varpi_v)$. (Recall that we are defining the basic functions outside of a finite set of places which includes the places of ramification of $\eta$; therefore, $\eta_v(\varpi_v)$ makes sense.)

\begin{lemma}
 The basic function $f_{4,v}^{t,0}$ defined in the previous subsection, and the basic functions $f_{i,v}^{t,0}$, $i=1,2,3$, obtained from $f_{4,v}^{t,0}$ via \eqref{compatibility}, satisfy the axioms of \S \ref{basicfnassumptions}.
\end{lemma}

\begin{proof}
The summand $F_1$ of (\ref{basicfnvariation})  has inverse Fourier transform
$$ L(\eta,2s+1) \cdot (I-q^{-2s-1} \varpi^{-2}\cdot) \mathcal F^{-1}f.$$
In particular, from (\ref{babyfn2b}) we deduce that its value on $\mathfrak o^\times$ is equal to
$$ \frac{L(\eta,2s+1)}{\zeta(2s+2t+2)} L(\eta,t+1).$$

Notice that $\iota_t(a\cdot f) = \eta(a) |a|^{-t} a^{-1}\cdot (\iota_t f)$. Therefore
$$\iota_t\mathcal F^{-1}(F_1)= L(\eta,2s+1) \cdot (I-q^{-2s-2t-1}  \varpi^{2}\cdot) \iota_t\mathcal F^{-1}f.$$
In particular, from (\ref{babyfn2a}) we deduce that its value on $\mathfrak o^\times$ is equal to
$$ \frac{L(\eta,2s+1)}{\zeta(2s+2t+2)} L(\eta,t+1).$$ 

Finally, by applying inverse Fourier transform once more we get
$$\mathcal G_t^{-1}(F_1)= L(\eta,2s+1) \cdot (I-q^{-2s-2t-1}  \varpi^{-2}\cdot) \mathcal G_t^{-1}f,$$
which by (\ref{basicfnbaby}) takes on $\mathfrak o^\times$ the value
$$ L(\eta,2s+1).$$

It can be seen that the inverse Fourier transform of $|x|^{-\frac{\alpha}{2}-1}\cdot \Kl(x) + q^{-\alpha-2} 1_{|x|=q}(x)$ is
$$q^{-\alpha} 1_{|x|\le q^{-2}} + |x|^\alpha \psi(\frac{1}{x}) 1_{|x|<1}.$$ 

Hence, 
$$\mathcal F^{-1} F_2 (x)= L(\eta,2s+1)q^{-2s+t} 1_{|x|\le q^{-2}} + L(\eta,2s+1) |x|^{2s-t} \psi\left(\frac{1}{x}\right) 1_{|x|<1}.$$
Its value on $\mathfrak o^\times$ is equal to $0$.

Now,
$$\iota_t\mathcal F^{-1} F_2 (x)= L(\eta,2s+1)q^{-2s+t} \eta(x)|x|^{-t-1} 1_{|x|\ge q^2} + L(\eta,2s+1) \eta(x)|x|^{-2s-1} \psi(x) 1_{|x|>1},$$
with value on $\mathfrak o^\times$ equal to $0$.

Finally, it can be seen that $\eta(x)|x|^{-\delta-1} 1_{|x|>1}$ is the Fourier transform of
$$\left(\frac{L(\eta,-\delta)}{L(\eta,\delta+1)} \eta(x) |x|^\delta - \frac{L(\eta,-\delta)}{\zeta(1)}\right) 1_\mathfrak o(x).$$
Notice that $\eta(x)|x|^{-t-1} 1_{|x|\ge q^2}= \epsilon q^{-t-\frac{1}{2}} \varpi\cdot \left(\eta(x) |x|^{-t-1}1_{|x|>1}(x)\right)$. Therefore, 
$$ \mathcal G_t^{-1} F_2 (x)=L(\eta,2s+1) q^{-2s} \epsilon \left(\frac{L(\eta,-t)}{L(\eta,t+1)} \epsilon q^t \eta(x) |x|^t - \frac{L(\eta,-t)}{\zeta(1)}\right) 1_{|x|<1}(x) +$$
$$ + L(\eta,2s+1) \left(\frac{L(\eta,-2s)}{L(\eta,2s+1)}  \eta(x+1) |x+1|^{2s} - \frac{L(\eta,-2s)}{\zeta(1)}\right) 1_\mathfrak o(x).$$
Its value at integral points of $\mathbb A^1\smallsetminus\{-1,0\}$ is:
$$ -\epsilon q^{-2s-1} L(\eta,2s+1).$$

Based on these calculations, the axioms for the basic functions $f_{i,v}^{t,0}$ obtained from $f_{4,v}^{t,0}$ via \eqref{compatibility} are easily checked. We present their regular values and asymptotic behavior in a table, which is the analog of \eqref{tablebaby}, using the notation of Table \eqref{tablegerms}. By ``reg.val'' we mean the value $f_{i,v}^{t,0}(\mathcal B^\reg(\mathfrak o_v))$.

\begin{equation}\label{table}
 \begin{array}{|c|c|c|}
 \hline  
 i & f_{i,v}^{t,0} \mbox{ around } \xi=0 & f_{i,v}^{t,0} \mbox{ around } \xi=-1 \\ 
\hline 
4  & C_1(0)=\frac{L(\eta,2s+1)L(\eta,t)}{\zeta(2s+2)} & C_3(0)=\mbox{Reg.\ val.} \\
& C_2(0) =\frac{L(\eta,2s+1)L(\eta,-t)}{\zeta(2s+2t+2)} & C_4 =0  \\
\hline
3  &  D_1(0) = L(\eta,2s+1)L(\eta,t+1)\cdot  & D_3(0) = \mbox{Reg.\ val.} \\
& \cdot (1-q^{-2s}+q^{-2s+t}-\epsilon q^{-2s-1})&\\
& D_2(0) =L(\eta,2s+1) &  \\
\hline
2  &  D_1(0) = \frac{L(\eta,2s+1)L(\eta,t+1)}{\zeta(2s+2t+2)} & D_3(0) = \mbox{Reg.\ val.} \\
& D_2 = 0 &  \\
\hline
1  &  C_1 (0) = L(\eta,2s+1) L(\eta,t) \cdot & C_3(0) = L(\eta,2s) \\
& \cdot (1-q^{-2s-2t}+q^{-2s-t}-\epsilon q^{-2s-1}) &  \\
& C_2 (0)= L(\eta,2s+1)L(\eta,-t)\cdot & C_4(0) = L(\eta,-2s) \\
& \cdot (1-q^{-2s}+q^{-2s+t}-\epsilon q^{-2s-1})&\\
\hline
\mbox{(Continued) } i & f_{i,v}^{t,0} \mbox{ around } \xi=\infty & f_{i,v}^{t,0} (\mathcal B^\reg(\mathfrak o_v)) \\
\hline
4 & C_5(0)= L(\eta,2s+1) & L(\eta,2s+1)\cdot \\
  & & \cdot [1-q^{-2s-2}(1+\epsilon q^{-t}+q^{-2t})]  \\
\hline
3 & D_4(0)= \frac{L(\eta,2s+1)L(\eta,t+1)}{\zeta(2s+2t+2)} & \frac{L(\eta,2s+1)L(\eta,t+1)}{\zeta(2s+2t+2)}\\
  &  D_5 =0 &\\
\hline
2 & D_4(0)= L(\eta,2s+1)L(\eta,t+1)\cdot & \frac{L(\eta,2s+1)L(\eta,t+1)}{\zeta(2s+2t+2)}\\
  & \cdot (1-q^{-2s}+q^{-2s+t}-\epsilon q^{-2s-1})& \\
  &  D_5(0) = L(\eta,2s+1) &\\
\hline
1 & C_5 =0 & 1\\
\hline
\end{array}
\end{equation}

\end{proof}

I remark the following:
\begin{enumerate}
 \item The constant terms ($C_1(0), C_3(0), D_1(0)$), as well as the regular values, give convergent Euler products when $\Re(s),\Re(t),\Re(s\pm t)\gg 0$ (i.e., ``for large values of the parameters''). They admit meromorphic continuation to $t=0$, as long as the rest of the parameters remain large.
 \item For the rest of the (non-zero) terms, the non-convergent factors of the Euler products can always be interpreted as an $L$-function. They, too, admit meromorphic continuation to $t=0$ when the rest of the parameters remain large.
 \item In the limit when $t\to 0$, $f_{1,v}^{t,0}$ converges to the following function supported on $\mathfrak o_v$:
\begin{equation} \label{basicfn}
\begin{array}{|c|c|c|}
   \hline
 |\xi|_v< 1 & |\xi+1|_v<1 & |\xi|_v=|\xi+1|_v=1 \\

\hline
\begin{cases}
 \frac{1+\eta_v(\xi)}{2}, &\mbox{ if }\eta_v\ne 1 \\
 1-\log_q|\xi|_v, &\mbox{ if } \eta_v=1
\end{cases} & L(\eta_v,2s) + L(\eta_v,-2s)\cdot \eta_v(\xi+1)|\xi+1|_v^{2s} & 1 
\\
\hline
  \end{array}
\end{equation}
\end{enumerate}

\subsection{Continuation to $t=0$ and $s=0$.} \label{sscontinuation}

We remain in the setting of the relative trace formula, i.e., the basic functions discussed in the previous subsection. From the calculations of \S \ref{transformcalculations} it is easy to see that the basic functions $f_{1,v}^{t,0}$ and $f_{4,v}^{t,0}$ satisfy the assumptions of \S \ref{basicfnassumptions} even for $t$ around zero, as long as $s\gg 0$. Therefore, Proposition \ref{convergence} continues to hold for the spaces $\mathcal S_1^t, \mathcal S_4^t$, i.e., summation over the regular points of $\mathcal B(k)$ converges absolutely.

Hence, by analytic continuation we get that Corollary \ref{GPoisson} continues to hold for $t=0$ and large values of $s$; we state it only for the subspace of $\mathcal S_4^0$ which corresponds to the space with parameter $s$: $\mathcal S(\mathcal W^s(\adele))$ of the Kuznetsov trace formula.  We recall once again that $|\bullet|^{-s-1} \mathcal S(\mathcal W^s(\adele))$ can be identified with a subspace $(\mathcal S_4^0)'$ of $\mathcal S_4^0$, cf.\ Lemma \ref{Wembedding}, and the preimage of this in $\mathcal S_1^0$ under $\mathcal G$ is denoted by $\mathcal S(\mathcal Z^s(\adele))$.  For later use, we introduce the following notation for their basic functions:
$$ f_{\mathcal Z_v^s}^0= f_{1,v}^{t,0} \mbox{ (with implicit parameter $s$)},$$
$$ f_{\mathcal W_v^s}^0= |\bullet|^{s+1} f_{4,v}^{t,0} \mbox{ (with implicit parameter $s$)}.$$

By the Matching Theorem 5.1 and the Fundamental Lemma 5.4 of \cite{SaBE1}, the space $\mathcal S(\mathcal Z^s(\adele))$ specializes at $s=0$ to the Schwartz space $\mathcal S(\mathcal Z(\adele))$ of the relative trace formula for $T\backslash G/T$. (I remark here that the basic functions here differ by the basic functions of \cite{SaBE1} by a scalar so that the regular value of $f_{\mathcal Z_v^s}^0$ is one.) Thus, the space $\mathcal S(\mathcal Z^s(\adele))$ is a degeneration of a space of orbital integrals, which carries no 
representation-theoretic information and no Hecke action.\footnote{The space can be actually obtained from the split-torus relative trace formula of Jacquet, by twisting torus periods by a continuous family of characters; thus, it can be endowed with a Hecke action, but of course our purpose is to ignore this action, and any relevance to the split-torus RTF, in hopes that the method can be generalized. For local purposes, in order to shorten calculations, we feel free to use this fact (such as in the proof of Proposition \ref{propfeKuz}).} 
It will only be used to analytically continue the Kuznetsov trace formula with parameter $s$ to $s=0$; thus, a problem of analytic continuation which has a flavor of Langlands' ``Beyond Endoscopy'' is being treated by algebraic means, by transforming the non-compactly supported orbital integrals of the Kuznetsov trace formula to the essentially compactly supported ones of the spaces $\mathcal S(\mathcal Z^s(\adele))$. But, first, let us formulate the Poisson summation 
formula for large values of $s$. 

We set, for large values of $s$,
\begin{equation} \RTF(f) := PS_1(f) \mbox{ for } f\in \mathcal S(\mathcal Z^s(\adele)),
\end{equation}
\begin{equation} \KTF(f) := PS_4(|\bullet|^{-s-1} f) \mbox{ for } f\in \mathcal S(\mathcal W^s(\adele)).
\end{equation}
These are the same definitions for the torus relative trace formula and the Kuznetsov trace formula as given in \eqref{RTFdef}, \eqref{KTFdef}, except that the definition for $\RTF$ is now being applied to deformations of the space $\mathcal S(\mathcal Z(\adele))$. The analytic continuation of Corollary \ref{GPoisson} to $t=0$, applied to those spaces, reads:

\begin{theorem}\label{PSF} 
 For $\Re(s)\gg 0$ and $f\in \mathcal S(\mathcal Z^s(\adele))$ we have:
$$\RTF(f) = \KTF( |\bullet|^{s+1} \mathcal G f).$$
\end{theorem}

Explicitly, this means
\begin{equation}\label{eqPSF}
 \tilde O_0(f) + \tilde O_{-1}(f) + \sum_{\xi \in k\smallsetminus\{-1,0\}} f(\xi) = \tilde O_0(\mathcal G f) + \tilde O_\infty(\mathcal Gf) + \sum_{\xi\in k\smallsetminus\{0\}} \mathcal Gf(\xi).
\end{equation}

Now suppose that $f$ varies in an analytic section of $s\mapsto f_s\in \mathcal S(\mathcal Z^s(\adele))$, 
The basic functions $f_{\mathcal Z_v^s}^0$ satisfy the assumptions of \S \ref{basicfnassumptions} \emph{for all $s$} and their stalks at $\infty$ are trivial, i.e., they are of rapid decay in a neighborhood of infinity. Thus, Proposition \ref{convergence} continues to hold for elements of the global Schwartz space $\mathcal S(\mathcal Z^s(\adele))$:

\begin{proposition}\label{convergence2}
 For $f_s \in \mathcal S(\mathcal Z^s(\adele))$, any value of $s$, the sum
 $$\sum_{\xi \in k\smallsetminus\{-1,0\}} f_s(\xi)$$
 converges absolutely and is bounded by polynomial seminorms on $\hat\otimes_{v\in T} \mathcal S(\mathcal Z^s_v)$ in vertical strips.
\end{proposition}

\begin{corollary} \label{analyticcont} 
 For $s\mapsto F_s\in \mathcal S(\mathcal W^s(\adele))$ an analytic section, $\KTF(F_s)$  can be analytically continued to all $s\in \CC$ and is bounded by polynomial seminorms in vertical strips. In particular, if $F_s$ is of polynomial growth or rapid decay in a vertical strip, then so is $\KTF(F_s)$, and the value of $\KTF(F_s)$ at any specific $s$ depends only on $F_s$ and not on the section.
\end{corollary}

\begin{proof}
 The section $f_s=\mathcal G^{-1}\left(|\bullet|^{-s-1}F_s\right) \in \mathcal S(\mathcal Z^s(\adele))$ is also analytic by Proposition \ref{variousproperties-RTF}, and bounded by polynomial seminorms on $\mathcal S(\mathcal W^s(\adele))$. Applying to it the series that appears on the left hand side of \eqref{eqPSF} we get an analytic function of $s$, bounded by polynomial seminorms of the original section $s\mapsto F_s$. 

 We now study the other two terms on the left hand side of \eqref{eqPSF}, $\tilde O_0(F_s)$ and $\tilde O_{-1}(F_s)$. The reader should keep in mind the asymptotic behavior of basic functions described in \eqref{basicfn}, as well as Lemma \ref{analyticsections} and Proposition \ref{irregOI} which describe the form that analytic sections and irregular orbital integrals have when $s\to 0$; as remarked, similar descriptions hold for other ``bad'' values of $s$, i.e., when $2s$ is among the values of \eqref{exceptions}.
 
 The term $\tilde O_0(f_s)$ is clearly defined for all values of $s$, since the behavior of basic functions in the neighborhood of zero is identical to that in the baby case; the corresponding irregular orbital integrals for $s=0$ were defined in \eqref{defzerononsplit}, \eqref{defzerosplit}. 
 
 The behavior in a neighborhood of $\xi=-1$ is analogous to the behavior of orbital integrals in the baby case, \S \ref{basicbaby}, with the parameter $t$ replaced by $2s$. Thus, again $\tilde O_{-1}(f_s)$ is defined for every $s$. Therefore, the expression $PS_4(F_s)=PS_1(f_s)$ can be analytically continued to every $s$.
 
 By the polynomial growth of partial abelian $L$-functions on vertical strips, $\tilde O_0(f_s)$ and $\tilde O_{-1}(f_s)$ are bounded by polynomial seminorms of $f_s$, and hence by polynomial seminorms of $F_s$.
\end{proof}

Notice that the continuation is not at all obvious from the definition of $\KTF$, as even the individual terms in the sum do not admit analytic continuation. As we shall see, this corollary is equivalent to the analytic continuation of the weighted sum of $L$-functions $L(\pi,\frac{1}{2}+s)L(\pi\otimes\eta,\frac{1}{2}+s)$ which appear on the spectral side of the Kuznetsov trace formula with non-standard sections.

\part{Spectral analysis.}

\section{Main theorems of spectral decomposition}\label{sec:spectral-main}

In this section we use several results that will be proven in following sections in order to deduce the main conclusions of spectral analysis: 

\begin{enumerate}
 \item For the Kuznetsov trace formula with parameter $s$, the contribution of each cuspidal automorphic representation admits analytic continuation to all $s\in \CC$.
 \item For matching functions, the contributions of each cuspidal automorphic representation to the relative trace formula for the torus and the Kuznetsov trace formula (with $s=0$) coincide.
\end{enumerate}

\subsection{Notation} \label{notation-spectral}

For a finite set $S$ of places (including the Archimedean ones, as per our standard assumptions from \S \ref{ssnotation}), denote by $\mathcal S(\mathcal Z(\adele))_S$, $\mathcal S(\mathcal W^s(\adele))_S$ etc.\ the subspaces of the global Schwartz spaces that we have seen thus far, consisting of vectors of the form:
$$ f^{0,S} \otimes f_S,$$
where $f^{0,S}$ denotes the tensor product over $v\notin S$ of the corresponding basic vectors and $f_S$ belongs to the (completed) tensor product of the local spaces over $v\in S$. For our introductory discussion, let us denote one of these spaces by $\mathcal S_S$, and denote by $\RTF$ the corresponding version of the relative trace formula (which for the spaces $\mathcal S(\mathcal W^s(\adele))$ was denoted by $\KTF$). By abuse of notation, we may consider $f_S$ as an element of $\mathcal S_S$, by tensoring it with $f^{0,S}$.

As explained in \cite[5.4]{SaBE1}, if $v$ is a non-Archimedean place, $f^0_v$ denotes the basic vector of $\mathcal S(\mathcal Z_v)$ or $\mathcal S(\mathcal W^s_v)$ and $h$ belongs to the unramified Hecke algebra $ \mathcal H(G(k_v),G(\mathfrak o_v))$ of $G(\mathfrak o_v)$-biinvariant, compactly supported measures on $G(k_v)$, then it makes sense to write 
$$ h\star f^0_v,$$
for the orbital integrals of the convolution by $h$ of the corresponding basic function ``upstairs''; equivalently, $h$ is considered as an element of component of the Bernstein center where unramified representations live, and the Bernstein center acts naturally on $\mathcal S(\mathcal Z_v)$ and $\mathcal S(\mathcal W^s_v)$.

Thus, the unramified Hecke algebra outside of $S$
\begin{equation}\label{Heckealgebra}\mathcal H^S:=\otimes^\prime_{v\notin S} \mathcal H(G(k_v),G(\mathfrak o_v))
\end{equation}
acts by mapping $\mathcal S_S$ into the corresponding space $\mathcal S(\mathcal Z(\adele))$ or $\mathcal S(\mathcal W^s(\adele))$; hence, we get a functional
\begin{equation}\label{Heckefunctional} \RTF^S: \mathcal H^S\otimes \mathcal S_S \ni h\otimes f_S \mapsto \RTF(h\star f^{0,S} \otimes f_S) \in \CC,
\end{equation}
which in the case of the Kuznetsov formula we will, correspondingly, denote by $\KTF^S$.

By the Satake isomorphism, we have
$$ \mathcal H(G(k_v),G(\mathfrak o_v)) = \CC[\check G\sslash\check G],$$
where $\sslash$ denotes the invariant-theoretic quotient, i.e., $\CC[\check G\sslash\check G] = \CC[\check G]^{\check G}$, and $\check G$ acts here by conjugaction. Hence:
\begin{equation}\mathcal H^S = \CC\left[\prod_{v\notin S} (\check G\sslash\check G)\right],\end{equation}
where by definition regular functions on an infinite product are the restricted tensor product of regular functions on the factors with respect to the constant function $1$. This isomorphism will be denoted: $h\mapsto \hat h$.

Notice that every automorphic representation which is unramified outside of $S$ determines a unique point on $\prod_{v\notin S} (\check G\sslash\check G)$ and, by strong multiplicity one, is determined by it (so we will identify $\pi$ as a point on this product space). 
We let
$$ U^S\subset \prod_{v\notin S} (\check G\sslash\check G)$$
denote the subset corresponding to unitary unramified representations; it is a compact subset with respect to the product Hausdorff topology. We will be writing $\CC[U^S]$ for the restrictions of polynomial functions to $U^S$ -- of course, $U^S$ is Zariski dense, so $\CC[U^S]= \CC\left[\prod_{v\notin S} (\check G\sslash\check G)\right]$.

Our goal is \emph{to express the functionals $\RTF^S$, $\KTF^S$ defined by (\ref{Heckefunctional}) as $(\mathcal S_S)^*$-valued measures on $U^S$}. (In particular, for any fixed $f_S$ this would make them scalar-valued measures on $U^S$, allowing us to extend them from the Hecke algebra -- polynomials on $U^S$ -- to continuous functions on $U^S$, and thus to separate points.)  This is not quite possible in the split case, but it is possible up to a derivative of a delta distribution. There is no clear a priori reason why ``unitary'' should play a role here; however, polynomials satisfy the conditions of the Stone-Weierstrass theorem when restricted to the unitary spectrum -- and this is what allows us to decompose a comparison of trace formulas spectrally. We include a proof of this application of the Stone-Weierstrass theorem for the sake of completeness:

\begin{lemma}\label{SW}
 $\CC[U^S]$ is dense on the space $C(U^S)$ of continuous functions on $U^S$.
\end{lemma}

\begin{proof}
 To apply Stone-Weierstrass, we need to show that as functions on $U^S$ the polynomials are closed under complex conjugation. For every Hecke element $h$ we let $h^*(g) = \overline{h(g^{-1})}$, and then for any unramified unitary representation $\pi$ (considered as a point on $U^S$), if $v\in \pi$ is a unitary unramified vector we have: 
 $$\widehat{h^*}(\pi) = \left< \pi(h^*)v, v\right> = \left< v, \pi(h) v \right> = \left< v, \hat h(\pi) v\right> = \overline{\hat h(\pi)}.$$
\end{proof}

We will denote by $\CC_1, \CC_\eta \in U^S$ the points corresponding to the characters $1, \eta$. For any unramified character $\chi$ of the Borel subgroup $B(\adele^S)$, we will denote by $\chi$ the point on $\prod_{v\notin S} (\check G\sslash\check G)$ corresponding to the unramified principal series representation unitarily induced from $\chi$. In particular, $\delta^{\pm\frac{1}{2}}$ and $\CC_1$ denote the same point on $U^S$ (where $\delta$ denotes the modular character of the Borel). If $\pi\in \prod_{v\notin S} (\check G\sslash\check G)$ is a point corresponding to a character $\chi$ of the Borel, by an \emph{evaluation (of polynomials) of order $d$ at $\pi$} we will mean any linear combination of the value and the first $d-1$ derivatives of the polynomials \emph{along the complex one-parameter family $s\mapsto\chi\delta^s$}; notice that the choice of $\chi$ vs.\ a Weyl group-conjugate character ${^w\chi}$ does not make a difference for this notion. When the point is mentioned many times, the analogous multiplicity is implied, e.g., an ``evaluation at $\CC_1, \CC_\eta$'', when $\eta=1$, is an 
evaluation at $\CC_1$ of order $2$. Notice that an evaluation of order $d$ is also an evaluation of any larger order, in this language.

We will denote by $\hat G^\aut$ the set of automorphic representations of $G(\adele)$ which appear in the Plancherel formula for $L^2([G])$, that is: the points  corresponding to cuspidal representations, unitary idele class characters of the Borel, and residual representations (in this case, idele class characters of $G$). \emph{We consider $\hat G^\aut$ as a subset of $U^S$} -- to be precise, it is a subset of $\lim_\to U^S$ as $S$ becomes larger, but we will freely talk about it as a subset of $U^S$, meaning its intersection with $U^S$. We denote by $\hat G^{\aut}_{\Ram}$ the subset of those which are not residual discrete series (i.e., characters, in the case of $\PGL_2$). These are the representations of ``Ramanujan type'', but of course we do not use their expected temperedness at any point. For $\PGL_2$ they coincide with the generic elements of $\hat G^\aut$, but since we will also use these sets for inner forms, we prefer the name ``Ramanujan'' to ``generic''. 

Hence we have:
$$ \hat G^\aut \supset \hat G^{\aut}_{\Ram} = \hat G^\aut_\cusp \sqcup \hat G^\aut_\Eis,$$
where $\hat G^\aut_\cusp$ denotes cuspidal representations and $\hat G^\aut_\Eis$ denotes principal series unitarily induced from unitary idele class characters of the Borel subgroup. These sets are clearly measurable with respect to the standard Borel structure on $U^S$ (notice that we are not using the Fell topology anywhere, but the topology induced from the space $\prod_{v\notin S} (\check G\sslash \check G)$), since $\hat G^\aut_\cusp$ consists of a countable number of points and $\hat G^\aut_\Eis$ is a countable union of ``lines''. Any reference to ``measures'' on $U^S$ or $\hat G^\aut_\Ram$ will imply regular Borel measures. 

Here is a delicate point involving inner forms: we will also need to consider the sets $\widehat {G^\alpha}^\aut$, for inner forms of $G$ corresponding to nontrivial torsors of the torus $T$ (when it is not split). Of course, as a subset of $U^S$ it belongs to $\widehat G^\aut$ by the global Jacquet--Langlands correspondence. We will therefore never use $\widehat {G^\alpha}^\aut$ explicitly -- only $\widehat G^\aut$ will appear. However, as much as it simplifies notation, one should not assume that the Jacquet--Langlands correspondence is being used: the comparison of trace formulas \emph{shows}, a posteriori, that those elements of $\widehat {G^\alpha}^\aut$ which have nonzero contribution to the torus RTF correspond, as points of $U^S$, to points of $\hat G^\aut$. We point the reader to the proof of Proposition \ref{factorsthrough}, which clarifies the issue.

\subsection{Lifts and relative characters} \label{sslifts} 

Up to now, except for the baby case, we have avoided talking about the spaces ``upstairs'', which give rise to our Schwartz spaces via orbital integrals. 	We now recall from \cite[section 3]{SaBE1} that any element $f\in \mathcal S(\mathcal Z(\adele))$ of the form $f = \otimes_v f_v$ is obtained by the $G(\adele)$-orbital integrals (diagonal action) of an element $\Phi = \otimes \Phi_v$ of the space
\begin{equation}\label{spaceabove} \bigotimes_v' \bigoplus_\alpha \mathcal S(Y^\alpha_v\times Y_v^\alpha),\end{equation}
where, for each place $v$, the sum over $\alpha$ runs over isomorphism classes of torsors $R^\alpha$ of the torus $T$ over $k_v$, and $Y^\alpha$ denotes the ``pure inner form'' $Y^\alpha\simeq T^\alpha\backslash G^\alpha$ of $Y=T\backslash G$, where $Y^\alpha = Y\times^T R^\alpha$, $T^\alpha=\Aut_T R^\alpha \simeq T$ and $G^\alpha=\Aut_G (G\times^T R^\alpha)$. The restricted tensor product is taken with respect to the characteristic functions of the $\mathfrak o_v$-points of $(Y_v \times Y_v)(\mathfrak o_v)$ (at non-Archimedean places).

An element $\Phi$ of \eqref{spaceabove} whose orbital integrals give $f$ as a function on $\mathcal B^\reg_{\mathcal Z}(\adele)$ will be called a \emph{lift} of $f$. (Such an element is non-unique.) In order for this notion to be meaningful, we need to fix Haar measures on the groups $G^\alpha_v$, and we start by fixing any choice of measures on $G_v$ which \emph{factorize the Tamagawa measure on $G(\adele)$}. Since inner twists preserve rational volume forms, this factorization also determines Haar measures on the inner forms $G^\alpha_v$, in such a way that if $R^\alpha$ is a \emph{globally} defined torsor, the product volume on $G^\alpha(\adele)$ is also the Tamagawa measure. We can also talk about local lifts of $f_v\in \mathcal S(\mathcal Z_v)$, but then we will implicitly mean an element of $\bigoplus_\alpha \mathcal S(Y^\alpha_v\times Y_v^\alpha)$ \emph{together} with a choice of Haar measure on $G_v$ (and the induced Haar measures on its inner forms).

Similarly, an element $f_s\in \mathcal S(\mathcal W^s(\adele))$ can be obtained by the orbital integrals of an element $\Phi_s\in\mathcal S^s(\overline X\times X (\adele), \mathcal L_\psi\boxtimes\mathcal L_\psi^{-1})$ in the notation of \cite[\S 6.1]{SaBE1}. Here $X\simeq N\backslash\PGL_2$, $\mathcal L_\psi$ denotes the complex line bundle whose sections are functions on $\PGL_2(\adele)$ such that $f(ng)=\psi(n)f(g)$, and $\mathcal L_\psi^{-1}$ denotes the inverse line bundle (described similarly using the character $\psi^{-1}$). 
The exponent $s$ denotes a certain non-standard space of test functions defined in \cite{SaBE1}. These spaces also have the natural structure of a Fr\'echet bundle over $\CC$, as their local components are Schwartz cosheaves which are isomorphic away from ``infinity'', and they are also isomorphic in a neighborhood of infinity where they are, up to smooth functions, equal to elements of generalized principal series varying analytically with $s$; we leave the details to the reader. The notion of \emph{polynomial families of seminorms} (in bounded vertical strips) is defined in a completely analogous way as for $\mathcal S(\mathcal W^s(\adele))$  (cf.\ Appendix \ref{sec:families}), and one can easily see that polynomial seminorms on $\mathcal S(\mathcal W^s(\adele))$ are the $G(\adele)$-invariant polynomial seminorms on $\mathcal S^s(\overline X\times X (\adele), \mathcal L_\psi\boxtimes\mathcal L_\psi^{-1})$.

The generalization of the notion of a \emph{character} from the adjoint quotient of a group $H/(H-\mbox{conj})$ to a quotient $X_1\times X_2/G$ such as in the relative trace formula can be found in the literature under the names \emph{relative character, spherical character} or \emph{Bessel distribution}. We will use the former term. Unlike the case of the group, there is no canonical normalization of relative characters in general, and they depend on the functionals chosen to define them. 

\begin{definition}
 A \emph{relative character} on $\mathcal S(\mathcal Z_v)$ is a functional which (for any choice of Haar measure on $G_v$) factors through a sequence of morphisms: 
 $$\bigoplus_\alpha \mathcal S(Y^\alpha_v\times Y_v^\alpha) \to \bigoplus_\alpha \pi_\alpha\otimes\tilde \pi_\alpha\to \CC,$$
 where, for each $\alpha$, $\pi_\alpha$ is an irreducible admissible representation of $G^\alpha$ and the last arrow is the canonical pairing. 
 
 An \emph{automorphic relative character} on $\mathcal S(\mathcal Z(\adele))$ is a functional which factors through a sequence of morphisms:
 $$ \bigotimes_v^\prime \bigoplus_\alpha \mathcal S(Y^\alpha_v\times Y_v^\alpha) \to \bigoplus_\beta \pi_\beta\otimes\tilde\pi_\beta\to \CC,$$
 where $\beta$ runs over isomorphism classes\footnote{As before, we use $\beta$ for globally defined torsors and $\beta_v$ for their localizations, while the symbols $\alpha$, $\alpha_v$ are reserved for torsors defined locally. } of \emph{$k$-rational} torsors of $T$, for each $\beta$, $\pi_\beta$ is an irreducible automorphic representation of $G^\beta$, and it is understood that the projection to the $\beta$-summand factors through $\bigotimes^\prime_v \mathcal S(Y^\beta_v\times Y_v^\beta)$ and is equivariant with respect to $G^\beta(\adele)$.
 
 The same definitions (without the need for torsors) apply to the spaces $\mathcal S(\mathcal W^s_v)$, $\mathcal S(\mathcal W(\adele))$.
\end{definition}

Now we come to the definition of global periods and the corresponding relative characters.
We fix a point $\varphi\in \hat G^\aut_\disc$. Let $\beta$ be a $k$-rational $T$-torsor, and let $V^\beta_\varphi$ be the $\varphi$-isotypic subspace of $L^2([G^\beta])$. At this point we do not need to use the Jacquet--Langlands correspondence or strong multiplicity one for inner forms; however, if we do not use strong multiplicity one, we should define ``corresponds to $\varphi$'' as the limit, over finite sets $S$ of places, of the subspaces of $L^2([G^\beta])$ on which the unramified Hecke algebra outside of $S$ acts as $\varphi$. 

Then, with the exception of the case $T=$split and $\varphi=$the trivial representation, the \emph{period integral} over $[T]$ is defined, possibly after regularization (cf.\ Lemma \ref{regT}):
\begin{equation}\label{periodintegral} V^\beta_{\varphi} \ni \phi \mapsto \int_{[T^\beta]}^* \phi(t) dt.
\end{equation}
By Frobenius reciprocity this gives a map:
$$ V^\beta_\varphi \to C^\infty(Y^\beta(\adele))$$
(notice, in this case, that $Y^\beta(\adele)=T^\beta(\adele)\backslash G^\beta(\adele)$),
and conjugate-dually:
\begin{equation}\label{conjugateperiod} \mathcal S(Y^\beta(\adele))\to V^\beta_\varphi.
\end{equation}

We make a remark on measures here: We use measures defined by global volume forms throughout, including to define a pairing between $C^\infty(Y^\beta(\adele))$ and $\mathcal S(Y^\beta(\adele))$, but the map \eqref{conjugateperiod} \emph{depends only on the measure on $G^\beta(\adele)$, not on the measures on $T^\beta(\adele)$ and $Y^\beta(\adele)$} (as long as they are chosen compatibly); thus, it is well-defined even in the case when $T$ is a split torus, in which case global volume forms do not give well-defined measures on $[T]$ and $Y(\adele)$.

Similar definitions hold when we replace the period over $T$ by the period over $N$ against an idele class character $\psi$ or $\psi^{-1}$, and $\mathcal S(Y(\adele))$ by $\mathcal S^s(\overline X(\adele), \mathcal L_\psi)$, in the above notation, provided that the integrals make sense (which they do, as we will see, in Proposition \ref{continuousmap}, when $\Re(s)\gg 0$).

Finally, for $\varphi$ in the continuous spectrum, the same definitions hold (with $\beta$ necessarily the trivial torsor), except that $V_\varphi$  is not a subspace of $L^2([G])$. However, any choice of ``continuous'' Plancherel measure $d\varphi$ will endow the spaces $V_\varphi$ with a unitary structure $\left<\, , \, \right>_\varphi$. Instead of fixing such a Plancherel measure, we will use the canonical product $\left<\, , \, \right>_\varphi d\varphi$.

\begin{definition}
For $\varphi\in \hat G^\aut_\disc$ (except for the trivial representation when $T$ is split) the \emph{period relative character} is the relative character obtained by combining \eqref{conjugateperiod} with its dual, for all classes $\beta$ of $T$-torsors over $k$:
  \begin{equation}\label{torusperiod} \mathcal J_\varphi: \bigotimes_v^\prime \bigoplus_\alpha \mathcal S(Y^\alpha_v\times Y_v^\alpha) \to \bigoplus_\beta V_\varphi^\beta\otimes\overline{V_\varphi^\beta}\to \CC,
  \end{equation}
where, by definition, the projection to the $\beta$-summand  factors through $\bigotimes_v^\prime \mathcal S(Y^\beta_v\times Y^\beta_v)$.
    
  We similarly define (without the use of nontrivial torsors) the \emph{period relative character} on $\mathcal S(\mathcal W^s(\adele))$ (it makes sense, as we will see, when $\Re(s)\gg 0$ and $\varphi\in \hat G^\aut_\Ram$), and denote it by $\mathcal I_\varphi$.
  
  For the continuous spectrum we similarly define measures $\mathcal J_\varphi d\varphi$, $\mathcal I_\varphi d\varphi$ valued in the space of functionals on $\mathcal S(\mathcal Z(\adele))$, resp.\ $\mathcal S(\mathcal W^s(\adele))$.
\end{definition}

We will see an alternate, more direct definition of these relative characters in \S \ref{ssoutline}. 

\subsection{Spectral decomposition: results} \label{ssspectral-results}

\begin{theorem}\label{spectral-RTF}
Consider the case $\mathcal S_S = \mathcal S(\mathcal Z(\adele))_S$. Fix an element $f \in \mathcal S_S$. Then the functional $\RTF^S$, \eqref{Heckefunctional}, on $\mathcal H^S$ is the sum of the following two summands:
 \begin{enumerate}
  \item a finite complex measure $\nu_{f}$ on $\widehat G^\aut_\Ram$;
  \item an evaluation of $\hat h$ at $\CC_1, \CC_\eta$ and (in the split case) $\CC_1$.
 \end{enumerate}
The measure $\nu_{f}$ is equal to the period relative character $\mathcal J_\varphi(f) d\varphi$,
and its norm (in the Banach space of finite measures on $\widehat G^\aut_\Ram$) is bounded by seminorms on $\mathcal S_S$.
\end{theorem}

\begin{theorem}\label{spectral-KTFs}
Consider the case $\mathcal S_S = \mathcal S(\mathcal W^s(\adele))_S$, where $\Re s \gg 0$. Fix an element $f \in \mathcal S_S$. Then the functional $\KTF^S$, s.\ \eqref{Heckefunctional}, on $\mathcal H^S$ is the sum of the following two summands:
 \begin{enumerate}
  \item a finite complex measure $\mu_{f}$ on $\widehat G^\aut_\Ram$;
  \item an evaluation of $\hat h$ at $\delta^{\frac{1}{2}+s}$ and $\eta\cdot \delta^{\frac{1}{2}+s}$.
 \end{enumerate}
The measure $\mu_{f}$ is equal to the period relative character $\mathcal I_\varphi(f) d\varphi$,
and its norm is bounded by polynomial seminorms on $\mathcal S_S$. In particular, for an analytic section of rapid decay in vertical strips $s\mapsto f_s \in \mathcal S(\mathcal W^s(\adele))_S$, the corresponding measures $\mu_{f_s}$ are also of analytic of rapid decay (in the Banach space of finite measures on $\widehat G^\aut_\Ram$).
\end{theorem}

\subsection*{Analytic continuation} We will then use a second ``miracle'', which corresponds to the reflection of the functional equation of the $L$-function $L(\pi,\frac{1}{2}+s) L(\pi\otimes \eta,\frac{1}{2}+s)$ at the level of orbital integrals, to prove:

\begin{theorem}\label{spectral-KTF}
Consider the case $\mathcal S_S = \mathcal S(\mathcal W^s(\adele))_S$, with arbitrary $s$. The expression of $\KTF^S$ as a sum of a measure $\mu_{f}$ on $\widehat G^\aut_\Ram$ bounded by polynomial seminorms and an evaluation, as in Theorem \ref{spectral-KTFs}, holds whenever $\Re(s)\ne \pm\frac{1}{2}$, with the possible modification that the evaluation is at the set of four points: $\delta^{\frac{1}{2}\pm s}$ and $\eta\cdot \delta^{\frac{1}{2}\pm s}$.

Moreover, for an analytic $s\mapsto f_s\in \mathcal S(\mathcal W^s(\adele))_S$ the measure $\mu_{f_s}$ is well-defined as a measure on $\widehat G^\aut_\Ram\smallsetminus \{\delta^{\frac{1}{2} \pm s} , \eta\cdot \delta^{\frac{1}{2}\pm s}\}$ (possibly infinite if $\Re(s) = \pm\frac{1}{2}$), and its restriction to any closed subset not including $\delta^{ \frac{1}{2} \pm s}, \eta\cdot \delta^{\frac{1}{2}\pm s}$ is finite and varies analytically in $s$.
\end{theorem}

The reason that points with $\Re(s)=\pm\frac{1}{2}$ are excluded is that the evaluations and the continuous spectrum of $\widehat G^\aut_\Ram$ are not disjoint in this case, and their contribution cannot be separated. It is easy to see (cf.\ the argument in the proof of Theorem \ref{comparison-spectral}) that the measure $\mu_{f}$ is uniquely defined. 

In particular, for cuspidal representations which are always disjoint from the points $\delta^{\frac{1}{2}\pm s}$ and $\eta\cdot \delta^{\frac{1}{2}\pm s}$, we get the following:
\begin{corollary}\label{analyticcont-cuspidal}
 For every holomorphic section $s\mapsto  f_s\in \mathcal S(\mathcal W^s(\adele))_S$ and every $\varphi \in \hat G^\aut_\cusp$ (the cuspidal automorphic spectrum) the function $s\mapsto \mathcal I_\varphi(f_s)$ (defined, originally, for $\Re(s)\gg 0$) extends to a holomorphic function in the domain of $f_s$.
\end{corollary}

This implies, in particular, the meromorphic continuation of the partial $L$-function $L^S(\pi,\frac{1}{2}+s)L^S(\pi\otimes\eta,\frac{1}{2}+s)$, which is a factor of $\mathcal I_\varphi(f_s)$ (see \eqref{unrW}).

\subsection{Comparison}

The above theorems allow us to spectrally decompose the comparison between the two relative trace formulas:

\begin{theorem}\label{comparison-spectral}
 Fix an element $f\in \mathcal S(\mathcal Z(\adele))_S$ and let $f'= |\bullet|\cdot \mathcal G f\in \mathcal S(\mathcal W^0(\adele))_S$. If $\nu_{f}$, $\mu_{f'}$ are the measures on $\hat G^\aut_\Ram$ obtained by Theorems \ref{spectral-RTF} and \ref{spectral-KTF}, then $\nu_{f} = \mu_{f'}$.
\end{theorem}

\begin{proof}
 By Theorem \ref{spectral-RTF}, the functional $h\mapsto \RTF(h\star f)$ is the sum of the integral against $\nu_{f}$ and an evaluation at $\CC_1$, $\CC_\eta$, $\CC_1$, while the functional $h\mapsto \KTF(h\star f')$ is the sum of the integral against $\mu_{f'}$ and an evaluation at $\CC_1$, $\CC_1$, $\CC_\eta$ and $\CC_\eta$. (Of course, a posteriori it will turn out that some repetitions are superfluous.)

 Essentially by definition (see Corollary \ref{analyticcont}), the two functionals above coincide. In particular, for every $h_1\in \mathcal H^S$ such that  $\hat h_1$ vanishes at $\CC_1$, $\CC_1$, $\CC_\eta$ and $\CC_\eta$ (with the implied multiplicity) we have an equality of functionals on $\mathcal H^S$:
 $$ h\mapsto \RTF (h\star h_1\star f) = \KTF (h\star h_1\star f').$$
 
 Both functionals are represented by measures: 

 $$ \RTF (h\star h_1\star f) = \int \hat h \cdot \hat h_1 \nu_{f},$$
 $$ \KTF (h\star h_1\star f') = \int \hat h \cdot \hat h_1 \mu_{f'}.$$
 
 By the Stone-Weierstrass theorem (see Lemma \ref{SW}), the functions of the form $\hat h$, $h\in \mathcal H^S$, are dense in the space of continuous functions on $U^S$. Therefore, the two measures $\hat h_1 \nu_{f}$ and $\hat h_1 \mu_{f'}$ coincide. Since $h_1$ was arbitrary, with only requirement its vanishing (with multiplicity) at the points $\CC_1$, $\CC_\eta$ which don't belong to $\hat G^\aut_\Ram$, it follows that $\nu_{f}=\mu_{f'}$.
\end{proof}

\subsection{Outline of the proofs}\label{ssoutline}

We will give the proofs of Theorems \ref{spectral-RTF}, \ref{spectral-KTFs}, \ref{spectral-KTF} using results that will be proven in the next section.

We start with Theorem \ref{spectral-RTF}: 

\begin{proof}[Outline of the proof of Theorem \ref{spectral-RTF}] Let $f\in \mathcal S(\mathcal Z(\adele))$ and lift it to an element $\Phi\in \bigotimes_v' \bigoplus_\alpha \mathcal S(Y^\alpha_v\times Y_v^\alpha),$ as in \S \ref{sslifts}. Let $\Sigma\Phi$ denote the corresponding automorphic function on $\sqcup_\beta [G^\beta]^2$ (here $\beta$ runs over torsors defined over $k$) obtained by summation over all $k$-points and Frobenius reciprocity:
$$ \Sigma\Phi(g_1,g_2) = \sum_{(\gamma_1,\gamma_2)\in (Y^\beta\times Y^\beta)(k)} \Phi(\gamma_1 g_1,\gamma_2 g_2), \,\, \mbox{ when } g_1,g_2\in G^\beta(\adele).$$

We will see (Proposition \ref{constantterm-torus} and Corollary \ref{constantterm-torus-cor}) that in the nonsplit case $\Sigma\Phi$ is of rapid decay; more precisely, the map $\Sigma$ is a continuous map from $\mathcal S((Y^\beta\times Y^\beta)(\adele))$ to $\mathcal S([G^\beta]^2)$. In the split case, it is \emph{asymptotically $B$-finite} in both variables with simple exponents $\delta^{\frac{1}{2}}$; more precisely, the map $\Sigma$ is a continuous map from $\mathcal S((Y\times Y)(\adele))$ to a Fr\'echet space $\mathcal S^+_{[\delta^{\frac{1}{2}}]}([G])\hat\otimes \mathcal S^+_{[\delta^{\frac{1}{2}}]}([G])$ of asymptotically $B$-finite functions on $[G]^2$ with the given exponent (in both variables). This notion will be defined in detail in \S \ref{asBfinite}; it means that at the cusp the function is equal to a rapidly decaying function plus an element of a principal series with the given exponent. 

In any case we will see (Proposition \ref{RTF-regularizedip}) that the relative trace formula has the following expression:
$$ \RTF(f) = \left< \Sigma\Phi\right>^*,$$
where $\left< \Sigma\Phi\right>^*$ is the sum over all $\beta$ of the integral over the diagonal copy of $[G^\beta]$, suitably regularized in the split case.  

In the nonsplit case, this immediately implies the stated spectral decomposition of the theorem, by the Plancherel formula for $L^2([G^\beta])$: the function $\Sigma\Phi$ defines a finite signed measure $\nu_\Phi$ on $\hat G^\aut$ such that
$$ \left<\Sigma\Phi\right>^* = \left< \Sigma\Phi\right> = \int_{\hat G^\aut} \nu_\Phi,$$
and, more generally,
\begin{equation}\label{Plancherel-aut}
 \left<h\star \Sigma\Phi\right>^* = \int_{\hat G^\aut} \hat h(\pi) \nu_\Phi(\pi)
\end{equation}
for all $h\in \mathcal H^S$, acting on the first variable. Apart from the evaluations at $\CC_1$ and $\CC_\eta$, other characters of some of the inner forms $[G^\beta]$ do not contribute because they are nontrivial on $T(\adele)$ -- and, clearly, $\Phi$ belongs to a representation induced from the trivial character of $T(\adele)$. Therefore, the integral \eqref{Plancherel-aut} can be split into the sum of evaluations at $\CC_1,\CC_\eta$ and a measure on $\hat G^\aut_\Ram$. (I remark again that, if one does not want to use the Jacquet--Langlands correspondence, one should describe it as a measure on the union of $\widehat{G^\beta}^\aut_\Ram$, considered as subsets of $U^S$; after the comparison of Theorem \ref{comparison-spectral}, this turns out to be a subset of $\hat G^\aut_\Ram$.) Moreover, the total mass of $\nu_\Phi$ is bounded by seminorms on $\oplus_\beta \mathcal S([G^\beta]^2)$, hence by seminorms on $\oplus_\beta \mathcal S((Y^\beta\times Y^\beta)(\adele))$. It is easy to see from the definitions 
that the measure $\nu_\Phi$ coincides with the measure $\mathcal J_\varphi(f) d\varphi$ defined in \S \ref{sslifts}.

In the split case, one needs to extend the Plancherel formula to asymptotically $B$-finite functions with exponents $\delta^\frac{1}{2}$. We will do this in Proposition \ref{Plancherel-Bfinite}, and will describe the topology on the space of such functions in \S \ref{asBfinite}. The rest of the steps are the same, and again the only character contributing will be the trivial one, as others are nontrivial on $T(\adele)$. This completes the proof of Theorem \ref{spectral-RTF}.
\end{proof}

We now outline the proof of Theorem \ref{spectral-KTFs}: 

\begin{proof}[Outline of the proof of Theorem \ref{spectral-KTFs}]
Lift an element (or a section of polynomial growth on vertical strips) $f_s\in \mathcal S(\mathcal W^s(\adele))$ to an element (resp.\ a section of polynomial growth) $\Phi_s\in\mathcal S^s(\overline X\times X (\adele), \mathcal L_\psi\boxtimes\mathcal L_\psi^{-1})$ as in \S \ref{sslifts}.

We will introduce \emph{algebraic height functions} $r$ and $R$ on the adelic points of $X$, resp.\ on $[G]$ in \S \ref{heightfunctions}. We will see (Lemma \ref{belongsCN}) that for every integer $N$ and any $\Re(s)$ large enough, the elements of $\mathcal S^s(\overline X\times X(\adele), \mathcal L_\psi\boxtimes\mathcal L_\psi^{-1})$ belong to the Banach space $C(X\times X(\adele), \mathcal L_\psi\boxtimes\mathcal L_\psi^{-1})_{-N}$ of continuous sections $\Phi$ which satisfy
$$ \sup_{(x_1,x_2)\in X\times X(\adele)} |\Phi(x_1,x_2)| r(x_1)^N r(x_2)^N <\infty.$$
(In fact, recall that in the second variable the elements of $\mathcal S^s(\overline X\times X(\adele), \mathcal L_\psi\boxtimes\mathcal L_\psi^{-1})$ are of rapid decay.) More precisely, the map
\begin{equation}\label{SchwartztoBanach}\mathcal S^s(\overline X\times X(\adele), \mathcal L_\psi\boxtimes\mathcal L_\psi^{-1}) \to C(X\times X(\adele), \mathcal L_\psi\boxtimes\mathcal L_\psi^{-1})_{-N}\end{equation}
is bounded by polynomial seminorms on the former.

As before, we define the map $\Phi\mapsto \Sigma\Phi$ as:
$$ \Sigma\Phi(g_1,g_2) = \sum_{(\gamma_1,\gamma_2)\in (X\times X)(k)} \Phi(\gamma_1g_1,\gamma_2g_2),$$
whenever it converges. We define the space $C([G]^2)_{-N}$ in a completely analogous way, using the height function $R$ on $[G]$. Then  (Proposition \ref{continuousmap}) there is a positive constant $c$ such that for large $N$ the map: $\Phi\mapsto \Sigma\Phi$ represents a continuous morphism:
$$C(X\times X(\adele), \mathcal L_\psi\boxtimes\mathcal L_\psi^{-1})_{-N}\to C([G]^2)_{-cN}.$$

We can now write the Kuznetsov trace formula $\KTF(f_s)$, for $s\gg 0$, as the sum of two terms, the term $\tilde O_0(f_s)$ of \eqref{Poissonsum} and the rest, which we will denote by $\KTF_\emptyset(f_s)$ (this is the classical Kuznetsov trace formula, without the contribution of ``infinity''). As in the torus case, we will see (Proposition \ref{KTF-ip}) that
$$\KTF_\emptyset(f_s)= \left< \Sigma\Phi_s\right>,$$
where the angular brackets again denote the integral over the diagonal copy of $[G]$.

The Plancherel formula for $L^2([G])$ now gives rise, as in the proof of Theorem \ref{spectral-RTF}, to a spectral decomposition of the functional $h\mapsto \KTF(h\star f_s)$. This spectral decomposition only includes generic representations, therefore the corresponding measure $\mu_{f_s}$ of Theorem \ref{spectral-KTFs} is concentrated on $\hat G^\aut_\Ram$. The measure $\mu_{f_s}$ is bounded by the $L^2$-norm of $\Sigma\Phi$, in particular by its $C([G]^2)_{-cN}$-norm;  more precisely, since the measure is invariant under the diagonal $G(\adele)$-action on this space, it is bounded by the norm on the $G(\adele)^\diag$-coinvariants of $C([G]^2)_{-cN}$ (the largerst quotient on which the diagonal $G(\adele)$-action is trivial).

The composition of the inclusion \eqref{SchwartztoBanach} with $\Sigma$ gives rise to a map of coinvariants:
$$ \mathcal S(\mathcal W^s(\adele))\to 
\left(C([G]^2)_{-cN}\right)_{G(\adele)},$$
which is bounded by polynomial seminorms on the former, since those are the $G(\adele)$-invariant polynomial seminorms on $\mathcal S^s(\overline X\times X (\adele), \mathcal L_\psi\boxtimes\mathcal L_\psi^{-1})$, and the map \eqref{SchwartztoBanach} is bounded by polynomial seminorms. Hence, the resulting map:
\begin{equation}\mathcal S(\mathcal W^s(\adele)) \ni f_s \mapsto \mu_{f_s}\in \mathcal M(\hat G^\aut_\Ram)\end{equation}
(finite measures on $\hat G^\aut_\Ram\subset U^S$) is bounded by polynomial seminorms on family on spaces on the left.

Finally, consider the functional: $$\mathcal H^S \ni h\mapsto \tilde O_0(h\star f_s).$$ We will prove in Lemma \ref{irregular-spectral} that it is an evaluation at $\delta^{\frac{1}{2}+2}$, $\eta\cdot\delta^{\frac{1}{2}+2}$.

This completes the proof of Theorem \ref{spectral-KTFs}. 
\end{proof}

We have also established several useful facts towards the proof of Theorem \ref{spectral-KTF}, to which we come now. We first need to discuss the aforementioned ``miracle'', which is a reflection of the functional equation of $L$-functions at the level of orbital integrals: 

\begin{theorem}\label{2ndmiracle}
There is an isomorphism of Fr\'echet spaces 
$$\mathcal T: \mathcal S(\mathcal W^{-s}(k_v)) \xrightarrow{\sim} \mathcal S(\mathcal W^s(k_v))$$
with the following properties:
\begin{itemize}
 \item Polynomial families of seminorms on the right are bounded by polynomial families on the left, and vice versa.
 \item It preserves basic vectors (i.e., $\mathcal T f_{v,-s}^0 = f_{v,s}^0$, whenever they are defined) and is equivariant with respect to the action of the spherical Hecke algebra on those (i.e., $\mathcal T (h\star f_{v,-s}^0) = h\star f_{v,s}^0$ for all $h\in \mathcal H(G(k_v),G(\mathfrak o_v))$).
 
 In particular, $\mathcal T$ defines an isomorphism of global Schwartz spaces
$$\mathcal S(\mathcal W^{-s}(\adele)) \xrightarrow{\sim} \mathcal S(\mathcal W^s(\adele)).$$
 
 \item It preserves the functional $\KTF$.
\end{itemize}

\end{theorem}

This will be proven in \S \ref{ssfe}.

Now let us see how to deduce Theorem \ref{spectral-KTF} from this -- it will be by application of the Phragm\'en-Lindel\"of theorem.

\begin{proof}[Proof of Theorem \ref{spectral-KTF}]
For a bounded vertical strip $V\subset \CC$, let us denote by $\partial V^-$ its left boundary, and by $\partial V^+$ its right one.
Let $V$ be a sufficiently wide vertical strip which is symmetric around $s=0$. Let $V\ni s\mapsto f_s\in\mathcal S(\mathcal W^s(\adele))_S$ be an analytic section of rapid decay on $V$, and lift it to a section $s\mapsto \Phi_s$ of $\mathcal S^s(\overline X\times X(\adele), \mathcal L_\psi\boxtimes\mathcal L_\psi^{-1})$. Given $f\in\mathcal S(\mathcal W^s(\adele))_S$ for a specific value of $s$ (as in the Theorem), we can make sure that $f_s=f$ for that value of $s$.

For every $h\in \mathcal H^S$, the function $s\mapsto \KTF(h\star f_s)$ is holomorphic of finite order on $V$ (Corollary \ref{analyticcont}). 

Now fix a holomorphic section $s\mapsto h_s\in \mathcal H^S$ such that:
\begin{itemize}
 \item the evaluations of $\hat h_s$ at the points $\delta^{\frac{1}{2}\pm s}$ and $\eta\delta^{\frac{1}{2} \pm s}$ (to the corresponding order, if these points coincide) vanish;
 \item $\left\Vert \hat h_s \right\Vert_{L^\infty(U^S)}$ is bounded on $V$.
\end{itemize}
By ``holomorphic section'' we mean that, as measures on $G(\adele^S)$, they are supported on the same finite set of double $\prod_{v\notin S} G(\mathfrak o_v)$-orbits, where they vary analytically. This is equivalent to saying that $\hat h_s$ is holomorphic into the space of polynomials on $U^S$ of degree bounded by a fixed number. It is clear that such sections exist; let $f'_s=h_s\star f_s$. 

For any fixed $s\in \partial V^+$, consider the functional $h\mapsto \KTF(h\star f'_s)$ on $\mathcal H^S$. By Theorem \ref{spectral-KTFs} it is represented by a finite measure $\mu_{f_s'}$ on $\hat G^\aut_\Ram$. 
Moreover, this measure is of rapid decay on $\partial V^+$ since $f_s$ is and $\hat h_s$ is sup-bounded. In particular, for every fixed $h$ the holomorphic function $\KTF(h\star f'_s)$ is of rapid decay on $\partial V^+$.

Similarly, given $s\in \partial V^-$, the functional $h\mapsto \KTF(h\star f'_s)$ is, by Theorem \ref{2ndmiracle}, equal to the functional  $h\mapsto \KTF(h\star \mathcal T f'_s)$; recall that $\mathcal T f'_s \in \mathcal S(\mathcal W^{-s}(\adele))$ is of rapid decay on $V$, since $\mathcal T$ preserves this property. Therefore, by the same argument, when $s \in \partial V^-$ it is represented by a finite measure $\hat h_s \cdot \mu_{f'_s}$ on $\hat G^\aut_\Ram$, and for every fixed $h$ the holomorphic function $\KTF(h\star f'_s)$ is of rapid decay on $\partial V^-$.

Therefore, by the Phragm\'en-Lindel\"of principle, we have
$$ \left|\KTF(h\star f'_s)\right| \le  \sup_{t\in \partial V^+ \cup \partial V^-} \left\Vert  \hat h \cdot \mu_{f'_t}\right\Vert \le $$
$$ \le \Vert \hat h\Vert_{L^\infty(U^S)} \cdot \sup_{t\in \partial V^+ \cup \partial V^-} \left\Vert \mu_{f'_t}\right\Vert$$
for all $s \in V$. This shows that \emph{for every $s$} the functionals:
$$\mathcal H^S\ni h \mapsto \KTF(h\star f'_s)$$
can be continuously extended, by the Stone-Weierstrass theorem (Lemma \ref{SW}), to the space of continuous functions on $U^S$, and hence are represented by a family $\mu_{f'_s}$ of measures which is \emph{weak-star} holomorphic, i.e., for every $F\in C(U^S)$ the function $s\mapsto \int F \mu_{f'_s}$ is analytic.

It follows that this family is \emph{strongly} analytic into the Banach space of finite measures on $U^S$. Since the standard references on vector-valued holomorphic functions usually mention weak (not weak-star) holomorphy as the assumption for this conclusion, we revisit the steps of the proof of strong holomorphy to verify that they apply here; for simplicity, let us denote $v_s= \mu_{f'_s}$, and the integral above by $v_s(F)$.
 
 The basic property is strong continuity of the section $s\mapsto v_s$. This follows by observing that $\frac{1}{t} (v_{s+t}(F)-v_s(F))$ can be written as a Cauchy integral and bounded, for all $0<|t|<r$ by the maximum of $r^{-1} |v_z(F)|$ on the circle of radius $2r$ around $s$, where $r$ is a small positive number. The uniform boundedness principle then implies that the collection of vectors
 $$\frac{1}{t} (v_{s+t}-v_s), \,\, 0<|t|<r,$$
 is strongly bounded, and hence $v$ is continuous at $s$. Then the vector-valued Cauchy integral
 $$ \frac{1}{2\pi i} \int (z-s)^{-1} f(z) dz $$
 (along a small circle around $s$) makes sense and represents $v_s(F)$ for every $F$, therefore is equal to $v_s$, and $v$ is holomorphic.
 
By strong analyticity, the mass of every measurable subset of $U^S$ varies analytically in $s$, in particular the mass of $\mu_{f'_s}$ is concentrated on $\hat G^\aut_\Ram$, for every $s$. 

Since $\hat h_s$ was arbitrary with only requirement its vanishing at the points $\delta^{\frac{1}{2} \pm s}$ and $\eta\delta^{\frac{1}{2} \pm s}$, and given that, when $\Re s\ne \pm\frac{1}{2}$, those points do not meet $\hat G^\aut_\Ram$, it follows that for $\Re s\ne \pm \frac{1}{2}$ the measure $\mu_{f_s}=\hat h_s ^{-1} \mu_{f'_s}$ is also well-defined (and, clearly, independent of the choice of $h_s$).
In fact, we can think of $\mu_{f_s}$ as a (not necessarily finite) measure on $\hat G^\aut_\Ram \smallsetminus \{\delta^{\frac{1}{2} \pm s}, \eta\delta^{\frac{1}{2} \pm s}\}$ for \emph{every} $s$, and its restriction to a subset which doesn't contain the points $\delta^{\frac{1}{2} \pm s}$ and $\eta\delta^{\frac{1}{2} \pm s}$ in its closure is (finite and) locally analytic in $s$ (and finite).

Now consider the functional $h\mapsto \KTF(h\star f_s) - \int_{\hat G^\aut_\Ram} \hat h \mu_{f_s}$ on $\mathcal H^S$, when $\Re s\ne \pm\frac{1}{2}$; it is necessarily an evaluation at the points $\delta^{\frac{1}{2} \pm s}$ and $\eta\delta^{\frac{1}{2} \pm s}$. Indeed, if $\hat h$ vanishes at those points (to the corresponding order, if some of those points coincide), then $h = h_s$ for some $h_s$ as above, and the functional is zero on $h$.

This completes the proof of Theorem \ref{spectral-KTF}.
\end{proof}

\section{Completion of proofs} \label{sec:spectral-proofs}

In this section I prove all the auxiliary results used in the previous section.

\subsection{Height functions} \label{heightfunctions}

For any reductive group $G$ over a local field $F$, if we fix a faithful algebraic representation: $G\to \GL_N$, we get a natural algebraic height function $\Vert \bullet \Vert $ on $G(F)$ by pulling back the maximum of the operator norms of $g$ and $g^{-1}$ with respect to the norm $(r_1,\dots,r_N) \mapsto \max_i |r_i|$ on $F^n$. If $G$ is defined over a global field $k$, we take this representation to be defined over $k$, and then we can define the height function $\Vert\bullet\Vert$ on $G(\adele)$ (and, by restriction, on $G(k)$) as the product over all places of the local height functions; this product is finite for any element. 

These height functions on the group are a special case of the following (if we replace $X$ by $G$):

Let $X$ be a homogeneous, strongly\footnote{Strongly quasi-affine means that the canonical morphism: $X\to \bar X^\aff:=\spec k[X]$ is an open immersion.} quasi-affine $G$-variety over a global field $k$.
Choose finite sets $\{f_i\}_i$ of generators of $k[X]$ and $\{h_j\}_j$ of generators for the ideal of $\bar X^\aff\smallsetminus X$ and set, at every place $v$ and for $x\in X(k_v)$: 
$$r_v(x) = \max\left\{1,|f_i(x)|_v,(\max\{|h_j(x)|_v\})^{-1}\right\},$$
and globally for $x\in X(\adele)$:
$$ r(x) = \prod_v r_v(x)$$
(almost all factors are equal to one).

We call $r$ an ``algebraic height function'', or simply a height function, on $X$. The following are true:

\begin{lemma} \label{comparisons}
 \begin{enumerate}
  \item Any two height functions $r_1,r_2$ defined as above are polynomially equivalent in the following sense: there are positive constants $c_1,c_2, m_1, m_2>0$ such that:
  $$ c_1 r_1(x)^{m_1} \le r_2(x) \le c_2 r_1(x)^{m_2} \mbox{ for all } x\in X(\adele).$$
  \item There is a positive number $m$ and a constant $c$ such that for every $x\in X(\adele)$ and $g\in G(\adele)$ we have:
\begin{equation} c^{-1} \Vert g\Vert^{-m} \le \frac{r(x\cdot g)}{r(x)} \le c \Vert g\Vert^m.\end{equation}
  \item There is a positive number $M$ and a constant $c$ such that:
\begin{equation}\label{rationalpoints}
 \#\{\xi\in X(k)|r(\xi)<T\} \le c T^M, \,\,\text{for all }T\ge 1.
\end{equation}
  \end{enumerate}
\end{lemma}

\begin{proof}
 For the first statement (it is of course enough to prove one inequality), let us first consider the affine case. If $f_i$ and $F_i$ are generators, then $F_i = \sum_\alpha c_i^\alpha f_\alpha$, where $\alpha$ is a multiindex denoting a product, and if $d$ is the highest degree of the multiindices appearing we get: $|F_j| \ll \max\{|f_i|, |f_i|^d\} \ll \max\{1,|f_i|^d\}$, with the implicit constant being equal to $1$ at almost all places. 
 
 For the quasiaffine case, if we denote by $h_j$ and $H_j$ the generators of the ideal, then we similarly have: $\max|H_i|\le c\cdot \max\{|h_i|,|h_i|^d\}$ for some positive integral power $d$ and some constant $c\ge 1$ which is equal to $1$ at almost every place, hence $\max\{|h_i|\}^{-1} \le c \cdot \max\{\max\{|H_i|\}^{-1},\max\{|H_i|\}^{-\frac{1}{d}}\}\le c\cdot \max\{1, \max\{|H_i|\}^{-1}\}$. This proves the first claim.

 For the second, it is again enough to prove one side of the equality. The fact that $|f_i(x\cdot g)|_v  \ll \Vert g\Vert_v^m r_v(x)$, where $f_i$ is as in the definition of $r(x)$, follows from the corresponding statement for a representation of $G$: one can embed $X$ into the space of a $G$-representation and take the $f_i$'s to be coordinate functions. The power $M$ is uniform in $v$, and the implied constant can be taken to be $1$ at almost every place. To prove a similar bound for $(\max\{|h_j(x\cdot g)|_v\})^{-1}$, we may without loss of generality assume that the $h_j$'s span a $G$-stable vector subspace $W$ of $k[X]$, and that $\Vert \bullet \Vert$ was defined using the representation of $G$ on $W$, endowed with this basis (we will assume that it is faithful, because if it isn't the operator norm it defines is bounded by the operator norm of a faithful one). Then
 $$\max\{|h_j(x\cdot g g^{-1})|_v\} \le \Vert g\Vert_v \cdot \max\{|h_j(x g)|_v\},$$
 hence indeed $\max\{|h_j(x g)|_v\}^{-1} \le \Vert g\Vert_v  r(x)$.

 The third follows from the analogous statement for affine space, by embedding $X$ again into the space of a representation of $G$.
\end{proof}

Finally, recall the notation: $[G]=G(k)\backslash G(\adele)$, $[G]_\emptyset = A(k)N(\adele)\backslash G(\adele)$.
We also define natural height functions on $[G]$ as follows: having fixed $\Vert\bullet\Vert$ on $G(\adele)$, we set:
$$ R([g])= \inf_{\gamma\in G(k)} \Vert \gamma g\Vert.$$
This is polynomially equivalent (in the same sense as in Lemma \ref{comparisons}) to the usual height function on a Siegel domain $\mathscr S$ (s.\ below): 
$$\mathscr S\ni n a k \mapsto \Vert a\Vert.$$
Similarly on $[G]_\emptyset = A(k)N(\adele)\backslash G(\adele)$ we denote: 
$$ R([g])= \inf_{\gamma\in A(k)N(\adele)} \Vert \gamma g\Vert.$$

\subsection{Asymptotically $B$-finite automorphic functions and regularized inner product} \label{asBfinite}

\subsubsection{Asymptotically $B$-finite automorphic functions} 

Recall that the constant term of a function $\phi$ on $[G]$ is the function
\begin{equation}
 \phi_N(g):= \int_{[N]} \phi(ng) dn
\end{equation}
on $[G]_\emptyset$. 

Recall that a Siegel domain  \label{Siegeldomain} is a closed subset $\mathscr S = N_0 A_{t_0} K$ of $G(\adele)$, where $B=NA$ is a Borel subgroup, $N_0\subset N(k_\infty)$ is compact, $A_{t_0}$ is a certain closed subset of $A(k_\infty)$ and $K\subset G(\adele)$ is compact; the data is chosen so that the map: $\mathscr S\to [G]$ is surjective and proper, and the map $\mathscr S\to [G]_\emptyset$ is proper onto a neighborhood of the cusp. 
(To fix a geometric picture, recall that $N\backslash \SL_2 \simeq \mathbb A^2\smallsetminus\{0\}$, and that the cusp in $[G]_\emptyset$ is the image of $0\in \mathbb A^2$ under the map: $N\backslash \SL_2(\adele)\to [G]_\emptyset$.) 

The preimage in $\mathscr S$ of any neighborhood of the cusp (in $[G]$ or $[G]_\emptyset$) will be called a ``Siegel neighborhood'' (of the cusp); the reader should mark this distinction, as a Siegel domain is supposed to surject to $[G]$, while a Siegel neighborhood is not. For a sufficiently small Siegel neighborhood $\mathscr S'$, we have the property:
\begin{equation}\label{Siegelproperty}\gamma\in G(k), g_1,g_2\in \mathscr S', \gamma g_1=g_2 \Rightarrow \gamma\in B(k).\end{equation}

Pullbacks of functions (on $[G]$ or $[G]_\emptyset$) to $\mathscr S$ will be denoted by the same letter, without any notation for the pullback. 

Following standard language, we will call Schwartz functions on $[G]$ or $[G]_\emptyset$ ``rapidly decaying functions'' (denoted $\mathcal S([G])$, resp.\ $\mathcal S([G]_\emptyset)$). The notion of ``Schwartz space $\mathcal S({\mathscr S})$ of a Siegel neighborhood ${\mathscr S}$'' also makes sense, by considering the Siegel neighborhood as a closed semialgebraic subset of $G(\adele)$ and defining $\mathcal S({\mathscr S})$ as the stalk of $\mathcal S(G(\adele))$ over this subset, in the language of \cite[Appendix B]{SaBE1} (i.e.\ as a quotient, by restriction of functions).

We will for brevity say ``automorphic function'' for a smooth automorphic function of uniformly moderate growth. That is, a function $\phi$ on either $[G]$ or $[G]_\emptyset$ will be called an automorphic function if it is locally constant outside of a finite set $T$ of places, and lives, for some $r$, in the space $V_r^\infty$ of smooth vectors under the action of $G(k_T)$ of the Banach space $V_r$ of functions on $[G]$ defined by the norm: $\sup_{g\in [G]} |\phi(g)| R(g)^r$, where $R(g)$ denotes an algebraic height function on either $[G]$ or $[G]_\emptyset$ -- s.\ Appendix \ref{sec:Frepresentations} for a discussion of smooth vectors and representations.
Thus, if the function is locally constant at all non-Archimedean places, this is the usual notion of a function of ``uniformly moderate growth together with its derivatives'' -- the reader can restrict their attention to this case, and we will not be explaining easy extensions of well-known theorems to the ``almost smooth'' case. The most important of those theorems is the approximation of an automorphic function by its constant term:

\begin{theorem}[{\cite[Corollaries I.2.8, I.2.11]{MW}}] \label{constantterm}
 In a Siegel neighborhood ${\mathscr S}$, and for every $\phi\in V_r^\infty$ (where $V_r$ is as above) the difference $\phi-\phi_N$ is a rapidly decaying function on ${\mathscr S}$, i.e., an element of $\mathcal S({\mathscr S})$, which can be bounded in terms of seminorms of $V_r^\infty$. In other words, the map: $V_r^\infty \ni \phi\mapsto \phi-\phi_N \in \mathcal S({\mathscr S})$ is continuous.
\end{theorem}

We consider the normalized left action of $[A]$ on functions on $[G]_\emptyset$, i.e., the one that preserves inner products:
\begin{equation}\label{normalized-global} a\cdot \phi(g) = \delta^{-\frac{1}{2}}(a)\phi(ag).
\end{equation}

We say that an automorphic function $\phi$ on $[G]_\emptyset$ is $B$-finite if it is a sum of generalized eigenfunctions under the $[A]$-action. The set of \emph{exponents} of $\phi$ is the multiset of its $[A]$-characters (under the above normalized action). For an idele class character $\chi$ of $B(\adele)$, we let $\Re(\chi)=\Re(s)$, where $s \in \CC$ is such that $\chi \delta^{-s}$ is unitary.

 We say that it an automorphic function $\phi$ on $[G]_\emptyset$ is \emph{asymptotically $B$-finite} if  
 in a neighborhood of the cusp it coincides, up to a rapidly decaying function, with a $B$-finite function. The \emph{exponents} of $\phi$ are the exponents of the latter.

We say that an automorphic function $\phi$ on $[G]$ is \emph{asymptotically $B$-finite} if on a Siegel neighborhood it coincides, up to a rapidly decaying function, with a $B$-finite function on $[G]_\emptyset$; its exponents are the exponents of the latter.  
 For a multiset $M$ of exponents, we will be denoting by $\mathcal S^+_{M} ([G])$ the corresponding space of asymptotically $B$-finite functions; approximation by the constant term (Theorem \ref{constantterm}) implies that $\mathcal S^+_{M} ([G])$ is $G(\adele)$-invariant. 

If we fix the level, i.e., if we fix a compact open subgroup $J$ of the finite adeles of $G$ (or, more generally, if we fix it outside of a finite number $S$ of places including the Archimedean ones, so that $J$ is a subgroup of the adeles of $G$ outside of $S$), the space $\mathcal S^+_M ([G])^J$ has a natural Fr\'echet space structure and, more precisely, these spaces form a Fr\'echet bundle over the complex manifold of ordered sets $(\chi_i)_i$ of exponents (possibly with multiplicities). For simplicity, we only describe this structure along one-dimensional families of the form $(\chi_1, \dots, \chi_r)\cdot \delta^s$, where $s$ varies in $\CC$, and we implicitly fix the level throughout. For a given ordered $r$-tuple $(\chi_1,\dots, \chi_r)$, denote by $[\chi_i]_{i=1}^r$ (or, for simplicity, by $[\chi_i]$) the corresponding unordered multiset, and let $I([\chi_i]_{i=1}^r)$ denote the generalized principal series\footnote{The notion of ``generalized'' principal series here is defined with respect to the $[A]$-action, not the $A(\adele)$-action. That is, its elements are smooth functions on $[G]_\emptyset$, and generalized $[A]$-eigenfunctions with a given multiset of exponents. If the multiset contains a single element $\chi$, of course, this is nothing but the principal series representation $I(\chi)$.} with the given multiset of exponents. The association $(\chi_1,\dots,\chi_r)\mapsto I([\chi_i]_{i=1}^r)$ naturally forms a Fr\'echet bundle over the space of exponents, which can be trivialized by restricting to an orbit of a good maximal compact subgroup $K$ of $G(\adele)$.  
There is a natural map: $\mathcal S^+_{[\chi_i]} ([G])\ni \phi \mapsto \phi_{[\chi_i]} \in I([\chi_i])$ obtained by extracting from the constant term the (uniquely defined) element of $I([\chi_i])$ to which it is asymptotically equal in a neighborhood of the cusp, and this gives an embedding
\begin{equation}\label{Bfinitetopology} \mathcal S^+_{[\chi_i]} ([G])\to \mathcal S(\mathscr S) \oplus I([\chi_i]),
\end{equation}
where $\mathscr S$ is a Siegel domain and the map to the first summand is $\phi\mapsto \phi-\phi_{[\chi_i]}$. We postulate that this embedding is closed, thus defining the aforementioned ``natural'' Fr\'echet space structure on $\mathcal S^+_{[\chi_i]} ([G])$.

The image of this map is the subspace of elements $(\phi, f)$ such that $\phi+f$ descends to a function on $[G]$, and if we fix an element $f_\delta\in I(\delta)$ which is positive as a function, then the map: $(\phi,f)\mapsto (\phi \cdot f_\delta^s, f \cdot f_\delta^s)$ identifies those spaces along the family, i.e., we do indeed have a Fr\'echet bundle.

\subsubsection{Regularized integral}

We will define a regularized integral
$$\int^*_{[G]} \phi(g) dg$$
for asymptotically $B$-finite functions on $[G]$ whose exponents do not include $\delta^\frac{1}{2}$. It will, in fact, be a meromorphic family of $G(\adele)$-invariant functionals on the Fr\'echet bundle consisting of the spaces  $\mathcal S^+_{[\chi_i\delta^s]} ([G])$ as $s$ varies. For $\Re(s)\ll 0 $ it is simply the integral of $\phi$ over $[G]$. 

\begin{lemma}
For every holomorphic section:
$$ s\mapsto \phi_s \in \mathcal S^+_{[\chi_i\delta^s]} ([G])$$
the integral
$$ \int_{[G]} \phi_s(g) dg,$$
which converges for $\Re(s)\ll 0$, admits meromorphic continuation, with poles of order at most equal to the multiplicity of $\delta^\frac{1}{2}$ among the exponents $\chi_i\delta^s$, and is bounded by polynomial seminorms on the spaces $\mathcal S^+_{[\chi_i\delta^s]}$.
\end{lemma}

This meromorphically continued integral will be denoted by $\int^*_{[G]} \phi_s(g) dg$. The notion of polynomial seminorms (always: in bounded vertical strips) makes sense since the underlying Fr\'echet spaces have been identified as above for all values of $s$. 

\begin{proof}
 It suffices to choose a Siegel neighborhood $\mathscr S'$ and show the same statement for the integral restricted to the image of $\mathscr S'$ in $[G]$ (let us denote it by $[\mathscr S']$). Then we may replace $\phi_s$ by the eigen-part $\phi_{s,[\chi_i\delta^s]}$ of its constant term (notation as around \eqref{Bfinitetopology})  on $\mathscr S'$, because their difference is of rapid decay, locally uniformly in $\Re(s)$ (Theorem \ref{constantterm}), and bounded by ``constant'' seminorms on $\mathcal S^+_{[\chi_i\delta^s]}$ (by definition of the Fr\'echet bundle structure). By \eqref{Siegelproperty}, we may think of $[\mathscr S']$ as a neighborhood of the cusp in $B(k)\backslash G(\adele)$, and the integral of $\phi_{s,[\chi_i\delta^s]}$ there clearly has the properties of the lemma. 
\end{proof}

We use this integral now to define a \emph{regularized inner product} of two asymptotically $B$-finite functions on $[G]$: 
\begin{equation}
 \left< \phi_1,\phi_2\right>^*_{[G]} := \int^*_{[G]} \phi_1(g) \phi_2(g) dg.
\end{equation}
It makes sense as long as no pair of exponents $(\chi_1,\chi_2)$ of $\phi_1$ and $\phi_2$, respectively, satisfies: $\chi_1\cdot \chi_2 =1$, and it is $G(\adele)$-invariant. Notice that our pairings $\left< \,\, \right>$ will all be bilinear, instead of hermitian; in particular, for a unitary representation $\pi$ we will denote by $\left< \, , \, \right>_\pi$ the bilinear pairing between $\pi$ and $\bar\pi$.

\subsubsection{Plancherel decomposition of the regularized inner product}

We will now obtain a spectral decomposition of this regularized inner product, similar to the Plancherel decomposition for the inner product of two $L^2$-functions. For simplicity, we restrict ourselves from now on to the case of interest, i.e., elements of the space $\mathcal S^+_{\delta^\frac{1}{2}}([G])$, (those which only have a simple exponent equal to $\delta^{\frac{1}{2}}$) -- in particular, the exponent of the product of two such functions is $\delta^\frac{3}{2}$, so their regularized inner product is defined.

To formulate the ``Plancherel decomposition'', recall first the actual Plan\-cherel decomposition for the inner product of, say, two rapidly decaying functions on $[G]$:
\begin{equation}\label{Planchauto} \left<\phi_1,\phi_2\right>_{([G])} = \int_{\hat G^\aut_\Ram} \left< \hat\phi_1(\pi),\hat\phi_2(\bar\pi)\right>_\pi d\pi + \sum_{\chi\in \widehat{G^\ab}^\aut} \hat\phi_1(\chi)\hat\phi_2(\chi^{-1}).
\end{equation}

Here the notation is as follows: First of all, for an automorphic character $\chi$ of $G$, we set:
$$ \hat \phi(\chi) = (\Vol[G])^{-\frac{1}{2}} \int_{[G]}\phi(g)\chi^{-1}(g) dg.$$

Recall that elements of $\hat G^\aut_\Ram$ are by definition unitary representations. Fixing a Plancherel measure $d\pi$ on this set (of course, it is natural to choose it to be counting measure on the cuspidal spectrum, but this will not make a difference for us here) gives rise to morphisms: $\phi\mapsto \hat\phi(\pi) \in \pi$, unique up to scalars of norm one, which make the above formula hold if we set $\hat\phi(\bar\pi) = \hat{\bar\phi}(\pi)$ (considered as an element of $\bar \pi$). These morphisms can be explicated in terms of inner products with cusp forms or unitary Eisenstein series: for $v \in \pi$ we have 
$$\left<\phi(\pi),\overline{v}\right>_\pi = \int_{[G]}\phi(g) \overline{\ell(v)(g)} dg,$$
where $\ell:\pi\to C^\infty([G])$ is some embedding depending on the choice of Plancherel measure. The automorphic forms $\ell(v)$ are either cusp forms (and hence of rapid decay) or unitary Eisenstein series (and hence with normalized exponents of absolute value $1$). 

This leads to the following observation:
\emph{Replacing inner products by the regularized inner products that we defined before, the maps $\phi\mapsto \hat\phi(\pi)\in \pi\in \hat G^\aut_\Ram$ and $\phi\mapsto \hat \phi(\chi)$ make sense for $\phi\in \mathcal S^+_{\delta^\frac{1}{2}}([G])$, except when $\chi=1$.}
This is easily seen by checking the exponents of the inner products. 

We are now ready to extend the Plancherel decomposition to $\mathcal S^+_{\delta^\frac{1}{2}}([G])$:

\begin{proposition}\label{Plancherel-Bfinite} 
Let $\phi_1,\phi_2 \in \mathcal S_{\delta^\frac{1}{2}}^+([G])$, unramified outside a finite set of places $S$. Then the integral:
\begin{equation}\label{PBfeq}\int_{\hat G^\aut_\Ram} \left<\hat \phi_1(\pi),\hat\phi_2(\bar\pi)\right>_\pi d\pi
\end{equation}
is absolutely convergent, and the difference:
\begin{eqnarray}\label{difference2} \left< h\star \phi_1,\phi_2\right>^*_{[G]} - \int_{\hat G^\aut_\Ram} \left<\hat h(\pi) \cdot \hat \phi_1(\pi),\hat\phi_2(\bar\pi)\right>_\pi d\pi \nonumber\\
 - \sum_{\chi\in \widehat{G^\ab}^\aut, \chi\ne 1} \hat h(\chi) \hat \phi_1(\chi)\hat \phi_2(\chi^{-1}),\end{eqnarray}
as a functional of $h\in \mathcal H^S$,
is an evaluation of order $3$ (s.\ \S \ref{notation-spectral}) at $\CC_1$.
\end{proposition}

In the above expression the representations and characters which are ramified outside of $S$ (for which the Satake transform $\hat h(\pi)$, $\hat h(\chi)$ does not make sense) can, of course, be ignored since their contribution is zero.

\begin{proof} 

Let us denote by $L$ the difference \eqref{difference2}, considered as a functional of $h\in \mathcal H^S$. 

By continuity of both sides, it is enough to prove the statement for $\phi_1$ in the dense subspace of functions which can be represented as $h_1\star \phi_1'$, where $\phi_1'$ is asymptotically $B$-finite with \emph{double} exponent $\delta^\frac{1}{2}$, and $\hat h_1 (\CC_1)=0$. By invariance:
$$\left< \phi_1,\phi_2\right>^* = \left<\phi_1',h_1^\vee\star \phi_2\right>^*.$$
So, it is enough to show that the proposition for $\phi_1$ replaced by $\phi_1'$ and $\phi_2$ replaced by $\phi_2':=h_1^\vee\star\phi_2$, which is of rapic decay (i.e., $\phi_2' \in \mathcal S([G])$). The definitions still make sense, and obviously
$$\left<\hat h(\pi) \cdot \hat\phi'_1(\pi),\hat\phi'_2(\bar\pi)\right>_\pi = \left<\hat h(\pi) \cdot \hat \phi_1(\pi),\hat\phi_2(\bar\pi)\right>_\pi,$$
again by invariance (and similarly for the Gr\"ossencharacters).

We claim that for $\phi\in \mathcal S^+_{[\delta^\frac{1}{2}, \delta^\frac{1}{2}]}([G])$ and $h\in \mathcal H^S$ such that the evaluations of order $3$ of $\hat h$ at $\CC_1$ vanish, the function $h\star \phi$ belongs to $\mathcal S([G])$ and has integral zero over $[G]$. Indeed, the order of the exponent $\delta^\frac{1}{2}$ is equal to the maximal order at $r=0$, minus one, of the pole of an integral:
$$ \int_{[G]}^* \phi(g) \mathcal E(f_r) (g) dg,$$
where $r\mapsto f_r\in I(\delta^{r-\frac{1}{2}})$ is an analytic section (which can be assumed unramified outside of $S$) and $\mathcal E$ denotes its Eisenstein series. In particular, for $\phi\in \mathcal S^+_{[\delta^\frac{1}{2}, \delta^\frac{1}{2}]}([G])$ this order is at most $3$, and therefore if $\hat h$ vanishes to order $2$ at $\CC_1$ then
$$ \int_{[G]}^* h\star \phi(g) \mathcal E(f_r) (g) dg = \int_{[G]}^* \phi(g) \mathcal E(h^\vee\star f_r) (g) dg = $$
$$ = \hat h(\delta^{\frac{1}{2}-r}) \int_{[G]}^* \phi(g) \mathcal E(f_r) (g) dg $$
has at most a simple pole. 

The contribution of the trivial representation to $h\star \phi$ is the residue of such an integral when $f_r$ is chosen so that $\mathcal E(f_r)$ has constant residue $1$, therefore if, in addition, $\hat h$ vanishes of order $3$ then this contribution is zero. This proves the claim.

This being so, we can apply the usual Plancherel decomposition \eqref{Planchauto} to the inner product $\left< h\star\phi_1', \phi_2'\right>$, when $\phi_1', \phi_2'$ are as above and $\hat h$ vanishes to triple order at $\CC_1$, and we see that the integral in the definition of $L(h)$ is in this case absolutely convergent, and equal to zero. Such an $h$ can be chosen so that $\hat h(\pi)$ is bounded below on $\hat G^\aut_\Ram$, and therefore \eqref{PBfeq} is absolutely convergent. Thus, $L(h)$ is well-defined for every $h$, and its vanishing on those $h$ that vanish of order $3$ at $\CC_1$ shows that it is an evaluation of order $3$ at $\CC_1$.
\end{proof}

\subsection{The torus RTF}

For this subsection, for a $T$-torsor $\alpha$ over $k$, let $Y^\alpha=T^\alpha\backslash G^\alpha$ and $\Sigma: \mathcal S(Y^\alpha(\adele))\to C^\infty([G^\alpha])$ be as before:
$$ \Sigma\Phi(g) = \sum_{\gamma\in Y^\alpha(k)} \Phi(\gamma g)$$
(or the analogous functional in two variables, i.e., from $S(Y^\alpha\times Y^\alpha (\adele))$ to $C^\infty([G^\alpha]\times [G^\alpha])$). The sum converges locally uniformly in the variable $g$, thus indeed it represents a continuous morphism into $C^\infty([G])$. By Lemma \ref{lemmafunctors}, it actually maps continuously into a subspace $V_r^\infty$ of smooth vectors of moderate growth -- notation as in Theorem \ref{constantterm}.

In the split case, we fix a Borel subgroup of $G^\alpha=G$ containing $T$, $B=TN$, and let $\mathring Y$ denote the open $B$-orbit on $Y=Y^\alpha$; in the nonsplit case, we let $\mathring Y=Y^\alpha$.

\begin{proposition}\label{constantterm-torus} In the split case, the constant term of $\Sigma\Phi$ is given by:
\begin{eqnarray} \label{constanttermeq} (\Sigma\Phi)_N (g) = \left(\int_{N(\adele)}\Phi(T\cdot ug) du + \int_{N(\adele)} \Phi(Twug) du\right) + \nonumber \\
 \sum_{\gamma\in (\mathring Y\sslash N)(k)} \int_{N(\adele)} \Phi(\gamma ug) du. 
\end{eqnarray}
The term in brackets will be denoted by $(\Sigma\Phi)_{N,\delta}$; it belongs to $I_B^G(\delta^\frac{1}{2})$. The second term will be denoted by $(\Sigma\Phi)_{N,\rest}$, and it is rapidly decaying in a neighborhood of the cusp. 

In the nonsplit case, the constant term of $\Sigma\Phi$ is equal to $(\Sigma\Phi)_{N,\rest}$, and is rapidly decaying in a neighborhood of the cusp. 
\end{proposition}

\begin{proof}
The computation of the constant term in the above form, and the fact that $(\Sigma\Phi)_{N,\delta} \in I_B^G(\delta^\frac{1}{2})$, are immediate. 

The fact that $(\Sigma\Phi)_{N,\rest}$ is of rapid decay for $\delta(b)$ large follows from the fact that eventually, for $\delta(b)$ large, the points of $(\mathring Y\sslash N)(k)$ avoid any compact subset of $(Y\sslash N)(\adele)$.
\end{proof}

By Theorem \ref{constantterm}, this implies:

\begin{corollary}\label{constantterm-torus-cor}
The map $\Sigma$ represents a continuous morphism
$$ \Sigma: \mathcal S(Y^\alpha(\adele))\to \mathcal S([G^\alpha])$$
when $T$ is non-split, and a continuous morphism
$$ \Sigma: \mathcal S(Y(\adele))\to \mathcal S^+_{\delta^\frac{1}{2}}([G])$$
when $T$ is split.
\end{corollary}

\begin{proof}
 When $\alpha$ is a non-trivial torsor, then $[G^\alpha]$ is compact and the statement is equivalent to the continuity of $\Sigma$ as a map into $C^\infty([G^\alpha])$. We now assume that $\alpha$ is trivial, so $G^\alpha = G$.

 Recall that by \eqref{Bfinitetopology}, the topology on $\mathcal S^+_{[\chi_i]}([G])$ is determined by the asymptotics map $\phi\mapsto \phi_{[\chi_i]}$ to the generalized principal series $I([\chi_i])$, and by the difference $\phi - \phi_{[\chi_i]}$, which lives in the Schwartz space $\mathcal S(\mathscr S)$ of a Siegel domain. 

 In the case at hand, for $\phi=\Sigma\Phi$, the multiset of exponents $[\chi_i]$ consists of $\delta^\frac{1}{2}$ or is empty, and the difference of $\phi_{[\chi_i]}$ from the constant term $\phi_N$ is equal to the last term $(\Sigma\Phi)_{N,\rest}$ of \eqref{constanttermeq}, which when pulled back to a Siegel domain clearly represents a continuous morphism
$$\mathcal S(Y(\adele))\to \mathcal S(\mathscr S).$$

By Theorem \ref{constantterm} and continuity of the morphism 
$\mathcal S(Y(\adele))\to V_r^\infty$ (Lemma \ref{lemmafunctors}), this implies that the map $\Phi\mapsto (\phi - \phi_{[\chi_i]}) \in \mathcal S(\mathscr S)$ is continuous.

For the split case, the map $\mathcal S(Y(\adele)) \to I(\delta^\frac{1}{2})$ is given by the term $(\Sigma\Phi)_{N,\delta}$ of \eqref{constanttermeq}, which again is continuous.
\end{proof}

Now we will explicate the functional $\RTF$ using the regularized integral and inner product that we defined; recall that the functional $\RTF$ is defined in \eqref{RTFdef}. Let $f\in \mathcal S(\mathcal Z(\adele))$, and choose a lift  $\Phi \in \oplus_v' \oplus_\alpha \mathcal S(Y^\alpha(k_v) \times Y^\alpha(k_v))$ according to \S \ref{sslifts}. According to Proposition \ref{constantterm-torus}, the function $\Sigma\Phi$ on $\sqcup_\beta [G^\beta]\times[G^\beta]$ ($\beta$ running over global $T$-torsors) is of rapid decay in the nonsplit case, and asymptotically $B$-finite with simple exponent $\delta^\frac{1}{2}$ in both variables in the split case. Thus, the regularized integral $\left< \Sigma\Phi\right>^*$ over 
the diagonal 
of $\sqcup_\beta [G^\beta]$ makes sense, and we are now ready to express the $\RTF$ in terms of that. 

\begin{proposition}\label{RTF-regularizedip}
 In the above setting:
 $$ \RTF(f) = \left< \Sigma\Phi\right>^*.$$
\end{proposition}

To begin the proof of the proposition, let $\mathcal S(\mathcal Z_v)_\alpha$ denote the coinvariants of $\mathcal S(Y^\alpha(k_v) \times Y^\alpha(k_v))$ with respect to the diagonal $G^\alpha(k_v)$-action. 

\begin{lemma}\label{globaltorsors}
 If $f \in \otimes_v' \mathcal S(\mathcal Z_v)_{\alpha_v}$, and the collection $(\alpha_v)_v$ does not correspond to a $T$-torsor over $k$, then $\RTF(f)=0$.
\end{lemma}

\begin{proof}
 This is clear for the regular orbital integrals, i.e., the terms of \eqref{RTFdef} with $\xi\ne 0,-1$, each of which is an evaluation on $\otimes_v' \mathcal S(\mathcal Z_v)_{\alpha_v}$ for \emph{precisely one} $T$-torsor $\alpha$ over $k$.
 
 The proof of Theorem \ref{Poissonbaby} also shows that $\tilde O_0(f)$ (and $\tilde O_{-1}(f)$) are zero unless the support of a preimage of $f$ meets $Y^\alpha(\adele)\times Y^\alpha(\adele)$ for some \emph{globally defined} torsor $\alpha$.
\end{proof}

Given the previous lemma, it is enough for the proof of Proposition \ref{RTF-regularizedip}, by continuity, to restrict to the subspace spanned by functions of the form $\Phi(x,y)=\Phi_1(x) \Phi_2(y)$, where $\Phi_2 = h\star \delta_{y_0}$, where $\delta_{y_0}$ denotes a delta function at a $k$-point on $Y^\alpha$ (for some $\alpha$) (the notion of ``delta function'' depends on a choice of measure), and $h \in C_c^\infty(G(\adele)) dg$. In particular, by invariance we can denote:
$$ \RTF(\Phi) = \RTF(\Phi_1 , h\star \delta_{y_0}) = \RTF(h^\vee \star \Phi_1, \delta_{y_0}),$$
and vice versa every expression of the form  $\RTF(\Phi_1', \delta_{y_0})$ can be interpreted as $\RTF(\Phi_1,\Phi_2)$ for some $\Phi_1,\Phi_2$ as before, since by the Theorem of Dixmier and Malliavin $\Phi_1'=h\star \Phi_1$ for some $h, \Phi_1$.\footnote{At non-Archimedean places this is obvious when one is working with locally constant functions, but for ``almost smooth'' functions it requires a proof, which we omit.}

\begin{lemma}\label{regT}
Let $y_0\in Y^\alpha(k)$ (for some $\alpha$) with stabilizer $T$, choose a Haar measure on $T(\adele)$, hence on $Y^\alpha(\adele)$, by our standard choice of Tamagawa measure on $G(\adele)$, thus defining a ``delta function'' $\delta_{y_0}$. Let $\Phi\in \mathcal S(Y^\alpha(\adele))$. In the split case, $ \RTF(\Phi,\delta_{y_0})$ is equal to {the analytic continuation to $t=0$ of } 
\begin{equation}\label{intT}
\int_{[T]} \Sigma\Phi(a) \cdot \min(|a|,|a|^{-1})^t da,
\end{equation}
where $\Sigma\Phi(g)= \sum_{\gamma\in Y^\alpha(k)} \Phi(\gamma g)$ as before, and $|a|$ denotes the adelic absolute value after some identification of $T$ with $\Gm$ over $k$.

In the nonsplit case, it is equal to the integral of $\Sigma\Phi$ over $[T]$.
\end{lemma}

We will denote this regularized integral by $\int_{[T]}^*$. (Similarly we define $\int^*_{T(\adele)}$.)

\begin{proof}
Notice that the invariant-theoretic quotient $Y^\alpha\sslash T = \mathcal B$, under which the preimage of any $\xi\in \mathcal B$ is a $T$-stable subscheme of $Y^\alpha$.

If $\Phi\equiv 0$ on the preimage of the irregular points $\xi=0, -1$, then
$$ \int_{[T]} \sum_{\gamma\in Y^\alpha(k)} \Phi(\gamma a) da = \sum_{\xi\in Y^\alpha(k)/T(k)} \int_{T_{\tilde \xi}(k)\backslash T(\adele)} \Phi(\tilde \xi a) da,$$
where $\tilde \xi$ denotes a preimage of $\xi$ on $Y^\alpha(k)$, and the summands corresponding to $\xi=0,-1$ vanish, while for the rest we have $T_{\tilde\xi}=1$, so we obtain
$$ \sum_{\xi \ne 0,-1} f(\xi),$$
and a posteriori the original integral was convergent.

Let us now study the terms $\tilde O_0(f)$ and $\tilde O_{-1}(f)$ in the general case. 

There is a $T$-stable neighborhood of the preimage of $0$ which is $T$-equivariantly isomorphic to $\Ga^2$, under an action of $T$ as the special orthogonal group of a quadratic form. In that case, the proof of Theorem \ref{Poissonbaby} shows that in the nonsplit case we have
$$ \tilde O_0(f) = \Vol([T]) \Phi(y_0)$$
(really, a sum of such evaluations over all $\alpha$'s, but under our assumption on the support of $\Phi$ only the given $\alpha$ contributes).

In the split case, Lemma \ref{irregularsplit} expresses $\tilde O_0$ in terms of zeta integrals for the action of $T=\Gm$ on $\Ga^2$, and we leave it to the reader to check that the expression in terms of zeta integrals coincides with the limit \eqref{intT} used to define $\int^*$. The analysis for $\tilde O_{-1}$ is completely analogous.

Thus, the expression \eqref{RTFdef} is in every case equal to \eqref{intT}.
\end{proof}

The proof of Proposition \ref{RTF-regularizedip} will be complete if we relate the regularized integral over $[T]$ to the regularized inner products defined previously:

\begin{lemma}\label{regpairing}
 Let $\Psi$ be any asymptotically $B$-finite automorphic function, and assume that it does not have the exponent $\delta^\frac{1}{2}$. Let $\Phi = h\star \delta_{y_0} \in \mathcal S(Y(\adele))$. Then:
\begin{equation}
 \left<\Psi,\Sigma\Phi\right>^* = \int_{[T]}^* h^\vee\star\Psi (a) da.
\end{equation}
\end{lemma}

\begin{proof}
Since this is clear, and there is no need for regularization, when $T$ is nonsplit, we restrict to the split case and fix an identification: $T\simeq \Gm$, so that we have an adelic absolute value $|\bullet|$ on $T(\adele)$.

Let us fix a smooth partition of unity $1=f_1+f_2$ on $[T]$ so that $f_1(a)=0$ for $|a|$ small, and $f_2(a)=0$ for $|a|$ large. For $t\in \CC$, let $f_1^t(a) = |a|^{-t} f_1(a)$ and $f_2^t(a) = |a|^t f_2(a)$. Consider them as generalized functions on $T(k)\backslash G(\adele)$, by the measures we fixed to define $\delta_{y_0}$. We have
$$\int_{[T]}^* h^\vee\star\Psi (a) da = \mbox{ the analytic continuation to $t=0$ of } $$
$$\int_{[T]} h^\vee \star \Psi(a) f_1^t(a) da + \int_{[T]} h^\vee \star \Psi(a) f_2^t(a) da,$$
so it suffices to show that the analytic continuation to $t=0$ of $\int_{[T]} h^\vee \star \Psi(a) f_i^t(a) da$ is equal to
$$\left<\Psi, \Sigma(h\star f_i)\right>^*,$$
where by abuse of notation, although $h\star f_i$ is not an element of $\mathcal S(Y(\adele))$, we write $\Sigma(h\star f_i) (g) = \sum_{\gamma\in T(k)\backslash G(k)} h\star f_i(\gamma g)$.

Exactly as in Proposition \ref{constantterm}, $\Sigma(h\star f_i^t)$ is $B$-finite with exponent $\delta^{1-t}$; more precisely, the functions $\Sigma(h\star f_i^t)$ form an analytic section of the Fr\'echet bundle with fibers $\mathcal S^+_{\delta^{1-t}}$. Therefore, by definition,
$$\left<\Psi, \Sigma(h\star f_i)\right>^* = \mbox{ the analytic continuation to $t=0$ of }\left<\Psi, \Sigma(h\star f_i^t)\right>,$$
and for $t$ large:
$$\left<\Psi, \Sigma(h\star f_i^t)\right> = \int_{T(k)\backslash G(\adele)} \Psi(g) h\star f_i^t(g) dg = \int_{[T]} h^\vee \star \Psi(a) f_i^t(a) da.$$

\end{proof}

\subsection{The Kuznetsov trace formula} 

For this subsection we let $X$ be as in \cite[\S 4]{SaBE1} (hence, isomorphic to $N\backslash \PGL_2$ -- we fix such an isomorphism over $k$),\footnote{In \cite{SaBE1} we took some care to not choose an isomorphism $X\simeq N\backslash G$; here we consider such an isomorphism as given globally over $k$. We also identify $N$ with $\Ga$ over $k$, so that the chosen adele class character $\psi$ gives rise to a character, denoted by the same letter, on $[N]$.} and $\mathcal L_\psi$ the complex line bundle whose sections (locally at a place $v$) are functions on $\PGL_2(k_v)$ satisfying $f(ng)=\psi(n) f(g)$ (in particular, the line bundle is trivialized over the point corresponding to $1\in G$). Again, we define the morphism ``summation over $k$-points'' from continuous sections of $\mathcal L_\psi$ over $N\backslash G(\adele)$ to functions on $[G]$, whenever it converges, as:
\begin{equation}
 \Sigma\Phi(\bullet):=\sum_{\gamma\in N(k)\backslash G(k)} \Phi(\gamma g);
\end{equation}
by the same symbol we will also denote the corresponding map from sections of $\mathcal L_\psi\boxtimes \mathcal L_\psi^{-1}$ on two copies of $X$ to functions on $[G]\times[G]$.

Choose an algebraic height function $r$ on $X$, as described in \S \ref{heightfunctions}. For example, one may identify the affine closure of $X$ with $\spec k[x^2,y^2,xy]$ (the quotient of affine $2$-space by the action of $\{\pm 1\}$), and define height functions on $X(k_v)$ by:
$$r_v((x,y)\mod\{\pm 1\})=\max\{m_v, m_v^{-1}\}, \,\,\text{where } m_v = \max\{|x^2|_v,|y^2|_v\}.$$ 

For every positive number $N$ we consider the Banach space of continuous sections $\Phi$ of $\mathcal L_\psi$ on $X$ satisfying
\begin{equation}\label{bound} \sup_{\xi\in X(\adele)} |\Phi(\xi)| r(\xi)^N < \infty.\end{equation}
(Recall that the absolute value of $\mathcal L_\psi$ is the trivial line bundle.)  
We will denote this space by $C_{-N}(X(\adele),\mathcal L_\psi)$.

Notice that for smooth sections of (uniformly) moderate growth, the estimate (\ref{bound}) is automatic, for every $N$, in a neighborhood of the cusp, i.e., in a neighborhood of the image of ``zero'' under the isomorphism $N\backslash\SL_2 = \mathbb A^2\smallsetminus\{0\}$. Hence, for those, the growth condition is restrictive only in a neighborhood of ``infinity''.

Similarly, we fix a height function $R(g)$ on $[G]$, and let $C_{-N} ([G])$ denote the Banach space of continuous functions $\phi$ on $[G]$ which satisfy
$$ \sup_{x\in X(\adele)} |\phi(g)| R(g)^N < \infty.$$

\begin{proposition} \label{continuousmap} 
There is a positive constant $c$ such that for large $N$ the map: $\Phi\mapsto \Sigma\Phi$ represents a continuous morphism:
\begin{equation}C_{-N}(X(\adele),\mathcal L_\psi)\to C_{-cN}([G]).
\end{equation}
\end{proposition}

\begin{proof}
For any compact neighborhood $U$ of the identity in $G(\adele)$ there are positive constants $c_1,c_2$ such that $c_1 r(x)\le r(x g) \le c_2 r(x)$ for all $g\in U$ and $x\in X(\adele)$ (Lemma \ref{comparisons}). This shows that all our estimates below are locally uniform on $[G]$, but we will not comment further on that.

Using constants as in Lemma \ref{comparisons}, we have
$$ \sum_{\xi\in X(k)} |\Phi(\xi)| \ll \sum_{\xi\in X(k)} |r(\xi)|^{-N} \ll  \int_0^\infty T^{-N} T^M dT < \infty,$$
for large $N$, so convergence of the sum representing $\Sigma\Phi(g)$ (by the above, locally uniform in $g$) is not an issue.

To prove the asymptotic properties of $\Sigma\Phi$ we need an estimate in the opposite direction than that of Lemma \ref{comparisons}: we need to show that when $g\in [G]$ becomes ``large'', $r(x\cdot g)$ also becomes ``large''. 
We may fix a Siegel domain $\mathscr S\subset G(k_\infty)$ for $[G]$, $g=nak$, and assume that $R(gak)= \Vert a\Vert$ for large $\Vert a\Vert$, where $\Vert a\Vert$ represents some algebraic height function on the torus $A$. For simplicity, let us assume that $k=\QQ$ (the general case is similar, with different constants). For simplicity of notation, also, we will present the case of $N\backslash \SL_2 \simeq \mathbb A^2\smallsetminus\{0\}$ instead of that of $N\backslash\PGL_2$, denoting it again with $X$.

We may choose coordinates $(x,y)$ on $\mathbb A^2$ so that the torus $\Gm\simeq A$ acts as: $(x,y)\cdot a = (ax,a^{-1}y)$. We define the height function for $N\backslash \SL_2$ as $r_v(x,y) = \max(m_v,m_v^{-1})$, where $m_v=\max\{|x|_v,|y|_v\}$. We also define a height function for $\mathbb A^2$ as $r_v'(x,y)=m_v$. Our goal is to estimate
\begin{equation} \label{sumforestimate}\sum_{\xi\in \mathbb A^2\smallsetminus\{0\}(k)} r(\xi\cdot a)^{-N}.
\end{equation}
for $a\in A(\RR)$ in a Siegel neighborhood (see \S \ref{Siegeldomain}). We may choose an isomorphism: $A(\RR)\simeq \RR^\times$ so that points in the Siegel neighborhood have $|a|\ge 1$. We split the sum in those $\xi$ with $x=0$ and the rest, and since $r'(\xi)\le r(\xi)$ we have that \eqref{sumforestimate} is
$$ \le \sum_{x\in k^\times, y\in k} r'((x,y)\cdot a)^{-N} + \sum_{y\in k^\times} r((0,y)\cdot a)^{-N}.$$

For $x\ne 0$ we have

$$ \frac{r'_\infty(ax,a^{-1}y)}{r'_\infty(x,y)} \ge \frac{|ax|_\infty}{\max\{|x|_\infty,|y|_\infty\}} \cdot \frac{\max\{|x|_\infty,|y|_\infty\}}{\max\{1,|x|_\infty,|y|_\infty\}} \ge $$
$$ \ge
\begin{cases}
 |a| |x|_\infty, & \mbox{ if } |x|_\infty, |y|_\infty \le 1,\\
 |a|, & \mbox{ if } |x|_\infty \ge 1, |y|_\infty,\\
 |a| \frac{|x|^2_\infty}{|xy|_\infty} & \mbox{ if } |y|_\infty \ge |x|_\infty, 1.
\end{cases}$$

Using the product formula to express $|x|_\infty = |x|_f^{-1} \ge r'_f(\xi)^{-1} = r'_f(\xi\cdot a)^{-1} $
(the index $~_f$ denoting the product over all finite places, and we have used the fact that $a \in G(k_\infty)$), in the first case we get
$$\frac{r'_\infty(ax,a^{-1}y)}{r'_\infty(x,y)} \ge  \frac{|a| }{r_f(\xi\cdot a)} \Rightarrow r'(\xi\cdot a)^2 \ge |a| r'(\xi).$$
We have used here that $r'_f(ax,a^{-1}y) = r'_f(x,y)$ and $r'_\infty(ax,a^{-1}y)\ge 1$.

In the second case we get, by multiplying with $r'_f(x,y)$,
$$r'(\xi\cdot a) \ge |a| r'(\xi),$$
and in the third, using the additional fact that $|xy|_\infty \le 2 r_\infty'(\xi\cdot a)^2$, we get
$$\frac{r'_\infty(ax,a^{-1}y)}{r'_\infty(x,y)} \ge \frac{|a|}{2 r'_f(\xi\cdot a)^2\cdot r_\infty'(\xi\cdot a)^2} \Rightarrow r_f'(\xi\cdot a)^4 \ge \frac{1}{2}|a| r'(\xi).$$

Since $r'(\xi\cdot a)\ge 1$ we have in all cases
\begin{equation}
 r'(\xi\cdot a)^4 \ge \frac{1}{2}|a| r'(\xi).
\end{equation}

Hence, using constants as in (\ref{rationalpoints}):
$$\sum_{\xi\in X(k), x\ne 0} r'(\xi\cdot a)^{-N} \ll |a|_{\infty}^{-\frac{N}{4}} \sum_{\xi\in X(k), x\ne 0} r'(\xi)^{-\frac{N}{4}} \ll $$ 
$$ \ll |a|_{\infty}^{-\frac{N}{4}}\cdot \int_0^\infty T^{-N} T^M dT \ll |a|_{\infty}^{-\frac{N}{4}}$$
for $N$ large enough, with the implicit constant depending on $N$.

When $x$ is zero, we similarly have
$$ \frac{r_\infty(ax,a^{-1}y)}{r_\infty(x,y)} \ge \frac{|ay^{-1}|_\infty}{\max\{|y|_\infty,|y|_\infty^{-1}\}}$$
which is equal to $|a|$ when $|y|_\infty\le 1$. When $|y|_\infty >1$ it is
$$ |a| |y|_\infty^{-2} = |a| |y|_f^2 \ge |a| r_f(\xi)^{-2} = |a| r_f(\xi\cdot a)^{-2},$$
hence, as before, we have in both cases: 
$$ r(\xi\cdot a)^3 \ge |a| r(\xi),$$
and we can estimate using \eqref{rationalpoints}:
$$\sum_{y\in k} r((0,y)\cdot a)^{-N} \ll |a|^{-\frac{N}{3}}$$
when $N$ is large.

This gives the following estimate for \eqref{sumforestimate}:
$$\sum_{\xi\in \mathbb A^2\smallsetminus\{0\}(k)} r(\xi\cdot a)^{-N} \ll |a|^{-\frac{N}{4}},$$
which completes the proof.
\end{proof} 

Now we look at the spaces $\mathcal S^s(\overline X\times X (\adele), \mathcal L_\psi\boxtimes\mathcal L_\psi^{-1})$ of non-standard Whittaker functions, whose orbital integrals give rise to the space $\mathcal S(\mathcal W^s(\adele))$.

\begin{lemma}\label{belongsCN}
 For any given $N>0$ there is $\sigma_0\in \mathbb R$ such that for $\Re(s)\ge \sigma_0$ the space $\mathcal S^s(\overline X\times X (\adele), \mathcal L_\psi\boxtimes\mathcal L_\psi^{-1})$ belongs to $C_{-N}(X\times X(\adele),\mathcal L_\psi\boxtimes\mathcal L_\psi^{-1})$, and the embedding is bounded by polynomial seminorms on the former.
\end{lemma}

\begin{proof}
 This is immediate from the definition of these spaces in \cite[\S 6.1]{SaBE1}.
\end{proof}

In particular, for $\Re(s)$ large, the integral $\left< \Sigma\Phi_s\right>$ over the diagonal converges and can be decomposed spectrally by the Plancherel formula for $L^2([G])$. Now recall that the Kuznetsov trace formula was defined in \eqref{KTFdef}; we claim:

\begin{proposition}\label{KTF-ip}
 In the above setting, and for $\Re(s)\gg 0$:
 $$\KTF(f_s)= \left< \Sigma\Phi_s\right> + \tilde O_0(f_s).$$
\end{proposition}

\begin{proof}
 It is clear that:
 $$\left< \Sigma\Phi_s\right> = \sum_{\xi\in (X(k)\times X(k))/G(k)} \int_{G_\xi(k)\backslash G(\adele)} \Phi_s(\xi g) dg = $$
 $$ = \sum_{\xi\in X\times X/G(k)} \int_{G_\xi(\adele)\backslash G(\adele)} \Phi_s(\xi g) dg $$
 (using the fact that the Tamagawa volume of $[N]$ is $1$)
 $$ = \sum_{\xi\in k^\times} f_s(\xi) + \int_{N\backslash G(\adele)} \Phi_s(g,g) dg,$$
 so it is enough to identify this last integral with the ``irregular orbital integral'' $\tilde O_\infty(f_s)$. 
 
 Suppose that $\Phi_s = \prod_v \Phi_{v,s}$ (hence $f_s=\prod_v f_{v,s}$), where $\Phi_{v,s}$ is a section of $S^s(\overline X\times X (k_v), \mathcal L_\psi\boxtimes\mathcal L_\psi^{-1})$. Notice that for every place $v$, $\int_{N\backslash G(k_v)} \Phi_{v,s}(g,g)$ is the ``inner product'' denoted by $\left<f_{v,s}\right>$ in \cite[\S 4.9]{SaBE1}. (The fact that $s\ne 0$ plays no role in its definition.) By \cite[Proposition 4.8]{SaBE1}, and the definition of $\tilde O_\infty$ in \S \ref{ssirregular}, this is equal to $\tilde O_\infty(f_{v,s})$.

\end{proof}

The functional $\tilde O_0$ was defined as the analytic continuation to $t=0$ of the functional \eqref{zerocontrib}, and an explicit expression of this appears in \eqref{defzerononsplit}, \eqref{defzerosplit}. The last piece that we need for the proof of Theorem \ref{spectral-KTF} is:

\begin{lemma}\label{irregular-spectral}
 The functional $h\mapsto \tilde O_0(h\star f_s)$ is an evaluation at  $\delta^{\frac{1}{2}+s}$, $\eta\cdot\delta^{\frac{1}{2}+s}$ (s.\ \S \ref{notation-spectral}).
\end{lemma}

\begin{proof}
For every place $v$, by definition of the space $\mathcal S^s(\overline X(k_v))$ (s.\ \cite[\S 4.5 and 6.1]{SaBE1}), there is a natural quotient map obtained by the asymptotics of sections
 $$ \mathcal S^s(\overline X\times X (k_v), \mathcal L_\psi\boxtimes\mathcal L_\psi^{-1}) \twoheadrightarrow I([\delta^{\frac{1}{2}+s}, \eta\delta^{\frac{1}{2}+s}]) \otimes \mathcal S(X (k_v), \mathcal L_\psi^{-1}),$$
where $I([\delta^{\frac{1}{2}+s}, \eta\delta^{\frac{1}{2}+s}])$ is the generalized principal series with the given multiset of exponents (as in \S \ref{asBfinite}, but here locally at $v$). Let us denote:
 $$V_v:= I([\delta^{\frac{1}{2}+s}, \eta\delta^{\frac{1}{2}+s}]) \otimes \mathcal S(X (k_v), \mathcal L_\psi^{-1}).$$
 
 Correspondingly, there is quotient of coinvariants under the diagonal $G(k_v)$-action:
 $$ \mathcal S(\mathcal W^s_v) \twoheadrightarrow (V_v)_G,$$
 and $(V_v)_G$ is two-dimensional, whose dual is spanned by the local distributions appearing in the definition of $\tilde O_0$ of \eqref{defzerononsplit} and \eqref{defzerosplit}. (In the nonsplit case, those are denoted by $\tilde O_{0_v,\pm}$, and in the split case they are denoted by $\tilde O_{0_v}$ and $\tilde O_{u_v}$.)
  
 In the nonsplit case the restricted tensor product $\bigotimes_v' V_v$ has a canonical quotient to:  
 $$ \left(I(\delta^{\frac{1}{2}+s}) \oplus I(\eta\delta^{\frac{1}{2}+s})\right) \otimes \mathcal S(X (\adele), \mathcal L_\psi^{-1}),$$
 and it is clear from the definition of $\tilde O_0$ in \eqref{defzerononsplit} that $\tilde O_0$ factors through this quotient; in particular, the functional $h\mapsto \tilde O_0(h\star f_s)$ is an evaluation at is an evaluation at $\delta^{\frac{1}{2}+s}$, $\eta\cdot\delta^{\frac{1}{2}+s}$.
 
 In the split case the argument is more complicated. From the two summands of \eqref{defzerosplit}, it is clear that the first one:
 $$2a_0^{S_0} \prod_{v\in S_0} \tilde O_{0_v},$$
 (where we use $S_0$ to denote a set of places strictly larger than $S$) when applied to $h\star f$, is an evaluation at $\delta^{\frac{1}{2}+s}$. 
 
 The problem lies in analyzing the second term:
 $$ a_{-1}^{S_0} \prod_{v\in S_0} \tilde O_{0_v}  \cdot \left(\sum_{v\in S_0} \frac{\tilde O_{u_v}}{\tilde O_{0,v} }\right).$$
 and showing that, when applied to $h\star f$, it is a linear combination of an evaluation of $\hat h$ at $\delta^{\frac{1}{2}+s}$ and \emph{its derivative at $t=0$ along the one-parameter family}: $t\mapsto \delta^{\frac{1}{2}+s+t}$.
 
 For this it suffices to show the following:
\begin{lemma}
For every $v\in S_0\smallsetminus S$ the functional:
 $$ h\mapsto \tilde O_{u_v}(h\star f^0_v)$$
 is equal to an evaluation of $\hat h$ at $\delta^{\frac{1}{2}+s}$ minus:
 $$ \tilde O_{0_v}(f_v^0) \cdot \frac{d}{ds} \hat h\left(\delta^{\frac{1}{2}+s}\right).$$
\end{lemma}

For this we recall several facts about orbital integrals for elements of $\mathcal S^s(\overline X\times X (k_v), \mathcal L_\psi\boxtimes\mathcal L_\psi^{-1})$. First of all, at non-Archimedean places the stalk of $\mathcal S(\mathcal W^s)$ at zero is canonically isomorphic to \emph{a stalk of the $(N_v,\psi_v^{-1})$-coinvariants of the generalized principal series} $I([\delta^{\frac{1}{2}+s}, \delta^{\frac{1}{2}+s}])$. Indeed, the theory of asymptotics \cite[\S 5]{SV} provides a canonical, $G$-equivariant map $e_\emptyset^*: C^\infty(N\backslash G, \mathcal L_\psi)\to C^\infty(N\backslash G)$ which is characterized by the property that 
the 
restriction of any function to a Cartan subgroup $A$ normalizing $N$ is equal to the restriction of its image under $e_\emptyset^*$ when restricted to a subset of the form $\delta(a)\ll 1$. This map allows us to canonically identify the stalk of the cosheaf $\mathcal S^s(\overline X,\mathcal L_\psi)$ ``at infinity'' with the corresponding stalk of $I([\delta^{\frac{1}{2}+s}, \delta^{\frac{1}{2}+s}])$, in a $G$-equivariant way, and hence also their coinvariants. Let us denote by $V_{v,0}$ the corresponding stalk of $I([\delta^{\frac{1}{2}+s}, \delta^{\frac{1}{2}+s}])_{(N,\psi^{-1})}$: it is the stalk at the point corresponding to ``infinity'' on $X$; so we have a canonical isomorphism
$$ \mathcal S(\mathcal W^s_v)_0\simeq V_{v,0}.$$

Secondly, we notice that the normalized $k_v^\times$-action on this stalk $\mathcal S(\mathcal W_v^s)_0$, tensored by the inverse absolute value, i.e., the action
$$ a* f(\xi) := |a|^{-\frac{1}{2}} f(a\xi),$$
corresponds under the above isomorphism to the normalized $k_v^\times$-action on $I([\delta^{\frac{1}{2}+s}, \delta^{\frac{1}{2}+s}])$, where $k_v^\times$ is identified with $A_v$ via the character $\delta$ and the normalized action is
$$ a\cdot \Phi(g) := \delta(a)^{-\frac{1}{2}} \Phi(ag).$$

Thirdly, the endomorphism ring of $C^\infty(N\backslash G)^{K_v}$ (where $K_v=G(\mathfrak o_v)$) generated by the unramified Hecke algebra of $G_v$ is a subring of the endomorphisms generated by the unramified Hecke algebra of $A_v$; the normalization of the action of $A_v$ is compatible with the Satake isomorphism, i.e., $h\in \mathcal H(G_v,K_v)$ induces the same endomorphism as the image of $\hat h$ under the maps: $\CC[\check A]^W\hookrightarrow \CC[\check A] \xrightarrow{\sim} \CC[A_v/A(\mathfrak o_v)]\xrightarrow\sim\mathcal H(A_v,A(\mathfrak o_V))$. 

Finally, we recall the way that $\tilde O_{0_v}(h\star f_v^0)$ and $\tilde O_{u_v}(h\star f_v^0)$ determine the germ of $h\star f_v^0$ in the stalk $\mathcal S(\mathcal W^s_v)_0$: the asymptotic behavior of $h\star f_v^0$ close to $\xi=0$ is of the form
$$|\xi|^{s+1}\left(-\tilde O_{0_v}(h\star f_v^0) \log|\xi| + \tilde O_{u_v}(h\star f_v^0)\right)$$
(cf.\ \eqref{Oasympsplit}, but here with the extra factor $|\xi|^{s+1}$). 

Given all that, it suffices to show that, under the $*$-action of $\mathcal H(A_v,A(\mathfrak o_V))\simeq \mathcal H(k_v^\times,\mathfrak o_v^\times)$ on the stalk at $\xi=0$ of functions of the form
$$ F(\xi) = |\xi|^{s+1}\left(-C_1 \log|\xi| + C_2\right),$$
the functional ``$C_2$'' has the property
$$ C_2(h* F) = A \hat h(\delta^{\frac{1}{2}+s}) + C_1(F) \frac{d}{ds} \hat h\left(t\mapsto \delta^{\frac{1}{2}+s+t}\right).$$
It is easy to see that under the above normalized action of $A_v$:
$$ a* F(\xi) = |\xi|^{s+1} (-C_1 |a|^{\frac{1}{2}+s} \log|a| ) + \mbox{ a term proportional to }|a|^{\frac{1}{2}+s}.$$
If $h=$ the characteristic measure of $aA(\mathfrak o_v)$ then under the isomorphism $\mathcal H(A_v,A(\mathfrak o_v)) \simeq \CC[t,t^{-1}]$ (chosen so that the ``point'' $\delta^s$ on $\spec \CC[t,t^{-1}]$ corresponds to the evaluation $t\mapsto |\varpi|^s$) we have: $\hat h (t)= t^{\val(a)}$.

Therefore, we translate: $-C_1 |a|^{\frac{1}{2}+s} \log|a| = -C_1 \frac{d}{ds} \hat h\left(\delta^{\frac{1}{2}+s}\right)$,
and our claim is proven. 

\end{proof}

\subsection{Functional equation}\label{ssfe}

We now come to what was called the ``second miracle'', namely Theorem \ref{2ndmiracle}. To prove it, recall that $\mathcal S(\mathcal Z_v^s)$ denotes, by definition, the image of $\mathcal S(\mathcal W^s(k_v))$ under $f\mapsto \mathcal G^{-1} \left(|\bullet|^{-s-1} f\right)$.
We claim:

\begin{proposition}\label{propfetorus}
 Multiplication by $\eta(\xi)\eta(\xi+1)|\xi+1|^{2s}$ induces an isomorphism:
$$\mathcal R: \mathcal S(\mathcal Z_v^{-s}) \xrightarrow{\sim} \mathcal S(\mathcal Z_v^s)$$
which maps basic functions to each other. Hence, it also induces an isomorphism of global Schwartz spaces:
$$\mathcal S(\mathcal Z^{-s}(\adele)) \xrightarrow{\sim} \mathcal S(\mathcal Z^s(\adele)).$$
\end{proposition}

\begin{proof}
 The isomorphism follows immediately from the asymptotic behavior of elements in the space $\mathcal S_{1,v}^0$ as described by Table (\ref{tablegerms}); notice that both $\mathcal S(\mathcal Z_v^{-s}) $ and $\mathcal S(\mathcal Z_v^{-s}) $ correspond to the subspaces with $C_5=0$, in the notation of that table.

 The fact that basic functions are carried over to each other is clear from \eqref{basicfn}. 
\end{proof}

Now let $\mathcal T$ denote the following isomorphism:
$$ \mathcal T: \mathcal S(\mathcal W_v^{-s}) \to \mathcal S(\mathcal Z_v^{-s}) \to  \mathcal S(\mathcal Z_v^s) \to \mathcal S(\mathcal W_v^s),$$
where the middle arrow is that of the previous proposition, and the others are the isomorphisms we have been using throughout, the first given by $f\mapsto \mathcal G^{-1}\left(|\bullet|^{s-1} f\right)$ and the second given by $f\mapsto |\bullet|^{s+1} \mathcal G(f)$. To finish the proof of Theorem \ref{2ndmiracle}, we need to show that $\mathcal T$ is compatible with the action of the spherical Hecke algebra on the basic vectors, and that it preserves the property of sections to be analytic of rapid decay in vertical strips. The latter is clear, since this is a property of the transform $|\bullet|^{s+1} \mathcal G$ and its inverse, and is clearly true for the transform of Proposition \ref{propfetorus}.

\begin{proposition}\label{propfeKuz}
At almost every place $v$, and for every $h\in \mathcal H(G_v,K_v)$, the isomorphism $\mathcal T$ takes $h\star f_{\mathcal W_v^{-s}}^0$ to $h\star f_{\mathcal W_v^s}^0$.
\end{proposition}

\begin{proof}
 We will leave the details of the proof to the reader; it can be done by (very unpleasant) explicit calculations, or by ``cheating'' as in the proof of the fundamental lemma \cite[Theorem 5.4]{SaBE1}, by using the following fact: the space $\mathcal S(\mathcal Z_v^s)$ can be identified with the space obtained by orbital integrals for the quotient $(\Gm,|\bullet|_v^s)\backslash G/(\Gm,\eta_v(\bullet)|\bullet|_v^s)$, where $\Gm$ denotes the torus of diagonal elements in $\PGL_2$. Thus, the space \emph{does} have an action of the Hecke algebra (although strictly speaking we would like to avoid using that, and indeed we avoid it in global considerations).
 
 With respect to this action, the above transforms: $\mathcal S(\mathcal W_v^{\pm s}) \leftrightarrow \mathcal S(\mathcal Z_v^{\pm s})$, given by $\mathcal G$ and a suitable power of the absolute value, are equivariant when restricted to images of unramified vectors; this is proven exactly as in the proof of the Fundamental Lemma, \cite[Theorem 5.4]{SaBE1}, with the intervention of an intermediate ``space'' $\mathcal W_1^s$.
 
 Finally, we have $G_v$-equivariant maps:
$$ \mathcal S\left((\Gm, |\bullet|_v^{-s})\backslash\PGL_2\right) \ni f \mapsto f(w\bullet) \in \mathcal S\left((\Gm, |\bullet|_v^{s})\backslash\PGL_2\right)$$
and:
$$ \mathcal S\left((\Gm, \eta_v(\bullet)|\bullet|_v^{-s})\backslash\PGL_2\right) \ni f \mapsto f(w\bullet) \in \mathcal S\left((\Gm, \eta_v(\bullet)|\bullet|_v^{s})\backslash\PGL_2\right).$$ 

It can easily be checked that at the level of orbital integrals, i.e., $G_v$-coinvariants, these descend to the transform $\mathcal R$ of Proposition \ref{propfetorus}.
\end{proof}

Even if we want to forget about the Hecke action on $\mathcal S(\mathcal Z_v^s)$ (more precisely: on its ``unramified'' vectors) when $s\ne 0$, for $s=0$ we clearly have a Hecke action that comes from its original definition in terms of coinvariants of $\mathcal S(T_v\backslash G_v \times T_v\backslash G_v)$ (and suitable inner forms, which however do not play a role when discussing unramified vectors). A corollary of the above and the Fundamental Lemma of \cite{SaBE1} (or a direct corollary of the last proof) is:

\begin{corollary}\label{corollaryR}
 The transform $\mathcal R$ of Proposition \ref{propfetorus}, when $s=0$, satisfies:
 $$\mathcal R\left(h\star f_{\mathcal Z_v}^0\right) = h\star f_{\mathcal Z_v}^0$$
 for every $h\in \mathcal H(G_v,K_v)$. 
\end{corollary}

\section{The formula of Waldspurger} \label{sec:Waldspurger} 

\subsection{Local periods}

We define local relative characters, i.e., the local analogs of the period relative characters defined in \S \ref{sslifts}. 

If $\alpha$ denotes an isomorphism class of $T$-torsors over $k_v$, $G^\alpha$ is the corresponding inner form of $G$, and $\pi_v^\alpha$ is a tempered irreducible representation of $G^\alpha$, then the integral 
\begin{equation}\label{Tperiod} \tilde v\otimes v\mapsto \int_{T^\alpha_v} \left< \tilde\pi_v^\alpha(h) \tilde v, v\right>_{\tilde\pi^\alpha_v} dh
\end{equation}
converges and represents a $T^\alpha_v$-biinvariant functional on $\tilde\pi_v^\alpha\otimes \pi_v^\alpha$ and hence, by Frobenius reciprocity, a morphism: 
$$\tilde\pi^\alpha_v\otimes {\pi^\alpha_v} \to C^\infty(T^\alpha_v\backslash G^\alpha_v\times T^\alpha_v\backslash G^\alpha_v).$$
(Of course, $T^\alpha \simeq T$, but we include the exponent in analogy with the notation for $G$.)

Dualizing, with respect to a fixed invariant measure on $(T^\alpha\backslash G^\alpha)(k_v)$, and composing with the pairing of duality, we get a local relative character
\begin{equation}\label{Jpilocal}
J_{\pi_v^\alpha}: \mathcal S(T^\alpha_v\backslash G^\alpha_v\times T^\alpha_v\backslash G^\alpha_v)\to {\pi_v^\alpha}\hat\otimes\tilde\pi_v^\alpha\to \CC.
\end{equation}

As shown in \cite[\S 6]{SV}, these relative characters play a role in the Plancherel formula. To express it in a way suitable for our application, 
fix $k$-rational invariant volume forms on $G$ and $T\backslash G$; since inner twists preserve rationality of invariant volume forms, this fixes volume forms, on all inner twists $G^\alpha$, $T^\alpha\backslash G^\alpha$ over $k_v$, and hence invariant measures on the $k_v$-points. 

Recall the ``inner product'' functional on $\mathcal S(\mathcal Z_v)$, defined in \cite[\S 3.6]{SaBE1}: If $(\Phi^\alpha)_\alpha$ is a lift of $f\in\mathcal S(\mathcal Z_v)$ as in \S \ref{sslifts}, then the inner product is essentially the integral over the diagonal: 
\begin{equation}
 \left<f\right> := (k_v^\times:N_{k_v}^{E_v}E_v^\times)^{-1}\cdot \sum_\alpha  (-1)^\alpha \AvgVol(T(k_v)) \int_{T^\alpha\backslash G^\alpha(k_v)} \Phi^\alpha(x,x) dx,
\end{equation}
where the ``average volume'' is the one defined in \S \ref{sstoriprelim}. 

By the Plancherel formula, this admits a decomposition into relative characters:
\begin{equation}\label{Plancherel} \left<f\right> = \frac{\AvgVol(T(k_v)) }{(k_v^\times:N_{k_v}^{E_v} E_v^\times)}\cdot \sum_\alpha  (-1)^\alpha \int_{\widehat{G_v^\alpha}^\temp} J_{\pi_v^\alpha} (f) d\pi_v^\alpha.
\end{equation}
Then, it was shown in \cite[Theorem 6.2.1]{SV} that the measure $d\pi_v^\alpha$ can be taken to be the Plancherel measure for $G_v^\alpha$ corresponding to the chosen Haar measure, and then the  
relative characters $J_{\pi_v^\alpha}$ appearing in \eqref{Plancherel} are the same as the ones for \eqref{Jpilocal}, provided that the measure on $T^\alpha_v$ used to define them is compatible with the choices of measures on $G^\alpha_v$, $T^\alpha_v\backslash G^\alpha_v$.

\begin{remark}
One easily checks that the product $\AvgVol(T(k_v)) J_{\pi_v^\alpha} d\pi_v^\alpha$, as a measure on $\widehat{G^\alpha}$ valued in the dual of $\mathcal S(\mathcal Z_v)$, does not depend on choices of measures (as long as, of course, the measures on $G_v^\alpha, T_v^\alpha$ and $T_v^\alpha\backslash G_v^\alpha$ are chosen compatibly). 
\end{remark}

For the split form $G^\alpha = G$, and for $\tilde\pi_v\otimes \pi_v$ varying in the family of tempered principal series representations $I(\chi_v^{-1})\otimes I(\chi_v)$, as $\chi_v$ varies in the characters of the Borel subgroup, the integral \eqref{Tperiod} is \emph{meromorphic} in $\chi_v$, and hence the relative characters $J_{\chi_v} := J_{\pi_v}$ extend by meromorphicity to a dense open set of $\chi_v$'s. In fact, it is easy to see that the integral \eqref{Tperiod} continues to converge for $|\chi_v| = \delta^\sigma$ with $|\sigma|< \frac{1}{2}$, a neighborhood of the set of unitary characters of the Borel which after induction contains all generic unitary representations, therefore $J_{\pi_v}$ is defined for all generic unitary representations of $G_v$.

We can repeat the above for tempered representations in the Whittaker case, and we will denote the corresponding local functionals by $I_{\pi_v}$; the only difference is that the analog of \eqref{Tperiod}:
\begin{equation}\label{Whperiod} \tilde v\otimes v \mapsto \int_{N_v} \left< \tilde\pi_v(n) \tilde v, v\right> \psi_v(n) dn
\end{equation}
is not convergent and needs to be regularized as in \cite[\S 6]{SV}. (Of course, in this case we have no torsors or inner forms of $G$ showing up.) A priori, these functionals are defined on the coinvariants of \emph{standard} test functions $\mathcal S(\overline X\times X (k_v), \mathcal L_\psi\boxtimes\mathcal L_\psi^{-1})$, however it is easy to see, using Lemma \ref{belongsCN}, that for $\Re(s)\gg 0$ the elements of $\mathcal S^s(\overline X\times X (k_v), \mathcal L_\psi\boxtimes\mathcal L_\psi^{-1})$ belong to the \emph{Harish-Chandra Schwartz space} \cite{BePl}, and hence the functionals $I_{\pi_v}$ extend continuously to this space, i.e., to $\mathcal S(\mathcal W^s_v)$. 

Again, the relative characters $I_{\pi_v}$ play a role in the Plancherel decomposition of the ``inner product'' functional on $\mathcal S(\mathcal W_v)$, defined in \cite[\S 4.9]{SaBE1}:
\begin{equation}\label{Plancherel-W} \left<f\right> := \int_{N\backslash G(k_v)} \Phi(x,x) dx = \int_{\widehat{G_v}^\temp} I_{\pi_v} (f) d\pi_v,
\end{equation}
where $\Phi$ is a lift of $f$ to the space $\mathcal S^0(\overline X\times X (\adele), \mathcal L_\psi\boxtimes\mathcal L_\psi^{-1})$ of non-standard Whittaker functions, and measures are chosen in the same way as in the torus case.

Moreover, if $\pi_v= I(\chi_v)$ then the functionals $I_{\pi_v}$ extend to a meromorphic family $I_{\chi_v}$. To see that these functionals make sense for $\pi_v$ generic, unitary, but not necessarily tempered, it is well-known that the integral
\begin{equation}\label{Wpairing}[\tilde W,W]:= \int \tilde W\left(\begin{array}{cc} a \\ & 1\end{array}\right) W\left(\begin{array}{cc} a \\ & 1\end{array}\right) da,\end{equation}                                                                                                                                                                      
where $W$ belongs to the Whittaker model of $\pi_v$ and $\tilde W$ belongs to the Whittaker model of $\tilde\pi_v$, is convergent for $\pi_v$ unitary (generic) and represents a nonzero invariant pairing. It is known (cf.\ the proof of \cite[Theorem 18.3.1]{SV}) that if, for a tempered representation, the regularized \eqref{Whperiod}, combined by Frobenius reciprocity, takes $\tilde v\otimes v$ to the Whittaker function $\tilde W\otimes W$, then $\left<\tilde v,v\right> = [\tilde W,W]$. By meromorphicity, this continues to be true in the domain of convergence of \eqref{Wpairing}, and therefore \eqref{Whperiod} remains regular (and nonzero) for all unitary generic representations.

The following is shown in \cite[Theorem 6.4.1]{SV}:

\begin{proposition}\label{nonvanishing}
If $\pi_v\in \hat G_v^\temp$ then  $I_{\pi_v}\ne 0$ if and only if $\pi_v$ is generic; if $\pi_v^\alpha\in \widehat{G_v^\alpha}^\temp$ then $J_{\pi^\alpha_v}\ne 0$ if and only if $\Hom_{T_v^\alpha}(\pi_v^\alpha,\CC)\ne 0$.
\end{proposition}

We now have the matching of orbital integrals:
$$ |\bullet|\mathcal G: \mathcal S(\mathcal Z_v)\xrightarrow{\sim} \mathcal S(\mathcal W_v^0)$$
and a basic question concerns the push-forward, resp.\ pull-back, of the relative characters $J_{\pi_v^\alpha}$, resp.\ $I_{\pi_v}$. 

\begin{proposition}\label{factorsthrough}
For every $\pi_v\in \hat G_v^\temp$ there is a unique inner form $G^\alpha$ corresponding to a $T$-torsor $\alpha$ over $k_v$, and a Jacquet--Langlands lift $\pi_v^\alpha$ of $\pi_v$ to $G^\alpha$, such that the pull-back $\left(|\bullet|\mathcal G\right)^*I_{\pi_v}$ to $\mathcal S(\mathcal Z_v)$ is a nonzero multiple of a relative character attached to $\pi_v^\alpha$.
\end{proposition}

This proposition will be proven in \S \ref{ssWhittaker}, in order not to interrupt the local discussion. (The proof uses a global argument.)

The definition of Jacquet--Langlands lift that is used here is: those elements $\pi_v^\alpha$ of $\widehat{G_v^\alpha}$ for which there exist:
\begin{itemize}
 \item an automorphic representation $\pi \simeq \otimes^\prime_w \pi_w$ of $G$ with $\pi_w\simeq \pi_v$ at the place $w=v$;
 \item and an automorphic representation $\pi^\beta \simeq \otimes^\prime_w \pi^\beta_w$ of $G^\beta$, where $\beta$ is a $k$-rational torsor of $T$ with $\beta_v=\alpha$, $\pi^\beta_v\simeq \pi^\alpha_v$ and $\pi^\beta_w\simeq  \pi_w$ for almost all places $w$ (where $G^{\beta_w}$ is split). 
\end{itemize}

We will use the fact that there is at most one Jacquet--Langlands lift $\pi_v^\alpha$ of $\pi_v$ for any inner form $\alpha$, and that it is tempered if $\pi_v$ is. We will also use strong multiplicity one for inner forms of $G$. It is plausible that some of these facts can be obtained independently (for $T$-distinguished representations) from our methods by refining the arguments.

Moreover, for the theorem that follows we will need to use equality of formal degrees for Jacquet--Langlands lifts:  Recall that a Haar measure on $G_v$ induces a Haar measure on any inner form $G_v^\alpha$ (since inner twists preserve rationality of volume forms), and it is such Haar measures that we fix on all relevant inner forms. Then \cite[Theorem 7.2]{AP} (see also \cite[\S 3.1]{HII}) states that for any discrete series $\pi_v$ of $G$ and a Jacquet--Langlands lift $\pi_v^\alpha$ of it, the corresponding Plancherel measures satisfy
\begin{equation}\label{Plmeasures} d\pi_v(\pi_v) = d\pi_v^\alpha(\pi_v^\alpha).
\end{equation}

\begin{theorem}\label{charmatching}
 For $\pi_v\in \widehat{G_v}^\temp$ we have:
 \begin{equation}\label{eqcharmatching} \left(|\bullet|\mathcal G\right)^*I_{\pi_v} =
(-1)^\alpha  \gamma_v^*(\eta_v,0,\psi_v)\cdot  \frac{\AvgVol(T(k_v)) }{(k_v^\times:N_{k_v}^{E_v}E_v^\times)} J_{\pi^\alpha_v},\end{equation}
 where $\pi^\alpha_v$ is the Jacquet--Langlands lift of Proposition \ref{factorsthrough}. More generally, the same holds for every generic unitary representation (with $\pi_v^\alpha=\pi_v$ if it is not tempered), with the regularized relative characters $I_{\pi_v}, J_{\pi_v}$ defined above.
\end{theorem}

\begin{remark}
 The factor $\gamma_v^*(\eta_v,0,\psi_v)\cdot \frac{\AvgVol(T(k_v)) }{(k_v^\times:N_{k_v}^{E_v}E_v^\times)}$ is trivial if the volumes are chosen appropriately, see \eqref{AvgVol}.
\end{remark}

It is important to notice that our non-standard matching \emph{directly} implies this local result; the key point is that the transfer operator essentially preserves inner products, \cite[(5.2)]{SaBE1}:
\begin{equation}\label{preservesIP}
 \left< |\bullet|\mathcal Gf\right> = \gamma^*(\eta_v,0,\psi_v) \left<f\right>.
\end{equation}

This is a pleasant suggestion that such non-standard comparisons may be naturally suited for the proof of the global period conjecture \cite[Conjecture 17.4.1]{SV}.

\begin{proof}
 Combining \eqref{preservesIP} with the Plancherel formula \eqref{Plancherel-W} for $L^2(N_v\backslash G_v,\psi_v)$ we have, for $f\in \mathcal S(\mathcal Z_v)$,
 $$ \gamma^*(\eta_v,0,\psi_v) \left<f\right> = \left<|\bullet|\mathcal G f\right> = $$ 
 $$ = \int_{\widehat{G_v}^\temp} I_{\pi_v} (|\bullet|\mathcal G f) d\pi_v$$
 $$ = \int_{\widehat{G_v}^\temp} J'_{\pi_v^\alpha}(f) d\pi_v,$$
 where $\pi_v^\alpha$ is as in Proposition \ref{factorsthrough} and $J'_{\pi^\alpha_v}$ is the relative character asserted in that proposition. By uniqueness of the Plancherel formula \eqref{Plancherel}, and equality of formal degrees for the Jacquet--Langlands correspondence \eqref{Plmeasures} we get that for almost every $\pi_v$ the following are true:
 \begin{enumerate}
  \item  $$ J'_{\pi_v^\alpha} = \gamma^*(\eta_v,0,\psi_v)\cdot (-1)^\alpha  \frac{\AvgVol(T(k_v)) }{(k_v^\times:N_{k_v}^{E_v}E_v^\times)} J_{\pi_v^\alpha},$$ and:
  \item $J_{\pi_v^\beta} = 0$ for $\beta\ne \alpha$. (Notice: $\beta,\alpha$ depend on $\pi_v$.)
 \end{enumerate}

 The equality \eqref{eqcharmatching} actually extends to the whole continuous tempered spectrum, since both sides are continuous in the parameter, and more generally it extends to an equality of holomorphic functions in the parameter $\chi_v$ when the representation $\pi_v$ is of the form $I(\chi_v)$. In particular, it extends to generic, nontempered unitary representations (which are isomorphic to some $I(\chi_v)$ with non-unitary $\chi_v$). 
\end{proof}

\subsection{Unramified calculation}

We now compute the values of $J_{\pi_v}, I_{\pi_v}$ on the basic vectors $f_{\mathcal Z_v}^0$, resp.\ $f_{\mathcal W_v^s}^0$.

We recall from \eqref{basicvector} and \cite[\S 6.4]{SaBE1} that the basic vector $f_{\mathcal W_v^s}^0$ is equal to $\Vol(N\backslash G (\mathfrak o_v))^{-1}$ times the image of
$$ (H_s \star 1_{x_0K})\otimes 1_{y_0K} \in \mathcal S^s(\overline X\times X (k_v), \mathcal L_\psi\boxtimes\mathcal L_\psi^{-1}),$$
$X= N\backslash G$, where $H_s$ is series of elements in the unramified Hecke algebra with Satake transform
$$ \widehat{H_s} (\pi_v) = L(\pi_v,\frac{1}{2}+s) L(\pi_v\otimes \eta_v,\frac{1}{2}+s),$$
and $1_{x_0 K}, 1_{y_0 K}$ are the sections of $\mathcal L_\psi$, resp.\ $\mathcal L_\psi^{-1}$ (trivialized as functions on $G_v$ which vary by a character of $N_v$) which are equal to $1$ on $G(\mathfrak o_v)$ and equal to $0$ off $N_v G(\mathfrak o_v)$. 

Using the measures obtained by $k$-rational volume forms, at almost every place we have $\Vol(N\backslash G (\mathfrak o_v))^{-1} = \Vol(G(\mathfrak o_v))^{-1} = \zeta_v(2)$. For the function $1_{x_0K}\otimes 1_{y_0K}$ and an irreducible unramified representation $\pi_v$ one easily checks, at almost every place:
\begin{equation}\label{trr} I_{\pi_v}(1_{x_0K}\otimes 1_{y_0K}) = \frac{1}{\zeta_v(2) L(\pi_v,\Ad,1)}.
\end{equation}
\begin{remark}
 For a functional $\ell \in \Hom_{N_v}(\pi_v, \CC_\psi)$ corresponding to $I_{\pi_v}$ (i.e., $I_{\pi_v}$ is obtained from $\ell$ by Frobenius reciprocity and dualizing, both for $\pi_v$ and for $\overline{\pi_v}$) this calculation corresponds to the following; here $\phi_v^0$ is an unramified vector of norm one in $\pi_v$:
 \begin{equation}\label{unrN}
  |\ell(\phi_v^0)|^2 = \frac{\zeta_v(2)}{L(\pi_v,\Ad,1)}.
 \end{equation}
Indeed, comparing with \eqref{trr} one needs to divide by the square of the volume of $N\backslash G(\mathfrak o_v)$, which enters in the dualization.
\end{remark}

Combining all the above, for our basic function we will have:
\begin{equation}\label{unrW} I_{\pi_v} (f_{\mathcal W_v^s}^0) = \frac{L(\pi_v,\frac{1}{2}+s) L(\pi_v\otimes \eta_v,\frac{1}{2}+s)}{L(\pi_v,\Ad,1)}.\end{equation}

We now come to the calculation of $J_{\pi_v} (f_{\mathcal Z_v}^0)$. Using Theorem \ref{charmatching} and the remark following it, it is of course equal to $I_{\pi_v} (f_{\mathcal W_v^0}^0)$ (at almost every place). Let us also see this directly, and discuss what this means for local $T_v$-invariant functionals:

By definition, $f_{\mathcal Z_v}^0$ has value $1$ on $\mathcal B^\reg(\mathfrak o_v)$, therefore it is equal to $\Vol(N\backslash G (\mathfrak o_v))^{-1}=\Vol(G(\mathfrak o_v))^{-1}$ times what was denoted by $f_{\mathcal Z}^0$ in \cite[\S 6.4]{SaBE1}. In other words, it is $\Vol(G(\mathfrak o_v))^{-1}$ times the image of the characteristic function of $(T\backslash G \times T\backslash G)(\mathfrak o_v)$. (The reader can check from Table \eqref{table} that the regular value of $f_{\mathcal Z_v}^0$ is $1$, while the regular value of the image of this characteristic function would, of course, be equal to $\Vol(G(\mathfrak o_v))$.)

 Based on this, the formula
\begin{equation}\label{unrZ}
 J_{\pi_v}(f_{\mathcal Z_v}^0) = \frac{L(\pi_v,\frac{1}{2}) L(\pi_v\otimes \eta_v,\frac{1}{2})}{L(\pi_v,\Ad,1)}
\end{equation}
can be inferred from \cite[\S II]{Waldspurger-torus}, and also from the calculation of \cite[\S 5]{II} or the unramified Plancherel formula of \cite[Theorem 9.0.1]{SaSph}).

\begin{remark}
 For a functional $\ell \in \Hom_{T_v}(\pi_v, \CC)$ corresponding to $J_{\pi_v}$ (using local ``Tamagawa measures'' obtained from residually nonvanishing volume forms) this calculation corresponds to the following; here $\phi_v^0$ is an unramified vector of norm one in $\pi_v$:
 \begin{equation}\label{unrT}
  |\ell(\phi_v^0)|^2 = \frac{\zeta_v(2)L(\pi_v,\frac{1}{2}) L(\pi_v\otimes \eta_v,\frac{1}{2})}{L(\eta_v,1)^2 L(\pi_v,\Ad,1)}.
 \end{equation}
Indeed, comparing with the value of $J_{\pi_v}$ on the characteristic function of $(T\backslash G \times T\backslash G)(\mathfrak o_v)$, one needs to divide by the square of the volume of $T\backslash G(\mathfrak o_v)$, which is equal to $\frac{L(\eta_v,1)}{\zeta_v(2)}$.
\end{remark}

\subsection{Whittaker periods}\label{ssWhittaker}

In \S \ref{sslifts} we defined, for every $\varphi\in \widehat G^\aut_\Ram$ and $s\in \CC$ with $\Re(s)\gg 0$, the \emph{period relative character}
$$\mathcal I_\varphi: \mathcal S(\mathcal W^s(\adele)) \to \CC$$
or, more precisely, the corresponding functional-valued measure: $\mathcal I_\varphi d\varphi$, which is absolutely continuous with respect to Plancherel measure on $\hat G^\aut_\Ram$. 

We now fix a Plancherel measure on $\hat G^\aut_\Ram$, in order to have a well-defined $\mathcal I_\varphi$ (for almost all $\varphi$); we will take it to be equal to counting measure on the discrete spectrum. It is a corollary of Theorem \ref{spectral-KTF} that for an analytic section $s\mapsto f_s \in  \mathcal S(\mathcal W^s(\adele)) $, the function $s\mapsto \mathcal I_\varphi(f_s)$ is meromorphic in $s$ and analytic at $s=0$. The Euler factorization of this period is known (see, for instance, \cite[Theorem 4.1]{LM-Whittaker} or \cite[Theorem 18.3.1]{SV}):

\begin{theorem}
 For $\Re(s)\gg 0$ and $f_s=\prod_v f_{s,v}\in \mathcal S(\mathcal W^s(\adele))$ we have an Euler factorization:
 \begin{equation}\label{Euler-Whittaker}
  \mathcal I_\varphi(f_s) = \frac{1}{2} \prod_v' I_{\varphi_v}(f_{s,v})
 \end{equation}
 when $\varphi \in \hat G^\aut_\cusp$. For the continuous spectrum, $\mathcal I_\varphi(f_s)$ is proportional to $\frac{1}{2} \left(\prod_v' I_{\varphi_v}(f_s)\right)^*$ (to be explained).
\end{theorem}

Notice that the appearance of the factor $\frac{1}{2}$ is due to the fact that we define the inner product in $L^2([G])$ by integrating against the Tamagawa measure on $[G]=[\PGL_2]$. If, instead, we consider the cuspidal representation as one of $\GL_2$ with trivial central character, and define the inner product by integrating against the Tamagawa measure on $\GL_2(k)\backslash \GL_2(\adele)^1$, where $\GL_2(\adele)^1$ is the set of elements with $|\det(g)|=1$, then this factor would not appear.

The symbol $\prod_v'$ means that the product should be interpreted in terms of partial $L$-functions. Namely, we saw in \eqref{unrW} that for $v$ outside a large enough finite set $S$ of places we have:
$$ I_{\varphi_v}(f_{s,v}) =  \frac{L(\pi_v,\frac{1}{2}+s)L(\pi_v\otimes\eta_v,\frac{1}{2}+s)}{L(\pi_v,\Ad, 1)},$$
and in the cuspidal case we interpret:
\begin{equation}\label{Eulerproduct} \prod_v' I_{\varphi_v}(f_{s,v}): = \frac{L^S(\pi,\frac{1}{2}+s)L^S(\pi\otimes\eta_v,\frac{1}{2}+s)}{L^S(\pi,\Ad, 1)} \prod_{v\in S} I_{\varphi_v}(f_{s,v}).\end{equation}

For the continuous spectrum, $L^S(\pi,\Ad, t)$ has a simple pole at $t=1$, and we let $\left(\prod_v' I_{\varphi_v}(f_s)\right)^*$ be the dominant term in the corresponding Taylor series. 

The statement in the continuous case is actually a formal corollary of multiplicity one for Whittaker models. For a more precise and canonical relation between $I_\varphi$ and the ``Euler product'' \eqref{Euler-Whittaker} in the continuous case, see \cite[\S 18.1]{SV}.

By meromorphic continuation, this formula continues to hold for every $s$ where $\mathcal I_\varphi$ admits analytic continuation, in particular at $s=0$.

On the other hand, by Theorem \ref{comparison-spectral} we have, for those $\varphi$ and $f\in \mathcal S(\mathcal Z(\adele))$:
\begin{equation}\label{globaleq}
 \mathcal J_\varphi (f) = \mathcal I_\varphi(|\bullet|\mathcal G f).
\end{equation}

We will now prove Proposition \ref{factorsthrough}. We will use the previous theorem, but strictly speaking it is not necessary, as multiplicity one for Whittaker models implies that the restriction of $\mathcal I_\varphi$ to any place is a multiple of $I_{\pi_v}$ (where $\pi_v$ is the local component of the automorphic representation attached to $\varphi$), which is all that we use.

\begin{proof}[Proof of Proposition \ref{factorsthrough}]
We first claim the following: Suppose that $\alpha$ denotes the class of a $T_v$-torsor and $\pi_v^\alpha$ a Jacquet--Langlands lift of the given tempered representation $\pi_v$ to $G_v^\alpha$ which is $T_v^\alpha$-distinguished (equivalently, by Proposition \ref{nonvanishing}, such that $J_{\pi_v^\alpha}\ne 0$). Then, we claim that $\left(|\bullet|\mathcal G\right)^*I_{\pi_v}$ has to be a multiple of $J_{\pi_v^\alpha}$. 

To prove this claim we may, without loss of generality, assume that $T$ is globally nonsplit; indeed, $T_v$ can be realized as the local factor of a globally nonsplit torus. Then, for a dense set of such $\pi_v^\alpha$'s (i.e., for every such discrete series, and a dense set of unitary principal series), the representation $\pi_v^\alpha$ can be globalized to some non-residual automorphic representation $\pi^\beta$ (corresponding to $\varphi\in \hat G^\aut_\Ram$) where the $T$-period is nonzero, in particular with $\mathcal J_\varphi\ne 0$. For $\pi_v^\alpha$ discrete this follows, for instance, from \cite[Theorem 16.3.2]{SV}, combined with the fact that the local discrete spectrum is \emph{automorphically isolated}.
For $\pi_v^\alpha$ in the continuous spectrum, one can easily construct an Eisenstein series whose $T$-period integral does not vanish identically with the continuous parameter of the Eisenstein series. It is enough to prove that $\left(|\bullet|\mathcal G\right)^*I_{\pi_v}$ is a multiple of $J_{\pi_v^\alpha}$ for this dense set of $\pi_v^\alpha$'s; by continuity, it will follow for all. For such a $\pi_v^\alpha$, fix a global parameter $\varphi\in \hat G^\aut_\Ram$ obtained from this construction.

We may then choose $f^v\in \bigotimes'_{w\notin v}\mathcal S(\mathcal Z_w)$ such that, viewed as a functional on the variable $f_v\in \mathcal S(\mathcal Z_v)$ with $f=f^v\otimes f_v$, the expression \eqref{globaleq} is nonzero, and hence a nonzero multiple of $\left(|\bullet|\mathcal G\right)^*I_{\pi_v}$. On the other hand, the left hand side of \eqref{globaleq}, considered as a functional on $\mathcal S(\mathcal Z_v)$ by fixing $f^v$, is by construction a linear combination of relative characters corresponding to Jacquet--Langlands lifts of $\pi_v$ to groups of the form $G^{\beta'}_v$, where $\beta'$ is a non-trivial $T$-torsor over $k$. 

If $T_v$ is split or when $\pi_v^\alpha$ is in the continuous spectrum, the only Jacquet--Langlands lift of $\pi_v$ is $\pi_v=\pi_v^\alpha$ itself, of course, so the claim is proven in that case. In the nonsplit case and for $\pi_v^\alpha$ discrete, to show that there is a \emph{unique} $\pi_v^\alpha$ contributing to the pullback, it suffices to choose the function $f^v$ carefully: more precisely, since any two distinct inner forms over $k$ differ at at least two places, and vice versa we have the Hasse principle: the local forms (outside of $v$!) 
determine the global one, we can find a test function $f^v\in \bigotimes'_{w\notin v}\mathcal S(\mathcal Z_w)$ so that the only global form $G^{\beta'}$ which contributes to the period relative character $\mathcal J_\varphi(f^v\otimes f_v)$ is the form $G^\beta$ of the above construction; and hence (by strong multiplicity one for $G^\beta$) the only automorphic representation which contributes is the Jacquet--Langlands lift $\pi^\beta$ with $\pi^\beta_v = \pi_v^\alpha$. Hence, the functional
$$f_v\mapsto \mathcal J_\varphi(f^v\otimes f_v)$$
on $\mathcal S(\mathcal Z_v)$ has to be a relative character for $\pi_v^\alpha$.

This completes the proof of the initial claim. We have now proven Proposition \ref{factorsthrough} for those $\pi_v$ which admit a $T_v$-distinguished Jacquet--Langlands lift $\pi_v^\alpha$, for some $\alpha$ (let us call them ``special'' for the purpose of this proof). The proof of the proposition will be complete if we show that Plancherel-almost every $\pi_v$ is special. Employing \eqref{preservesIP} and the Plancherel formula for the spaces $L^2(T_v^\alpha\backslash G_v^\alpha)$, we get a decomposition of the inner product on $L^2(N_v\backslash G_v,\psi_v)$ as follows:
$$ \left< |\bullet|\mathcal G f\right> = \sum_\alpha \int_{\widehat{G_v^\alpha}^\temp} J'_{\pi_v^\alpha}(f) d_{\pi_v^\alpha}.$$

The right-hand side, on the other hand, can by what we just proved be expressed as an integral of multiples of the $I_{\pi_v}$'s over all ``special'' $\pi_v$'s. Uniqueness of the Plancherel formula now implies that Plancherel-almost all $\pi_v$ are special. 
\end{proof}

\subsection{Toric periods}

This is the theorem of Hecke and Waldspurger \cite{Waldspurger-torus} for $\PGL_2$:
\begin{theorem}\label{thmWaldspurger}
 For $\varphi\in \hat G^\aut_\cusp$ we have, as functionals on $\mathcal S(\mathcal Z(\adele))$:
 \begin{equation}\label{eqWaldspurger}
 \mathcal J_\varphi = \frac{1}{2} \prod^\prime_v J_{\varphi_v}.
 \end{equation}
\end{theorem}

The product is again interpreted as in \eqref{Eulerproduct}.
The analytic continuation of the implicit partial $L$-function to $s=0$ is guaranteed by the analytic continuation of $\mathcal I_\varphi$ as a functional on $\mathcal S(\mathcal W^s(\adele))$ (Corollary \ref{analyticcont-cuspidal}). 

\begin{remark}
 By dualizing, this says that, for a cuspidal automorphic representation $\pi$ of $G^\alpha$ (for some global $T$-torsor $\alpha$) and a vector $\phi= \otimes_v \phi_v \in \pi \simeq \otimes^\prime_v \pi_v$:
 \begin{equation}\left|\int_{[T]} \phi(t) dt\right|^2 = \frac{1}{2}\prod_v' \int_{T(k_v)} \left< \pi_v(t) \phi_v,\overline{\phi_v}\right> dt,
 \end{equation}
where the Euler product should be interpreted using the partial $L$-values corresponding to \eqref{unrT}. Here the measures used are Tamagawa measures defined via global volume forms; notice that in the split case the factor $L^S(\eta,1)^2=\zeta^S(1)^2$ which will be obtained from the denominator of \eqref{unrT} is undefined, and so is the Tamagawa measure on $[T]^2$; multiplying both sides locally by convergence factors $\zeta_v(1)^2$ we get a meaningful expression. Notice that $\left<\, , \,  \right>$ keeps denoting a bilinear form here, and we have used the unitary structure on $\pi$ to identify $\tilde\pi=\bar\pi$.
\end{remark}

\begin{proof}
As was mentioned after Theorem \ref{charmatching}, for suitable choices of measures (which can be taken to factorize global Tamagawa measures) the statement of the theorem reads:
$$\left(|\bullet|\mathcal G\right)^*I_{\pi_v} = (-1)^{\alpha_v}   J_{\pi^\alpha_v}.$$

We recall that the collection of torsors $\alpha_v$ is the one afforded by Proposition \ref{factorsthrough}.

By \eqref{globaleq} the statement of the theorem is true when $L^S(\pi,\frac{1}{2})L^S(\pi\otimes \eta,\frac{1}{2})=0$, in which case both sides are zero. Therefore, assume that $\varphi$ is such that $L^S(\pi,\frac{1}{2})L^S(\pi\otimes \eta,\frac{1}{2})\ne 0$, equivalently: $\mathcal I_\varphi$ applied to the space $\mathcal S(\mathcal W^0(\adele))$ is nonzero.

We know from Lemma \ref{globaltorsors} that $\mathcal J_\varphi$ is supported on ``global torsors'', and from Proposition \ref{factorsthrough} that each local factor $J_{\varphi_v}$ is supported on a unique local torsor. It follows that for every such $\varphi$ there is a unique \emph{global} torsor $\alpha$ such that $J_{\varphi_v}$ is supported on $\alpha_v$ for all $v$.\footnote{This also follows by characterizing the local torsor of Proposition \ref{factorsthrough} in terms of epsilon factors, Theorem \ref{thmepsilon}.}

Thus, the collection $(\alpha_v)_v$ corresponds to a global torsor, and the product of factors $(-1)^{\alpha_v}$ is trivial (the torsor is nontrivial at an even number of places).
\end{proof}

By dualizing and using \eqref{Tperiod}, we get the statement of Theorem \ref{Waldspurgertheorem}. Notice that the unramified factors for these local periods were given in \eqref{unrT}.

\subsection{Determination of the distinguished representation}

Finally, we return to the local setting to discuss one of the most mysterious issues of this field, the relation between $\epsilon$-factors and distinguished representations inside of a local $L$-packet. We would like, for any $\pi_v\in \widehat{G_v}^\disc$, to describe the local torsor $\alpha_v$ of Proposition \ref{factorsthrough} for which the corresponding Jacquet--Langlands lift $\varpi_v^\alpha$ is distinguished by the torus $T_v$ (equivalently: $J_{\pi_v}\ne 0$). (We only focus on discrete series, because for the continuous spectrum the answer is obviously the trivial torsor.)

The answer is well-known from the work of Tunnell \cite{Tunnell}, Saito \cite{Saito} (s.\ also Prasad \cite{Prasad-Saito}) in terms of $\epsilon$-factors. The question is whether this can also be seen from trace formula-theoretic considerations. This is indeed the case, although we do not have a completely independent way of verifying that the $\epsilon$-factor that we will define is the correct one -- this will follow from the work of Jacquet \cite{JW1} in combination with Jacquet--Langlands \cite{JL}. Thus, this subsection is more a comment on these papers than an independent treatment of the question.

Consider the automorphism of $\mathcal S(\mathcal Z(k_v))$ afforded by Proposition \ref{propfetorus} (with $s=0$):
$$ f\mapsto \mathcal R f(\xi):= \eta_v(\xi) \eta_v(\xi+1) f(\xi).$$
Of course, if $\eta_v=1$ then this is the identity, but if $\eta_v\ne 1$ then we have two $T$-torsors over $k_v$, which give rise to the inner forms $G$, $G^\alpha$. Recall that we have canonical maps
\begin{equation}\label{quot1}
  T(k_v)\backslash G(k_v)/T(k_v) \to \mathcal B(k_v),
\end{equation}
\begin{equation}\label{quot2}
T^\alpha(k_v)\backslash G^\alpha(k_v)/T^\alpha(k_v) \to \mathcal B(k_v),
\end{equation}
induced from the isomorphism of stacks $T\backslash G/T \simeq T^\alpha\backslash G^\alpha/T^\alpha$ \cite[(3.4)]{SaBE1}. The following is easy to compute:

\begin{lemma}\label{lemmaimage}
 A point $\xi\in \mathcal B^\reg(k_v) = k_v \smallsetminus\{-1,0\}$ belongs to the image of \eqref{quot1}, resp.\ \eqref{quot2}, if and only if $\eta_v(\xi)\eta_v(\xi+1) = 1$, resp.\ $\eta_v(\xi)\eta_v(\xi+1) = -1$.
\end{lemma}

Now we can prove the result of Tunnell and Saito, based on the aforementioned works of Jacquet and Langlands:

\begin{theorem}\label{thmepsilon}
Let $\pi_v$ be an irreducible, generic, unitary representation of $G_v$, and let $\epsilon_T(\pi_v, \frac{1}{2}) = \epsilon(\pi_v,\frac{1}{2},\psi_v) \epsilon(\pi_v\otimes\eta_v,\frac{1}{2}, \psi_v)$ be the central value of its $\epsilon$-factor.

Then $\epsilon_T(\pi_v, \frac{1}{2})$ is independent of $\psi_v$, and the local torsor $\alpha$ of Proposition \ref{factorsthrough} is trivial if $\epsilon_T(\pi_v, \frac{1}{2}) = 1$, and nontrivial if $\epsilon_T(\pi_v, \frac{1}{2}) = -1$. In other words (see Proposition \ref{nonvanishing}), $\pi_v$ is $T_v$-distinguished iff $\epsilon_T(\pi_v, \frac{1}{2}) = 1$, and its Jacquet--Langlands lift $\pi_v^\alpha$ to the nontrivial inner twist is $T_v$-distinguished iff $\epsilon_T(\pi_v,\frac{1}{2})=-1$.
\end{theorem}

\begin{remark}
 The result of Tunnell and Saito is more general; for example, it extends to all generic irreducible representations, and nontrivial characters of $T_v$.
\end{remark}

\begin{proof}
 The following is an obvious reinterpretation of the local functional equation of the special case of \cite[Theorem 2.18]{JL} with $s=\frac{1}{2}$ (in the coordinates of \cite{JL}):
 \begin{lemma}
Let $A$ be the split torus of diagonal elements in $\PGL_2$, and let $w=\left(\begin{array}{cc} &1 \\ -1 \end{array}\right)$.

Consider the automorphism $R$ of order $2$ of $C^\infty(A_v\backslash G_v)$ given by $R\phi(g) = \phi(wg)$. Let $\phi$ be in the image of $\pi_v\hookrightarrow C^\infty(A_v\backslash G_v)$. Then:
$$ R\phi(1) = \epsilon(\pi_v,\frac{1}{2},\psi_v) \phi(1).$$

Similarly if we replace $\pi_v$ by $\pi_v\otimes\eta_v$.
 \end{lemma}

 Indeed, a priori this is true with $\epsilon(\pi_v,\frac{1}{2},\psi_v)$ replaced by $\frac{L(\tilde \pi_v,\frac{1}{2})}{L(\pi_v,\frac{1}{2})}\epsilon(\pi_v,\frac{1}{2},\psi_v)$, but since representations of $\PGL_2$ are self-dual, the quotient of local $L$-functions is equal to $1$.
 
 This shows, in particular, that $\epsilon(\pi_v,\frac{1}{2},\psi_v)=\pm 1$ and is independent of $\psi_v$, and we will henceforth denote it by $\epsilon(\pi_v,\frac{1}{2})$ (and similarly for $\pi_v\otimes\eta_v$).
 
 By dualizing, if we consider a relative character:
 $$ J'_{\pi_v}: C_c^\infty(A_v\backslash G_v) \otimes C_c^\infty(A_v\backslash G_v,\eta_v) \to \pi_v\otimes\tilde\pi_v \to \CC$$
 then the automorphism of order $2$ taking each element $\phi_1\otimes\phi_2$ to $\phi_1(w\bullet) \otimes \phi_2(w\bullet)$ acts on $J'_{\pi_v}$ by the scalar $\epsilon_T(\pi_v) =\epsilon(\pi_v,\frac{1}{2}) \epsilon(\pi_v\otimes\eta_v, \frac{1}{2})$.
 
By \cite{JW1}, this can be applied to the automorphism $\mathcal R$ of $\mathcal S(\mathcal Z_v)$ (cf.\ the proof of Proposition \ref{propfeKuz}), and the relative character $J_{\pi^\alpha_v}$:

\begin{corollary}
 The automorphism $\mathcal R$ of $\mathcal S(\mathcal Z_v)$ acts on $J_{\pi^\alpha_v}$ by the scalar $\epsilon_T(\pi_v,\frac{1}{2})$.
\end{corollary}

On the other hand, the automorphism is described explicitly in Proposition \ref{propfetorus}. By Lemma \ref{lemmaimage} we deduce that the support of the distribution $J_{\pi^\alpha_v}$ is the image of $k_v$-points corresponding to the torsor described in the statement of the theorem. 
\end{proof}

\appendix

\section{Families of locally multiplicative functions} \label{sec:families}

\subsection{General formalism}\label{families-general}
In this appendix we will discuss the analytic structure on certain families $t\mapsto \mathcal A^t$ of Fr\'echet spaces of functions on a local field $F$, including notions of \emph{analytic sections} and of \emph{polynomial seminorms} and \emph{polynomial growth} or \emph{rapid decay} on (bounded) vertical strips. As in the rest of the paper, we fix a non-trivial complex additive character of $F$ and a self-dual Haar measure, following the conventions for Fourier transform that were explained in \S \ref{ssnotation}.

We will define analytic sections by embedding our spaces into \emph{Fr\'echet bundles} over the parameter space of $t$. The parameter space of $t$ is $\CC/2\pi i\Z  \log q$ in the non-Archimedean case, where $q$ is the residual degree of $F$. For Archimedean places (where most of the following work is focused), we set $\log q:= \infty$, and hence the parameter space is $\CC$. 

We recall that (the total space of) a \emph{Fr\'echet bundle} over a complex manifold $M$ is a complex Fr\'echet manifold $N$, together with a holomorphic map $\pi: N\to M$, with the fibers having the structure of a vector space and an open covering of $M$ by neighborhoods $U$ such that the restriction of $N$ over $U$ is (biholomorphically and linearly) isomorphic to a direct product $U\times A$, where $A$ is a Fr\'echet space. 
In particular, for a Fr\'echet bundle it makes sense to talk about analytic sections $M\to N$ or (weakly) analytic families of functionals (i.e., sections $L: M\to N^*$ which are weakly analytic: under an isomorphism with $U\times A$, locally, $\left< L(m), a\right>$ is analytic on $m\in U$ for every $a\in A$). Thus, once we embed the spaces $\mathcal A^t$ into the fibers $N_t$ of a Fr\'echet bundle over $\CC/\frac{2\pi i}{\log q} \Z$, by an \emph{analytic section} $t\mapsto f_t\in \mathcal A^t$ we will mean a section which is analytic as a section into $N$, and by an \emph{analytic family of functionals} $t\mapsto L_t\in (\mathcal A^t)^*$ we will mean one such that $\left<f_t,L_t\right>$ is analytic for every analytic section $t\mapsto f_t$. 

The embedding into a Fr\'echet bundle will depend on some choices, but it is easy to check that the notions that we define do not. Namely, our spaces will be global sections of Fr\'echet cosheaves over $\mathbb P^1(F)$, and we take a cover of $\PP^1(F)$ by small open (semialgebraic) subsets $U_i$, such that the restrictions $\mathcal A^t(U_i)$ of our Schwartz cosheaves are of the form that we will describe below, and have the structure of Fr\'echet bundles. Then we choose a partition of unity by Schwartz functions subordinate to the $U_i$'s, and multiplying each element of $\mathcal A^t$ by those we get an injective map:
\begin{equation}\label{intoFb} \mathcal A^t\to \bigoplus_i \mathcal A^t(U_i),
\end{equation}
which splits the obvious (extension) map in the other direction. This gives us the notion of analytic sections that we need. We will also endow the spaces $\mathcal A^t(U_i)$ with notions of ``polynomial seminorms'' and sections of ``polynomial growth'' or ``rapid decay''; then the same notions carry over to sections of $t\mapsto \mathcal A^t$ by the above embedding. (The ``polynomial'' and ``rapid'' notions refer, implicitly, to behavior on \emph{bounded vertical strips}; they will not be polynomial or of rapid decay on the whole complex plane. Moreover, in the non-Archimedean case, ``bounded vertical strips'' are compact since $t\in \CC/\frac{2\pi i}{\log q} \Z$, so there will be no notion of rapid decay, and every continuous section is automatically of polynomial growth; this comment will be implicit, and will not be repeated, every time there is a discussion of those notions.)

Now we come to describing, axiomatically, the cosheaves $\mathcal A^t$ over open sets $U_i$. By choosing the $U_i$'s small enough, the cosheaves will coincide with the cosheaves of Schwartz functions away from a point of $U_i$, which up to an automorphism of $\PP^1$ we can identify with $0$. In a neighborhood $U$ of zero, we consider two possibilities for the behavior of these functions:

\begin{enumerate}
 \item either $\mathcal A^t(U) = |\bullet|^t \mathcal S(U)\cdot K$ for all $t$, where $K$ is a fixed function not depending on $t$; in this case, we also assume that $K$ is such that the kernel of the map: $f \mapsto f\cdot K$ is closed in $\mathcal S(U)$; 
 \item or they are defined in such a way that their fibers over $0$
 are annihilated by the operator: 
\begin{equation}
 (\Id - \eta(a) |a|^{-t-\frac{1}{2}} a\cdot  )(\Id - |a|^{-\frac{1}{2}} a\cdot ),
\end{equation} for every $a\in F^\times$, where $a\cdot$ we denote the normalized action of $F^\times$ on functions on $\mathcal B$:
\begin{equation}
 (a\cdot f)(x) := |a|^\frac{1}{2} f(ax).
\end{equation}

More precisely, in this second case our functions will be of the form: 
$$C_1(\xi) +C_2(\xi) \eta(\xi) |\xi|^t$$
(with $\eta$ the quadratic character associated to a quadratic etale algebra $E$ over $F$) in a neighborhood of $\xi=0$, where $C_1$ and $C_2$ extend to smooth functions in a neighborhood of zero, except:
\begin{itemize}
 \item when
 \begin{equation}\label{exceptions2} t = 0 \mbox{ and }\eta=1,
 \end{equation}
 in which case the functions have the form: $$C_1(\xi)+C_2(\xi)\log|\xi|;$$
 \item when 
 \begin{eqnarray}\label{exceptions3} t\in 2\Z, \eta=1 \mbox{ and }F=\RR, \nonumber
  \\ \mbox{ or } t\in (2\Z+1), \eta\ne 1 \mbox{ and } F  = \RR,\\
  \mbox{ or }t\in \Z \mbox{ and }F  = \CC, \nonumber
 \end{eqnarray} 
  in which case the functions have the form: 
$$\begin{cases}
   C_1(\xi)+C_2(\xi)\eta(\xi)|\xi|^t\log|\xi|, & \mbox{ when } t\ge 0;\\
   C_1(\xi)\eta(\xi)|\xi|^t+C_2(\xi)\log|\xi|, & \mbox{ when } t<0.
\end{cases}$$
\end{itemize}
\end{enumerate}

\subsection{Fr\'echet bundle structure and polynomial seminorms in the first case}

In the first case, by multiplying with $|\bullet|^{-t}$ we identify all spaces $\mathcal A^t(U)$ with the Fr\'echet space\footnote{Recall that in the non-Archimedean space we work throughout with ``almost smooth'' functions, cf.\ \cite[Appendix A]{SaBE1}, for the sake of uniformity; with usual smooth functions we get LF-spaces instead.}  $\mathcal S(U)/\{f: fK=0\}$. A \emph{polynomial family of seminorms} will be a family $t\mapsto \rho_t$ of seminorms on $\mathcal A^t(U)$ which on bounded vertical strips is bounded by seminorms of the form
$$|P(t)| \cdot \rho,$$
where $\rho$ is a fixed seminorm of this Fr\'echet space and $P$ is a polynomial in $t$. A section $t\mapsto f_t \in \mathcal A^t(U)$ will be said to be of \emph{polynomial growth}, resp.\ \emph{rapid decay} on a bounded vertical strip $V:= \{ t| \sigma_1\le \Re(t)\le \sigma_2\}$ if it is so with respect to polynomial seminorms, i.e., $\rho_t(f_t)$ should be bounded by a polynomial, resp.\ absolutely bounded on $V$ for every polynomial family of seminorms $\rho_t$.

\subsection{Fr\'echet bundle structure and polynomial seminorms in the second case}

In the second case, for generic $t$, the space $\mathcal A^t(U)$ also has a natural Fr\'echet space structure, by identifying it as the quotient
\begin{equation}\label{Schwartzquotient}
(\mathcal S(U)\oplus\mathcal S(U) )/\mathcal S(U),
\end{equation}
where $\mathcal S(U)$ is embedded by: $f\mapsto (f, -\eta(\bullet)|\bullet|^t f)$. The Fr\'echet structure in the case $t=0, \eta=1$ can again be described as the quotient of $\mathcal S(U)\oplus\mathcal S(U) $, as discussed in \S \ref{ssPoissonbaby}, and similarly for Archimedean cases when $t\in \Z$.

However, to describe these spaces as a Fr\'echet bundle as $t$ varies, we will rework the definition of the Fr\'echet topology for every $t$ in terms of Fourier transforms. In fact, we may embed $\mathcal A^t(U)$ in the larger space of functions on $F^\times$ which coincide with Schwartz functions away from a compact neighborhood of zero; we call this space the ``model'' for $\mathcal A^t$. (Thus, strictly speaking, we compose the embedding \eqref{intoFb} with the embeddings into these ``model'' spaces, and we endow those with the desired structures.) With the exception of the values of $t$ described below, we can consider elements of the model $\mathcal A^t$ as tempered distributions on $F$ (recall that we have fixed a Haar measure on $F$); for example, the distribution $(C_1(\xi)+ C_2(\xi)\eta(\xi)|\xi|^t)d\xi$ is a well-defined measure for $\Re(t)\gg 0$, and has meromorphic continuation to all but countably many values of $t$ (finitely many in the $p$-adic case). Thus, we can apply Fourier transform to them. The exceptions are when 
$t+1$ is a pole of the local zeta function, i.e.:
\begin{eqnarray}\label{exceptions}
 t = -1 \mbox{ and }\eta=1; \nonumber \\
 t\in -2\mathbb N-1, \eta=1, F=\RR ; \nonumber \\ 
 t\in -2\mathbb N-2, \eta\ne 1, F  = \RR; \nonumber \\ 
 t\in -\mathbb N-1, F  = \CC.
\end{eqnarray}

\begin{lemma}\label{lemmaFourier}
Let $t$ be outside the values of \eqref{exceptions}. Fourier transform defines a topological isomorphism between the model $\mathcal A^t$ and the Fr\'echet space $\mathcal B^t$ of those smooth functions on $F$ which in a neighborhood of infinity are equal to $|\xi|^{-t-1} \eta(\xi)  h\left(\frac{1}{\xi}\right)$, for some $h\in \mathcal S(F)$ (with the obvious topology, which can be inferred from the above discussion). Moreover, Fourier transform descends to a topological isomorphism between $\mathcal A^t/\mathcal S(F)$ and the stalk of $\mathcal B^t$ at $\infty$ (with the obvious topology, given by the derivatives of $h$ at $0$). 
\end{lemma}

We will prove this lemma in a moment. Notice that multiplication by a smooth, nonvanishing function which is equal to $1$ in a neighborhood of $0$ and equal to $|\bullet|^t$ in a neighborhood of $\infty$ \emph{identifies} the spaces $\mathcal B^t$ as Fr\'echet spaces, so indeed they form a Fr\'echet bundle. This allows us to \emph{identify} the association $t\mapsto \mathcal A^t$ with the association $t\mapsto \mathcal B^t$ over values of $t$ different from \eqref{exceptions}, which has an obvious, natural structure of a Fr\'echet bundle. We pull back this Fr\'echet bundle structure to the spaces $\mathcal A^t$, for example: analytic sections into $\mathcal A^t$ (when $t$ is not contained in \eqref{exceptions}) are those sections whose Fourier transforms are analytic into $\mathcal B^t$.

One can show that the operation ``multiplication by $\eta(\bullet)|\bullet|^{-t}$'', which takes $\mathcal A^t$ to $\mathcal A^{-t}$, preserves this structure of a Fr\'echet bundle for values of $t$ outside of \eqref{exceptions}; this \emph{allows} us to extend the Fr\'echet bundle structure to arbitrary $t$. More precisely, this result is an easy corollary of the following direct characterization of analytic sections in terms of $\mathcal A^t$:

\begin{lemma}\label{analyticsections}
For $t$ not among the values \eqref{exceptions2}, \eqref{exceptions3}, the analytic sections into $\mathcal A^t$ can be described as those sections of the form $t\mapsto C_{1,t}(\xi) +C_{2,t}(\xi)\eta(\xi)|\xi|^t$, where $C_{i,t}$ are analytic maps into $\mathcal S(F)$ (of course, not uniquely defined in terms of the section). 

If $\eta=1$, such a section, defined in a punctured neighborhood of $t=0$, extends to an analytic section at $t=0$ if and only if $C_{1,t}$ and $C_{2,t}$ have at most simple poles with opposite residues at $t=0$. 
\end{lemma}

There is a similar description for the values of $t$ as in \eqref{exceptions3}, but it is left to the reader. 

\begin{proof}[Proof of Lemma \ref{lemmaFourier}]
 This is a generalization of \cite[Corollary 2.11]{SaBE1}, and the proof is similar, so we only emphasize the necessary additions.

 If $\Phi\in\mathcal S(F^2)$ then consider the following function in one variable, representing the orbital integrals of $\Phi$ with respect to the hyperbolic action of $F^\times$ against a character:
\begin{equation}\label{ttwisted} f_t(\xi) = \int_{F^\times} \Phi(\xi a, a^{-1}) \eta(a) |a|^{-t} d^\times a.
\end{equation}

One shows (as in \cite[Proposition 2.5]{SaBE1}) that $f_t\in \mathcal A^t$, and this allows one to identify the (model) space $\mathcal A^t$ with the Fr\'echet space of twisted coinvariants (with respect to the character $\eta(\bullet) |\bullet|^t$) of $\mathcal S(F^2)$. This applies to \emph{all} $t\in \CC/2 \pi i\Z \log q$.

Usual Fourier transform on $\mathcal S(F^2)$ is an automorphism which is anti-equivariant with respect to the hyperbolic $F^\times$-action, and if $f_t'$ denotes the corresponding function for the Fourier transform of $\Phi$, then the map $f_t \mapsto f_{-t}'$ is a (well-defined) topological isomorphism between $\mathcal A^t$ and $\mathcal A^{-t}$. Moreover, as in \cite[Proposition 2.10]{SaBE1}, it can be explicitly described (when $t$ does not belong to \eqref{exceptions}) by the operator
\begin{equation}\label{toperator}|\bullet|^{-t}\mathcal G_t:= |\bullet|^{-t}\mathcal F\circ \iota_t \circ \mathcal F,
\end{equation}
where $\mathcal F$ is usual Fourier transform in one variable and:
\begin{equation}\label{defiotat2}
\iota_t: f\mapsto \frac{\eta(\bullet)}{|\bullet|^{t+1}}f\left(\frac{1}{\bullet}\right). 
\end{equation}

The operator \eqref{defiotat2} is an automorphism of the space of continuous functions $h$ on $F$ such that $\lim_{\xi\to\infty} h(\xi) \eta(\xi) |\xi|^{t+1}$ exists, and clearly the Fourier transforms of functions of the form $f_t$ as above are smooth, contain the space $\mathcal S(F)$, and as in \cite[Lemma 2.9]{SaBE1} belong to this space of continuous functions. The fact that the operator \eqref{toperator} defines an isomorphism between $\mathcal A^t$ and $\mathcal A^{-t}$ now \emph{identifies} the Fourier transform of $\mathcal A^t$ as $\mathcal S(F)+\iota_t \mathcal S(F)=$ the space $\mathcal B^t$.

To show that Fourier transform is continuous with respect to the stated topologies, we first notice that this is obvious on the subspace $\mathcal S(F)\subset \mathcal A^t$. Since $\mathcal A^t = \mathcal S(F) + \mathcal G_t \mathcal S(F)$ (indeed: $\mathcal A^t = |\bullet|^t \mathcal A^{-t}$), this implies continuity.

\end{proof}

\begin{proof}[Proof of Lemma \ref{analyticsections}.]

Recall that the notion of ``analytic section'' is defined using the obvious Fr\'echet bundle structure on $\mathcal B^t$, when $t$ is not among the exceptions \eqref{exceptions}. Let $\Phi\in \mathcal S(F^2)$. Its image in $\mathcal A^t$ via $t$-twisted orbital integrals \eqref{ttwisted}, composed with Fourier transform, gives an element $F_t\in \mathcal B^t$. Following the proof of \cite[Proposition 2.10]{SaBE1}, it is easy to relate the two:
\begin{equation}\label{StoB}
 F_t(\xi) = \int_{F^\times} \hat \Phi^1(\xi a , a) \eta(a) |a|^{t+1} d^\times a,
\end{equation}
where $\hat \Phi^1$ denotes the Fourier transform of $\Phi$ in the first variable, and the integral should be understood as a zeta integral (in particular, it makes sense when $t$ is not in \eqref{exceptions}). 

The following diagram summarizes the relations between the various spaces and transforms:
\begin{equation}\label{tdiagram}
\xymatrix{
\mathcal S(F)^2 \ar[rrr]^{\mathcal F} \ar[d]_{\eqref{ttwisted}} \ar[dr]^{\eqref{StoB}}   &&&  \ar[d]^{\eqref{ttwisted}} \mathcal S(F^2) \\
\mathcal A^t \ar[r]^{\mathcal F}_{t \notin \eqref{exceptions}} & \mathcal B^t \ar[r]^{\iota_t} & \mathcal B^t \ar[r]^{|\bullet|^{-t}\mathcal F}_{t \notin \eqref{exceptions}} & \mathcal A^{-t}
}
\end{equation}

The relation \eqref{StoB} shows that any analytic section into $\mathcal B^t$ lifts to an analytic section $t\mapsto \hat \Phi^1_t \in \mathcal S(F^2\smallsetminus\{0\})$ and hence, via inverse Fourier transform the first variable, to an analytic section into $t\mapsto \Phi_t\in \mathcal S(F^2)$. Vice versa, let such an analytic section into $\mathcal S(F^2)$ be given, so $t\mapsto \hat \Phi^1_t \in \mathcal S(F^2)$ is also analytic. Taking into account that the zeta integrals (for $t\notin \eqref{exceptions}$) form an analytic family of distributions on $\mathcal S(F)$, we conclude that the image of $t\mapsto \hat \Phi^1_t$ in $\mathcal B^t$ via \eqref{StoB} is an analytic section.

Now it remains to show that analytic sections into $\mathcal S(F^2)$ descend to precisely those sections into $\mathcal A^t$ as in the statement of the lemma. Let $V^t$ denote the stalk of $\mathcal A^t$ at $\xi=0$; for $t$ outside the values of \eqref{exceptions2}, \eqref{exceptions3}, the germ of an element of the form $C_{1,t}(\xi) +C_{2,t}(\xi)\eta(\xi)|\xi|^t$ is determined by all the derivatives at zero of the smooth functions $C_{1,t}$ and $C_{2,t}$.

One can easily relate these derivatives to zeta integrals of derivatives of the original function $\Phi_t$, restricted to the two axes. For example, if $F=\RR$, from \eqref{ttwisted} we easily deduce that, for $t$ away from integral real values, 
$$C_{1,t}^{(n)}(0) = \int_{F^\times} \frac{\partial}{\partial x} \Phi(0, a^{-1}) \eta(a) a^n |a|^{-t} d^\times a,$$
and similar expressions hold for the derivatives of $C_{2,t}$ by replacing the $y$-axis by the $x$-axis.
(This Tate integral converges for $\Re(t)\gg 0$, and should be interpreted in terms of its meromorphic continuation otherwise.) From this we can deduce that, for an analytic section $t\mapsto \Phi_t\in \mathcal S(F^2)$, its image in $V^t$ has the stated form (i.e., any given derivative of $C_{1,t}$ and $C_{2,t}$ is analytic away from \eqref{exceptions2}, \eqref{exceptions3}, and at $t=0$ they have at most simple poles with opposite residues). 

Now we claim that any pair: $t\mapsto (C_{1,t}, C_{2,t})$ of the stated form, or any section into $V^t$ of the stated form, lifts to an analytic section of $\mathcal S(F^2)$, more precisely, that there are lifts:
\begin{eqnarray} \mbox{section of the stated form into }V^t &\rightsquigarrow& \mbox{ pair }(C_{1,t}, C_{2,t})\mbox{ of the stated form} \nonumber\\ &\rightsquigarrow &\mbox{analytic section into }\mathcal S(F^2). \label{lifts}
\end{eqnarray}
 Let $h_t$ denote a section of the stated form, either into pairs $(C_{1,t}, C_{2,t})$, or into $V^t$.

 Vice versa, given an analytic section $t\mapsto \Phi_t\in \mathcal S(F^2)$, we have already explained that its image in $V^t$ is of the stated form and admits a section of the stated form into pairs $(C_{1,t}, C_{2,t})$, which reduces the problem to the case when the image of $\Phi_t$ in $V^t$ is identically zero, i.e., its image $f_t\in \mathcal A^t$ lies in the subspace $\mathcal S(F^\times)$. We need to prove that $t\mapsto f^t$ is analytic into $\mathcal S(F^\times)$ or, what is equivalent, into $\mathcal S(F)$. But we have already seen that its image $\mathcal F(f_t)\in \mathcal B^t$ is analytic. Since $f_t\in \mathcal S(F^\times) $, $\mathcal F(f_t)$ will lie in $\mathcal S(F)\subset \mathcal B^t$, so the inverse Fourier transform of that is also analytic into $\mathcal S(F)$. This concludes the proof of the lemma.
\end{proof}

 Although Fourier transform is not defined at the values of $t$ of \eqref{exceptions}, it is easy to see that the transform \eqref{toperator}, or equivalently $\mathcal G_t$, \emph{is} defined and preserves analytic sections:
 
\begin{lemma}\label{Gtdefined}
The transform \eqref{toperator} defines an automorphism of the Fr\'echet bundle $\mathcal A^t$.
\end{lemma}

\begin{proof}
 Since the inverse of Fourier transform on $F^2$ is Fourier transform (with the inverse character), it follows that the inverse of the operator \eqref{toperator}: $\mathcal A^t\to \mathcal A^{-t}$ is the operator $|\bullet|^t \mathcal G_{-t}'$, where $\mathcal G_{-t}'$ is defined as $\mathcal G_{-t}$, but with inverse Fourier transform. Thus, for $t$ as in \eqref{exceptions}, one can define $\mathcal G_t$ as: $|\bullet|^t \mathcal G_{-t}'^{-1} |\bullet|^{-t}$.
\end{proof}

Now we come to the notion of \emph{polynomial seminorms} and \emph{polynomial sections}, resp.\ \emph{sections of rapid decay on vertical strips}. We have already seen what a polynomial family of seminorms on $\mathcal S(F)$ is (it is included in the first case): it is bounded on bounded vertical strips by finite sums of the form $|P(t)|\rho$, where $\rho$ is a fixed seminorm on $\mathcal S(F)$. Now fix a Schwartz partition of unity on $F$ as $u_1+u_2$, with $u_1\in \mathcal S(F)$ and $u_2\in \mathcal S(\PP^1(F)\smallsetminus\{0\})$, and use it to define a splitting of the map:
\begin{equation}\label{intoBt} \mathcal S(F)\oplus\mathcal S(F) \ni (F_1,F_2)\mapsto F_1(\bullet) + F_2(\frac{1}{\bullet}) |\bullet|^{-t-1}\in \mathcal B^t.
\end{equation}
We define a \emph{polynomial family of seminorms} on $\mathcal B^t$ to be a family of seminorms which, when pulled back to $ \mathcal S(F)\oplus\mathcal S(F)$ by \eqref{intoBt} are bounded by polynomial seminorms on bounded vertical strips. We say that a section $t\mapsto F_t\in \mathcal B^t$ is of \emph{polynomial growth}, resp.\ \emph{rapid decay}, if this is the case for any polynomial family of seminorms applied to it (always, implicitly, on bounded vertical strips). We can use inverse Fourier transform to translate these notions to $\mathcal A^t$ (away from the points \eqref{exceptions}) and then we have:

\begin{lemma}
 The sections of polynomial growth, resp.\ rapid decay, of $\mathcal A^t$ are those of the form $t\mapsto C_{1,t}(\bullet)+C_{2,t}(\bullet)\eta(\bullet)|\bullet|^t$, where $t\mapsto C_{1,t}, C_{2,t}$ are sections of the same type into $\mathcal S(F)$.
\end{lemma}

\begin{proof}

As for Lemma \ref{analyticsections} we will argue by lifting to $\mathcal S(F^2)$. We may ignore the values of $t$ in a neighborhood of bounded width of the real line. Our goal is to show that sections of polynomial growth (resp.\ rapid decay, but we will avoid repeating this in the rest of the argument) into $\mathcal B^t$ are precisely the images, under \eqref{StoB}, of sections of polynomial growth into $\mathcal S(F^2)$, a notion which is stable under (partial) Fourier transform; and, that the image in $\mathcal A^t$ of sections of polynomial growth into $\mathcal S(F^2)$ under \eqref{ttwisted} consists precisely of the sections described in the statement of the lemma. For what follows, we apply the operation \eqref{StoB} to a function denoted by $\Phi$, instead of $\hat\Phi^1$.

Clearly, a section $t\mapsto F_t \in \mathcal B^t$ of polynomial growth can be lifted via \eqref{StoB} to a section of polynomial growth into $\mathcal S(F^2)$ (in fact, the lift can be taken to be in $\mathcal S(F^2\smallsetminus\{(0,0)\})$, where $F^\times$ acts freely). 

Similarly, a section $t\mapsto  C_{1,t}(\bullet)+C_{2,t}(\bullet)\eta(\bullet)|\bullet|^t$ of polynomial growth, as in the statement of the lemma, can be lifted via \eqref{ttwisted} to a section $\Phi$ of polynomial growth into $\mathcal S(F^2)$ (in fact, into $\mathcal S(F^2\smallsetminus\{(0,0)\})$ or even $\mathcal S(F\times F^\times)$, where $F^\times$ acts freely).

There remains to show that sections of polynomial growth: $t\mapsto \Phi_t\in \mathcal S(F^2)$ descend to sections of polynomial growth, in the above sense, into $\mathcal A^t$ and $\mathcal B^t$. 

This is clearly true if $\Phi_t$ lies in $\mathcal S(F^2\smallsetminus\{(0,0)\})$, where $F^\times$ acts freely. This will turn out to be enough, since for notions of growth on bounded vertical strips we can ignore some values of $t$ on the real line. The conceptual reason is that away from integral real values of $t$, there are no non-trivial $(F^\times,\eta(\bullet)|\bullet|^t)$-equivariant distributions on the stalk at zero (through either of the two $F^\times$ actions appearing in \eqref{ttwisted}, \eqref{StoB}).

To turn this into a rigorous argument, denote by $*^t$ the action of $F^\times$    on $\mathcal S(F^2)$ twisted by the appropriate character depending on the case under consideration, so that \eqref{ttwisted}, resp.\ \eqref{StoB}, is zero on elements of the form $\Phi- a*^t \Phi$ (where $a\in F^\times$). The stalk $W$ of $\mathcal S(F^2)$ at $(0,0)$ is stable under the $F^\times$-action, and it is easy to see that if we fix $a\in F^\times$ with $|a|\ne 1$, for $t$ away from integral real values, the map
$$ (I-a*^t): W\to W$$
is \emph{bijective}. (Recall that an element $\bar\Phi$ of $W$ is determined by the values of all derivatives of a representative $\Phi$ at the origin.) Moreover, the inverse of this map preserves sections of moderate growth: if $t\mapsto \bar\Phi_t$ is a section of polynomial growth into $W$, then the section $t\mapsto \bar\Phi'_t \in W$ such that $\bar\Phi_t = \bar\Phi'_t - a*^t \bar\Phi'_t$ is also of polynomial growth.

Given a section $\Phi_t$ of polynomial growth into $\mathcal S(F^2)$, with image $\bar\Phi_t$ in $W$, we let $\bar\Phi_t'$ be as described and lift it to a section of polynomial growth $t\mapsto \Phi'_t\in \mathcal S(F^2)$. Then the section
$$ t\mapsto \Phi_t - \Phi'_t + a*^t \Phi'_t \in \mathcal S(F^2\smallsetminus \{0\})$$
is of polynomial growth and has the same image as $\Phi_t$ under \eqref{ttwisted}, resp.\ \eqref{StoB}. This proves the claim.
\end{proof}

Finally, we notice the following obvious fact:

\begin{lemma}\label{insection} 
For any given  $f_{t_0} \in\mathcal A^{t_0}$ one can find an analytic section $t\mapsto f^t\in\mathcal A^t$, of rapid decay in vertical strips, defined for all $t\in \CC/2 \pi i\Z \log q$, whose specialization to $t_0$ is the given element.
\end{lemma}

\subsection{The global case} 

In the global case, we will give ourselves spaces $\mathcal A^t$ which are restricted tensor products, with respect to a specified vector $f_{t,v}^0$ defined at almost every place, of local spaces $\mathcal A^t_v$ as above. We say ``tensor products'' in the completed sense here (for each finite set of places), i.e.:
$$\mathcal A^t = \lim_{\to} \hat\otimes_{v\in S} \mathcal A^t_v,$$
the limit taken as the finite set $S$ tends to include all places, and the identification for $S_1\subset S_2$ being by multiplication with the basic vector. With the individual tensor products being Fr\'echet spaces, the limit becomes and LF-space.

We will make the following assumption:
\begin{equation}\label{axiom-analytic}
 \mbox{\emph{The section $t\mapsto f^0_{t,v}$ is analytic as $t$ varies in $\CC/\frac{2\pi i}{\log q_v} \Z$.}}
\end{equation}

An \emph{analytic section} $t\mapsto f_t\in\mathcal A^t$, $t\in \CC$, will then be a section which is of the form
$$ f_t = \bigotimes_{v\notin S} f^0_{t,v} \otimes f_{S,t},$$
with $t\mapsto f_{S,t}$ an analytic section into $\hat\otimes_{v\in S} \mathcal A^t_v$. By taking tensor products for the embeddings \eqref{intoFb}, the notion of ``analytic section'' makes sense for the (completed) tensor product, and so do the notions of polynomial seminorms and sections of polynomial growth/rapid decay on vertical strips. These notions do not depend on the choice of $S$ used to represent $f_t$, because of axiom \eqref{axiom-analytic}.

\section{F-representations}\label{sec:Frepresentations}

\subsection{Definitions}

Let $G$ denote here the points of a reductive group over an Archimedean field, and fix a faithful algebraic representation: $G\to \GL_N$ to  get a natural algebraic height function $\Vert \bullet \Vert $ on $G$ by pulling back the maximum of the operator norms of $g$ and $g^{-1}$ with respect to the norm $(r_1,\dots,r_N) \mapsto \max_i |r_i|$. 

We recall the notion of F- and SF-representations of $G$, in the language of \cite{BeKr}. The same definitions can be given for $p$-adic groups, if one uses ``almost smooth'' vectors as we have done throughout this paper (calling them just ``smooth''). The results that we quote from \cite{BeKr} extend to the $p$-adic case, however we will omit this discussion, hoping to include it in future work. 

An \emph{F-representation} of $G$ is a continuous representation on a Fr\'echet space such that there is a sequence of seminorms defining the topology with the property that the action of any $g\in G$ is bounded with respect to \emph{each one of them}. In particular, Banach representations are F-representations, and the general F-representation is an inverse limit of Banach representations. We recall that an F-representation is equivalent to a continuous representation of \emph{moderate growth} of $G$ on a Fr\'echet space, that is, a representation with the property:
\begin{quote}\hspace{-1cm}(*)\hspace{0.4cm}
{For every (continuous) seminorm $p$ there exists a seminorm $q$ and a positive number $N$ such that $p(gv)\le \Vert g\Vert^N q(v)$ for every $g\in G$ and $v\in V$.}
\end{quote}
The topological convolution algebra $\mathcal R(G)$ of \emph{rapidly decaying $L^1$-measures on $G$} (i.e., measures $\mu$ such that $\Vert g\Vert^N\mu$ is an $L^1$-measure for every $N$) acts continuously on any F-representation. 

When $V$ is a continuous representation of $G$ on a Fr\'echet space, we endow the subspace $V^\infty$ of smooth vectors with the topology induced by all the seminorms of $V$ applied to all derivatives of a vector. We say that $V$ is \emph{smooth} if $V\simeq V^\infty$ \emph{as topological vector spaces}. 
If $V$ is an F-representation, then $V^\infty$ is a smooth F-representation, or \emph{SF-representation} in the language of \cite{BeKr}. Notice that when $p$ is a seminorm, $V$ is a smooth F-representation and $v\in V$ then there is an $N$ such that $\sup_{g\in G} \frac{p(gDv)}{\Vert g\Vert^N}<0$ for every $D\in \mathfrak U(\mathfrak g)$ (in fact, $N$ can be chosen independently of $v$), which is why these vectors are sometimes said to be of ``uniformly moderate growth'', but we will just be saying ``moderate growth''. (On the other hand, notice that if $N$ works for the seminorm $p$, then one needs a larger $N$, in general, for the seminorm $v\mapsto p(Dv)$, because of a factor accounting for the adjoint action of $G$ on $D$.)

Any F-representation is a continuous $\mathcal S(G)$-module (we identify Schwartz functions with Schwartz measures here by choosing a Haar measure, in order not to introduce extra notation). Moreover, a smooth F-representation $V$ is a \emph{non-degenerate} $\mathcal S(G)$-module, i.e., $\mathcal S(G)V=V$. In fact, it is already non-degenerate under the action of compactly supported elements of $\mathcal S(G)$; this is a consequence of the theorem of Dixmier and Malliavin.

Something stronger is true: the category of smooth F-representations is equivalent to the category of non-degenerate continuous $\mathcal S(G)$-module \cite[Proposition 2.20]{BeKr}. As the proof of Proposition 2.20 shows, non-degeneracy can be understood in a stronger sense: the topology on $V$ is the quotient topology for the action map:
\begin{equation}
 \mathcal S(G)\hat\otimes V\to V,
\end{equation}
where $\hat\otimes$ denotes the completed tensor product (projective, say, but since $\mathcal S(G)$ is nuclear it coincides with the injective one).

\subsection{Morphisms}
Let $V$ be a continuous representation on a Fr\'echet space and $N\in \RR_+$. Assume that a single semi-norm $p$ generates a complete set of seminorms under the action of $G$, i.e., the seminorms $p\circ g$, where $g$ ranges over $G$, form a complete set. Such are, for example, ``locally defined'' spaces of functions on a homogeneous space $X$ of $G$, e.g., $L^2_{\rm{loc}}(X)$, $C(X)$ etc. Then we define: 
\begin{equation}
 V_N:= \{v\in V| \sup_{g\in G} \frac{p(gv)}{\Vert g\Vert^N} <\infty\}.
\end{equation}

This subspace depends on the chosen height $\Vert \bullet\Vert$ on $G$ but not on the choice of $p$, and for any other compatible height function $\Vert \bullet\Vert'$ with corresponding spaces $V'_N$ there are positive constants $c,C$ such that $V_{cN}\subset V_N'\subset V_{CN}$ for every $N$.

\begin{lemma}
The subspace $V_N$, endowed with the norm: 
$$p_N(v):= \sup_{g\in G} \frac{p(gv)}{\Vert g\Vert^N},$$ 
is complete.
\end{lemma}

\begin{proof}
 Suppose that $(v_n)_n$ is a Cauchy sequence with respect to this norm, and let $v$ be its limit in $V$. For any $\epsilon>0$ and every $g$ we can choose an arbitrarily large $n_g$ such that: 
$$\frac{p(gv_{n_g}-gv)}{\Vert g\Vert^N }\le \epsilon.$$
 In particular, we can assume that all $n_g\ge m$, where $m$ is such that for all $n\ge m$ we have:
$$ \sup_{g\in G} \frac{p(gv_n-gv_m)}{\Vert g\Vert^N}<\epsilon.$$
Then:
$$ \sup_{g\in G} \frac{p(gv)}{\Vert g\Vert^N} \le \sup_{g\in G} \frac{p(gv_{n_g}-gv)}{\Vert g\Vert^N} + \sup_{g\in G} \frac{p(gv_{n_g})}{\Vert g\Vert^N} \le$$
$$\le \epsilon + \sup_{g\in G} \frac{p(gv_{n_g}-gv_m)}{\Vert g\Vert^N} + \sup_{g\in G} \frac{p(gv_m)}{\Vert g\Vert^N} \le 2\epsilon + p_N(v_m) <\infty,$$
hence $v\in V_N$.
\end{proof}

Hence, $V_N$ is a Banach representation, in particular an F-representation. We let: 
\begin{equation} V_{mg} = \bigcup_{N\in \mathbb N} V_N,\end{equation}
where ``mg'' stands for ``moderate growth''.

The spaces $V_N^{\infty}$ are SF-representations, and we set: 
\begin{equation}
 V_{mg}^{\infty} = \bigcup_N V_N^{\infty}.
\end{equation}
Clearly, as a subspace of $V$ it does not depend on choices.

Since $V_N \hookrightarrow V_{N+1}$ is not closed, there is no good topology on $V_{mg}$ (for instance, the finest topology making the inclusions of all $V_N$'s continuous is in general non-Hausdorff). By abuse of language, we will say that a map into $V_{mg}$ (resp.\ $V_{mg}^{\infty}$) is continuous if it factors as a continuous map through some $V_N$ (resp.\ $V_N^{\infty}$). For example, in this language the following lemma states that a certain map factors ``continuously'' through $V_{mg}$, resp.\ $V_{mg}^\infty$.

\begin{lemma}\label{lemmafunctors}
Let $V$ be as above, and let $W$ be an F-representation (resp.\ an SF-representation). Any morphism $T:W\to V$ factors (continuously) as:
$$
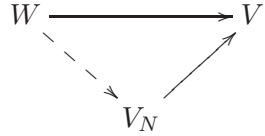
\xymatrix{
W \ar[rr] \ar@{-->}[dr]&& V\\
&V_N \ar[ur],
}
$$
resp.:
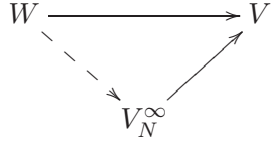
$$\xymatrix{
W \ar[rr] \ar@{-->}[dr]&& V\\
&V^{\infty}_N \ar[ur],
}$$
for some $N$.
\end{lemma}

\begin{proof}
Start by writing $W$ as an inverse limit of Banach representations of $G$, and letting $\nu_i$ be the corresponding seminorms. 

By continuity, there is an $i$ and a positive constant $C$ such that $\rho(Tw) \le C \nu_i(w)$ for every $w\in W$. Now recall that Banach representations have the moderate growth property: there exist $N$ and $D>0$ such that $\nu_i(gw)\le D \Vert g\Vert^N \nu_i(w)$ for any $w\in W$. Thus:
$$ \rho(g(Tw)) = \rho(T(gw)) \le C \nu_i(gw) \le CD \nu_i(w) \Vert g\Vert^N,$$
which shows that $T$ factors through a continuous map: $W\to V_N$. 

The statement on SF-representations follows by applying the functor ``smooth vectors'' to the morphism: $W\to V_N$. 
\end{proof}

\bibliographystyle{alphaurl}
\bibliography{biblio}

\begin{thebibliography}{GGP12}

\bibitem[Alt15]{Altug1}
Salim~Ali Altu\u{g}.
\newblock Beyond endoscopy via the trace formula: 1. {P}oisson summation and
  isolation of special representations.
\newblock {\em Compos. Math.}, 151(10):1791--1820, 2015.
\newblock \href {http://dx.doi.org/10.1112/S0010437X15007320}
  {\path{doi:10.1112/S0010437X15007320}}.

\bibitem[AP05]{AP}
Anne-Marie Aubert and Roger Plymen.
\newblock Plancherel measure for {${\rm GL}(n,F)$} and {${\rm GL}(m,D)$}:
  explicit formulas and {B}ernstein decomposition.
\newblock {\em J. Number Theory}, 112(1):26--66, 2005.
\newblock \href {http://dx.doi.org/10.1016/j.jnt.2005.01.005}
  {\path{doi:10.1016/j.jnt.2005.01.005}}.

\bibitem[Ber88]{BePl}
Joseph~N. Bernstein.
\newblock On the support of {P}lancherel measure.
\newblock {\em J. Geom. Phys.}, 5(4):663--710 (1989), 1988.
\newblock \href {http://dx.doi.org/10.1016/0393-0440(88)90024-1}
  {\path{doi:10.1016/0393-0440(88)90024-1}}.

\bibitem[BK14]{BeKr}
Joseph Bernstein and Bernhard Kr\"otz.
\newblock Smooth {F}r\'echet globalizations of {H}arish-{C}handra modules.
\newblock {\em Israel J. Math.}, 199(1):45--111, 2014.
\newblock \href {http://dx.doi.org/10.1007/s11856-013-0056-1}
  {\path{doi:10.1007/s11856-013-0056-1}}.

\bibitem[FLN10]{FLN}
Edward Frenkel, Robert Langlands, and B{\'a}o~Ch{\^a}u Ng{\^o}.
\newblock Formule des traces et fonctorialit\'e: le d\'ebut d'un programme.
\newblock {\em Ann. Sci. Math. Qu\'ebec}, 34(2):199--243, 2010.

\bibitem[GGP12]{GGP}
Wee~Teck Gan, Benedict~H. Gross, and Dipendra Prasad.
\newblock Symplectic local root numbers, central critical {$L$} values, and
  restriction problems in the representation theory of classical groups.
\newblock {\em Ast\'erisque}, (346):1--109, 2012.
\newblock Sur les conjectures de Gross et Prasad. I.

\bibitem[GN10]{GN}
Dennis Gaitsgory and David Nadler.
\newblock Spherical varieties and {L}anglands duality.
\newblock {\em Mosc. Math. J.}, 10(1):65--137, 271, 2010.

\bibitem[Her11]{Herman}
P.~Edward Herman.
\newblock Beyond endoscopy for the {R}ankin-{S}elberg {$L$}-function.
\newblock {\em J. Number Theory}, 131(9):1691--1722, 2011.
\newblock \href {http://dx.doi.org/10.1016/j.jnt.2011.01.019}
  {\path{doi:10.1016/j.jnt.2011.01.019}}.

\bibitem[Her12]{Herman-FE}
P.~Edward Herman.
\newblock The functional equation and beyond endoscopy.
\newblock {\em Pacific J. Math.}, 260(2):497--513, 2012.
\newblock \href {http://dx.doi.org/10.2140/pjm.2012.260.497}
  {\path{doi:10.2140/pjm.2012.260.497}}.

\bibitem[HII08]{HII}
Kaoru Hiraga, Atsushi Ichino, and Tamotsu Ikeda.
\newblock Formal degrees and adjoint {$\gamma$}-factors.
\newblock {\em J. Amer. Math. Soc.}, 21(1):283--304, 2008.
\newblock \href {http://dx.doi.org/10.1090/S0894-0347-07-00567-X}
  {\path{doi:10.1090/S0894-0347-07-00567-X}}.

\bibitem[II10]{II}
Atsushi Ichino and Tamutsu Ikeda.
\newblock On the periods of automorphic forms on special orthogonal groups and
  the {G}ross-{P}rasad conjecture.
\newblock {\em Geom. Funct. Anal.}, 19(5):1378--1425, 2010.
\newblock \href {http://dx.doi.org/10.1007/s00039-009-0040-4}
  {\path{doi:10.1007/s00039-009-0040-4}}.

\bibitem[Jac86]{JW1}
Herv{\'e} Jacquet.
\newblock Sur un r\'esultat de {W}aldspurger.
\newblock {\em Ann. Sci. \'Ecole Norm. Sup. (4)}, 19(2):185--229, 1986.
\newblock URL: \url{http://www.numdam.org/item?id=ASENS_1986_4_19_2_185_0}.

\bibitem[Jac87]{JW2}
Herv{\'e} Jacquet.
\newblock Sur un r\'esultat de {W}aldspurger. {II}.
\newblock {\em Compositio Math.}, 63(3):315--389, 1987.
\newblock URL: \url{http://www.numdam.org/item?id=CM_1987__63_3_315_0}.

\bibitem[JL70]{JL}
H.~Jacquet and R.~P. Langlands.
\newblock {\em Automorphic forms on {${\rm GL}(2)$}}.
\newblock Lecture Notes in Mathematics, Vol. 114. Springer-Verlag, Berlin,
  1970.

\bibitem[Kot84]{Kottwitz}
Robert~E. Kottwitz.
\newblock Stable trace formula: cuspidal tempered terms.
\newblock {\em Duke Math. J.}, 51(3):611--650, 1984.
\newblock \href {http://dx.doi.org/10.1215/S0012-7094-84-05129-9}
  {\path{doi:10.1215/S0012-7094-84-05129-9}}.

\bibitem[Lan04]{Langlands-BE}
Robert~P. Langlands.
\newblock Beyond endoscopy.
\newblock In {\em Contributions to automorphic forms, geometry, and number
  theory}, pages 611--697. Johns Hopkins Univ. Press, Baltimore, MD, 2004.

\bibitem[LM15]{LM-Whittaker}
Erez Lapid and Zhengyu Mao.
\newblock A conjecture on {W}hittaker-{F}ourier coefficients of cusp forms.
\newblock {\em J. Number Theory}, 146:448--505, 2015.
\newblock \href {http://dx.doi.org/10.1016/j.jnt.2013.10.003}
  {\path{doi:10.1016/j.jnt.2013.10.003}}.

\bibitem[MW95]{MW}
C.~M{\oe}glin and J.-L. Waldspurger.
\newblock {\em Spectral decomposition and {E}isenstein series}, volume 113 of
  {\em Cambridge Tracts in Mathematics}.
\newblock Cambridge University Press, Cambridge, 1995.
\newblock Une paraphrase de l'{\'E}criture [A paraphrase of Scripture].
\newblock \href {http://dx.doi.org/10.1017/CBO9780511470905}
  {\path{doi:10.1017/CBO9780511470905}}.

\bibitem[Ono63]{Ono-Tamagawa}
Takashi Ono.
\newblock On the {T}amagawa number of algebraic tori.
\newblock {\em Ann. of Math. (2)}, 78:47--73, 1963.

\bibitem[Pra07]{Prasad-Saito}
Dipendra Prasad.
\newblock Relating invariant linear form and local epsilon factors via global
  methods.
\newblock {\em Duke Math. J.}, 138(2):233--261, 2007.
\newblock With an appendix by Hiroshi Saito.
\newblock \href {http://dx.doi.org/10.1215/S0012-7094-07-13823-7}
  {\path{doi:10.1215/S0012-7094-07-13823-7}}.

\bibitem[Rud90]{Rudnick}
Ze\'ev Rudnick.
\newblock {\em {P}oincar{\'e} series}.
\newblock PhD thesis, Yale University, 1990.

\bibitem[Sai93]{Saito}
Hiroshi Saito.
\newblock On {T}unnell's formula for characters of {${\rm GL}(2)$}.
\newblock {\em Compositio Math.}, 85(1):99--108, 1993.
\newblock URL: \url{http://www.numdam.org/item?id=CM_1993__85_1_99_0}.

\bibitem[Sak13a]{SaBE1}
Yiannis Sakellaridis.
\newblock Beyond endoscopy for the relative trace formula {I}: {L}ocal theory.
\newblock In {\em Automorphic representations and {$L$}-functions}, volume~22
  of {\em Tata Inst. Fundam. Res. Stud. Math.}, pages 521--590. Tata Inst.
  Fund. Res., Mumbai, 2013.

\bibitem[Sak13b]{SaSph}
Yiannis Sakellaridis.
\newblock Spherical functions on spherical varieties.
\newblock {\em Amer. J. Math.}, 135(5):1291--1381, 2013.
\newblock \href {http://dx.doi.org/10.1353/ajm.2013.0046}
  {\path{doi:10.1353/ajm.2013.0046}}.

\bibitem[Sak16]{SaStacks}
Yiannis Sakellaridis.
\newblock The {S}chwartz space of a smooth semi-algebraic stack.
\newblock {\em Selecta Math. (N.S.)}, 22(4):2401--2490, 2016.
\newblock \href {http://dx.doi.org/10.1007/s00029-016-0285-3}
  {\path{doi:10.1007/s00029-016-0285-3}}.

\bibitem[SV]{SV}
Yiannis Sakellaridis and Akshay Venkatesh.
\newblock Periods and harmonic analysis on spherical varieties.
\newblock To appear in \emph{Ast\'erisque}, 291pp.
\newblock \href {http://arxiv.org/abs/1203.0039} {\path{arXiv:1203.0039}}.

\bibitem[Tun83]{Tunnell}
Jerrold~B. Tunnell.
\newblock Local {$\epsilon $}-factors and characters of {${\rm GL}(2)$}.
\newblock {\em Amer. J. Math.}, 105(6):1277--1307, 1983.
\newblock \href {http://dx.doi.org/10.2307/2374441}
  {\path{doi:10.2307/2374441}}.

\bibitem[Ven04]{V}
Akshay Venkatesh.
\newblock ``{B}eyond endoscopy'' and special forms on {GL}(2).
\newblock {\em J. Reine Angew. Math.}, 577:23--80, 2004.
\newblock \href {http://dx.doi.org/10.1515/crll.2004.2004.577.23}
  {\path{doi:10.1515/crll.2004.2004.577.23}}.

\bibitem[Vog93]{Vogan}
David~A. Vogan, Jr.
\newblock The local {L}anglands conjecture.
\newblock In {\em Representation theory of groups and algebras}, volume 145 of
  {\em Contemp. Math.}, pages 305--379. Amer. Math. Soc., Providence, RI, 1993.

\bibitem[Wal85]{Waldspurger-torus}
J.-L. Waldspurger.
\newblock Sur les valeurs de certaines fonctions {$L$} automorphes en leur
  centre de sym\'etrie.
\newblock {\em Compositio Math.}, 54(2):173--242, 1985.
\newblock URL: \url{http://www.numdam.org/item?id=CM_1985__54_2_173_0}.

\bibitem[Whi14]{White}
Paul-James White.
\newblock The base change {$L$}-function for modular forms and {B}eyond
  {E}ndoscopy.
\newblock {\em J. Number Theory}, 140:13--37, 2014.
\newblock \href {http://dx.doi.org/10.1016/j.jnt.2014.01.006}
  {\path{doi:10.1016/j.jnt.2014.01.006}}.

\bibitem[Zha14]{Wei1}
Wei Zhang.
\newblock Automorphic period and the central value of {R}ankin-{S}elberg
  {L}-function.
\newblock {\em J. Amer. Math. Soc.}, 27(2):541--612, 2014.
\newblock \href {http://dx.doi.org/10.1090/S0894-0347-2014-00784-0}
  {\path{doi:10.1090/S0894-0347-2014-00784-0}}.

\end{thebibliography}

\end{document}